\documentclass[a4paper,10pt]{article}

\usepackage[utf8]{inputenc}
\usepackage[width=15cm,height=22cm]{geometry}
\usepackage[dvipsnames]{xcolor}
\usepackage{amsmath,amsfonts,amssymb,amsthm}
\usepackage[hidelinks,pdfusetitle]{hyperref}
\usepackage[capitalise]{cleveref}
\crefname{equation}{}{}
\usepackage[usenames,dvipsnames]{pstricks}
\usepackage{bbm}
\usepackage{qtree}
\usepackage{stmaryrd}

\newtheorem{proposition}{Proposition}[section]
\newtheorem{theorem}[proposition]{Theorem}
\newtheorem{corollary}[proposition]{Corollary}
\newtheorem{lemma}[proposition]{Lemma}
\newtheorem{definition}[proposition]{Definition}

\newtheorem{remark}[proposition]{Remark}
\newtheorem{example}[proposition]{Example}
\numberwithin{equation}{section}

\newcommand*\samethanks[1][\value{footnote}]{\footnotemark[#1]}
\title{Growth of structure constants of free Lie algebras \\ relative to  Hall bases}
\author{Karine Beauchard\texorpdfstring{\thanks{Univ Rennes, CNRS, IRMAR - UMR 6625, F-35000 Rennes, France}}{}, J\'er\'emy Le Borgne\texorpdfstring{\samethanks}{},  Fr\'ed\'eric Marbach\texorpdfstring{\samethanks}{}}

\newcommand{\C}{\mathbb{C}}
\newcommand{\K}{\mathbb{K}}
\newcommand{\N}{\mathbb{N}}
\newcommand{\R}{\mathbb{R}}
\newcommand{\Z}{\mathbb{Z}}

\newcommand{\B}{\mathcal{B}}
\newcommand{\intset}[1]{\llbracket #1 \rrbracket}
\DeclareMathOperator{\Br}{Br}
\DeclareMathOperator{\Lyn}{Lyn}
\DeclareMathOperator{\Tr}{Tr}
\DeclareMathOperator{\ABr}{ABr}
\DeclareMathOperator{\br}{br}
\DeclareMathOperator{\ad}{ad}
\newcommand{\dad}{\underline{\ad}}
\DeclareMathOperator{\vect}{span}
\newcommand{\supp}{\operatorname{supp}}
\newcommand{\BrStarAlpha}{\ABr^\star}
\newcommand{\ev}{\textsc{e}}
\newcommand{\e}{\textsc{i}}
\newcommand{\step}[1]{\medskip \noindent \emph{#1}.}
\newcommand{\rf}[2]{T_{#1}(#2)}

\begin{document}

\maketitle

\begin{abstract}
    We derive \emph{a priori} bounds on the size of the structure constants of the free Lie algebra over a set of indeterminates, relative to its Hall bases. 
    We investigate their asymptotic growth, especially as a function of the length of the involved Lie brackets. 
    
    First, using the classical recursive decomposition algorithm, we obtain a rough upper bound valid for all Hall bases.
    We then introduce new notions (which we call \emph{alphabetic subsets} and \emph{relative foldings}) related to structural properties of the Lie brackets created by the algorithm, which allow us to prove a sharp upper bound for the general case. 
    We also prove that the length of the relative folding provides a strictly decreasing indexation of the recursive rewriting algorithm.
    Moreover, we derive lower bounds on the structure constants proving that they grow at least geometrically in all Hall bases.
    
    Second, for the celebrated historical length-compatible Hall bases and the Lyndon basis, we prove tighter sharp upper bounds, which turn out to be geometric in the length of the brackets.
    
    Third, we construct two new Hall bases, illustrating two opposite behaviors in the two-indeterminates case.
    One is designed so that its structure constants have the minimal growth possible, matching exactly the general lower bound, linked with the Fibonacci sequence.
    The other one is designed so that its structure constants grow super-geometrically.
    
    Eventually, we investigate asymmetric growth bounds which isolate the role of one particular indeterminate. 
    Despite the existence of super-geometric Hall bases, we prove that the asymmetric growth with respect to each fixed indeterminate is uniformly at most geometric in all Hall bases.
\end{abstract}

\newpage
\setcounter{tocdepth}{1}
\tableofcontents

\newpage

\newpage
\section{Introduction}


\subsection{Context and motivation}

We consider Hall bases of the free Lie algebra $\mathcal{L}(X)$ over a set $X$ with $|X| \geq 2$. 
This class of bases, understood in the generalized sense of Viennot \cite[Theorem~1.2]{MR516004} or Shirshov \cite[Definition~1]{zbMATH03234480}, includes many well-known and widely used bases of $\mathcal{L}(X)$ (see \cref{ss:def-hall} for more details).

Given a Hall set $\B \subset \Br(X)$, the free magma over $X$, we investigate the growth of the structure constants of $\mathcal{L}(X)$ relative to the Hall basis associated with $\B$.
More precisely, we estimate the $\ell^1$ norm of the Lie bracket $[a,b]$ of two elements $a$ and $b$ of the basis, denoted by $\| [a, b] \|_\B$, and we track its growth with respect to the size of the involved basis elements.
We investigate both \emph{symmetric estimates}, where the growth is measured with respect to the usual lengths $|a|$ and $|b|$ of the involved elements of the basis, and \emph{asymmetric estimates}, where the growth is tracked with a special focus on the number of occurrences of a singled-out indeterminate $X_0 \in X$.

To the best of our knowledge, this work is the first providing estimates on the structure constants of free Lie algebras.
Classical mathematical questions around structure constants involve their computation for particular algebras (see e.g.\ \cite{zbMATH05122976} for spherical Hecke algebras, \cite{zbMATH05853368} for symmetric Hopf algebras, \cite{zbMATH00062646} for a superoperator algebra,  \cite{zbMATH06436605} for Kac-Moody Lie algebras, \cite{zbMATH06458932, casselman2020simple, zbMATH03111938} for semisimple Lie algebras), the investigation of their sign \cite{zbMATH05304624, zbMATH05822267}, or the identification and reconstruction of the algebra from them \cite{zbMATH04093693, zbMATH03234481}.

Our initial motivation, both for symmetric and asymmetric estimates, stems from convergence issues for series of Lie brackets of analytic vector fields (see e.g.\ the open problem raised in \cite[Section 2.4.3]{2020arXiv201215653B} and the conditional result in \cite[Section 4.4.3]{2020arXiv201215653B}).
From an algebraic point of view, such estimates are linked with the intricacies of the Lie product in a Hall basis.
In particular, the methods we develop could probably be extended to estimate the computational complexity (with respect to time or space) of algorithms used by Lie algebraic packages to decompose Lie brackets of the free Lie algebra on a Hall basis (see \cref{sec:rewriting} for more details).

\subsection{Main results and plan of the paper}
\label{s:intro-results}

In this paragraph, we state rough pedagogical versions of our main results and sketch the plan of the paper.
We refer to \cref{ss:def-algebra} for definitions of the algebraic notions involved, to \cref{ss:def-hall} for definitions and properties relative to Hall sets and our choice of norm relative to a basis (see in particular \eqref{eq:norm-ab}), and to the sequel of the paper for enhanced statements of the results.

\bigskip

First, in \cref{sec:exp-n-n2}, using the classical recursive algorithm which allows to decompose a Lie bracket $[a,b]$ of two basis elements on the Hall basis (see e.g.\ \cite[Section~9]{zbMATH02156381}), we obtain the following very rough bound. 

\begin{theorem} \label{thm:enn2-easy}
    Let $\B \subset \Br(X)$ be a Hall set.
    For every $a < b \in \B$,
    \begin{equation}
        \| [a, b] \|_\B \leq 2^{
        \left(
        n^{\frac{n(n+1)}{2}}
        \right)
        }
        \quad \textrm{where} \quad n := |a| + |b|.
    \end{equation}
\end{theorem}

Then, in \cref{sec:n!}, we introduce the notion of \emph{alphabetic subsets} of $\B$. As a prime example, the set $X \subset \B$ is an alphabetic subset. This notion, along with an associated stability property, allows us to obtain the following upper bound on iterated Lie brackets of elements of such subsets, which is sharp in the sense of \cref{Cor:optn!}.

\begin{theorem} \label{thm:n!-easy}
    Let $\B \subset \Br(X)$ be a Hall set.
    Let $A$ be an alphabetic subset of $\B$.
    For every $n \geq 2$, and $t$ an iterated Lie bracket of $n$ elements of $A$,
    \begin{equation}
        \| t \|_\B \leq (n-1)!.
    \end{equation}
\end{theorem}

In \cref{Sec:refined_general_bound}, we introduce the notion of \emph{relative folding} of an element $b \in \B$ with respect to another element $a \in \B$ such that $a < b$.
This construction, along with the key remark that it yields an iterated Lie bracket over an alphabetic subset, allows us to describe the elements of $\B$ involved in the decomposition of $[a,b]$ and prove the following bound, which is sharp in the sense of \cref{cor:x-infini-sature}.
We also prove that the length of the relative folding provides a nice strictly decreasing indexation for the classical recursive rewriting algorithm (see \cref{ss:rec-theta}). 

\begin{theorem} \label{thm:en1-easy}
    Let $\B \subset \Br(X)$ be a Hall set.
    For every $a < b \in \B$,
    \begin{equation} \label{eq:en1-easy}
        \| [a, b] \|_\B \leq \lfloor e(|b|-1)! \rfloor.
    \end{equation}
\end{theorem}

Concerning lower bounds in the general case, we prove in \cref{sec:geom-lower} that the growth is always at least geometric in all Hall bases, with a common ratio depending on $|X|$.

\begin{theorem} \label{thm:lower-easy}
   Let $\B \subset \Br(X)$ be a Hall set.
   For every $n \geq 3$, there exist $a < b \in \B$ such that $|a| = 1$, $|b| = n - 1$ (so $|a| + |b| = n$) and
   \begin{equation} \label{eq:lower}
        \| [a, b] \|_\B \geq 
        \begin{cases}
            2^{n-2}
            & \text{when } |X| \geq 3, \\
            F_{n-2}
            & \text{when } |X| = 2, \\
        \end{cases}
   \end{equation}
   where $(F_\nu)_{\nu \in \N}$ denote the Fibonacci numbers.
\end{theorem}

We then turn to estimates specific to particular well-known Hall bases. 
In \cref{sec:length}, for the case of the historical Hall bases for which the order is compatible with the length of the brackets, we obtain the following geometric bound, which is sharp in the sense of \cref{p:length-sharp}.

\begin{theorem} \label{thm:length-easy}
    Let $\B \subset \Br(X)$ be a length-compatible Hall set.
    For every $a < b \in \B$,
    \begin{equation}
        \| [a, b] \|_\B \leq 2^{|b|-1}.
    \end{equation}
\end{theorem}

In \cref{sec:lyndon}, we obtain a geometric bound for the celebrated Lyndon basis, which is sharp in the sense of \cref{p:lyndon-sharp}.

\begin{theorem} \label{thm:lyndon-easy}
    Let $\B \subset \Br(X)$ be the Hall set of the Lyndon basis.
    For every $a < b \in \B$,
    \begin{equation}
        \| [a, b] \|_\B \leq 2^{|b|-1}.
    \end{equation}
\end{theorem}

When $|X| \geq 3$, the lower bound of \eqref{eq:lower} is optimal since it matches the upper bounds for the length-compatible and Lyndon bases.
When $|X| = 2$, we construct in \cref{s:fibo} an example of a Hall set of ``minimal worst product size'' in the sense that the following result holds.

\begin{theorem} \label{thm:fibo-easy}
    When $|X| = 2$, there exists a Hall set $\B \subset \Br(X)$ such that for every $a < b \in \B$ with $n := |a|+|b| \geq 3$,
    \begin{equation}
        \| [a, b] \|_\B \leq F_{n-2}.
    \end{equation}
\end{theorem}

In the opposite direction, we construct in \cref{sec:super-geom} a particular Hall set which shows that, contrary to the three previous cases, the growth can be super-geometric.
Our construction holds even when $|X| = 2$.

\begin{theorem} \label{thm:Mn-easy}
    There exists a Hall set $\B \subset \Br(X)$ such that, for every $M \geq 1$, for every $n \in \N^*$ large enough, there exist $a < b \in \B$ with $|a| = 2$ and $|b| = n$ such that
    \begin{equation}
        \| [a, b] \|_\B \geq M^n.
    \end{equation}
\end{theorem}

Eventually, we investigate in \cref{sec:asym} what we call \emph{asymmetric estimates}, i.e.\ estimates where we single out the role of a particular indeterminate $X_0 \in X$. 
For $b \in \Br(X)$, we denote by $n_0(b)$ the number of occurrences of $X_0$ in $b$, and $n(b) := |b|-n_0(b)$ the number of leaves of $b$ that are different from $X_0$.
We seek estimates on the structure constants where we attempt to have the lowest asymptotic growth with respect to $n_0$.
Our main result is the following.

\begin{theorem}
    \label{p:thm-asym}
    There exists a (universal) sequence $(C(n))_{n \in \N} \in \N^\N$, such that, if $\B \subset \Br(X)$ is a Hall set and $X_0 \in X$, then, for every $a < b \in \B$,
    \begin{equation}
        \| [a, b] \|_\B \leq C(n(b))^{n_0(b)} (n(b))!.
    \end{equation}
\end{theorem}

\subsection{Usual algebraic structures and notations}
\label{ss:def-algebra}

Let $\K=\R$ or $\C$. 
Implicitly, all vector spaces and algebras are constructed over the base field $\K$. 
Here and in the sequel, $X$ is a set, containing at least two elements, not necessarily finite.

\subsubsection{Free magma and free monoid}

For more detailed constructions of free magmas and monoids, we refer to \cite[Chapter~1, §7]{zbMATH00194100}. 

\begin{definition}[Free magma]
    We consider $\Br(X)$ the \emph{free magma} over $X$. 
    This set can be defined by induction: for $X_i \in X$, $X_i \in \Br(X)$ and if $t_1, t_2 \in \Br(X)$, then the ordered pair $(t_1, t_2)$ belongs to $\Br(X)$. 
    Intuitively, $\Br(X)$ can be seen as the set of formal brackets of elements of $X$. 
    It is also the set of rooted binary trees, with leaves labeled by elements of $X$. 
\end{definition}

\begin{definition}[Length, left and right factors]
    For $t \in \Br(X)$, $|t|$ denotes the length of $t$ i.e.\ the number of leaves of the tree.
    If $|t| > 1$, $t$ can be written in a unique way as $t = (t_1, t_2)$, with $t_1, t_2 \in \Br(X)$. 
    We use the notations $\lambda(t) = t_1$ and $\mu(t) = t_2$, which define maps $\lambda,\mu: \Br(X)\setminus X \to \Br(X)$. 
\end{definition}

\begin{definition}[Set of iterated left factors] \label{def:Lambda}
	For $t \in \Br(X)$, define 
	\begin{equation} \label{eq:LAMBDA}
		\Lambda(t) := \{ \lambda^k(t) ; \enskip k \in \N \text{ such that } |\lambda^k(t)| \geq 1 \} \subset \Br(X).
	\end{equation}
\end{definition}

\begin{example}\label{ex:tree}
    The element $t := (((X_0,X_1),((X_0,X_1),X_1)),((X_0,X_1),X_1))$ of $\Br(\{X_0,X_1\})$ can be visualized as the following tree:
    \begin{equation}
        \Tree [.  [. [. $X_0$ $X_1$ ] [. [. $X_0$ $X_1$ ] $X_1$ ] ] [. [. $X_0$ $X_1$ ] $X_1$ ] ]
    \end{equation}
    Here, $|t| = 8$, $\lambda(t) = ((X_0,X_1),((X_0,X_1),X_1))$
    and $\mu(t) = ((X_0,X_1),X_1)$.
    Moreover, $\Lambda(t) = \{ t, \lambda(t), (X_0, X_1), X_0 \}$.
\end{example}

\begin{definition}[Iterated bracketing] \label{def:ad-br}
    For $t_1, t_2 \in \Br(X)$, we introduce $\ad_{t_1}(t_2) := (t_1, t_2)$, the left bracketing by $t_1$ and $\dad_{t_1}(t_2) := (t_2, t_1)$, the right bracketing by $t_1$, which allow us to write compactly iterated brackets (e.g.\ $(((t_2, t_1), t_1), t_1) = \dad^3_{t_1}(t_2)$).
\end{definition}

\begin{definition}[Free monoid] \label{def:free.monoid}
    We denote by $X^*$ the \emph{free monoid} over $X$. 
    It is the set of finite sequences of elements of $X$, endowed with the concatenation operation.
    It can be thought of as the set of words over the alphabet whose letters are the elements of $X$.
\end{definition}


\subsubsection{Higher order brackets}

\begin{definition}[Higher order brackets]
    For $A \subset \Br(X)$, $\Br(A)$ denotes the free magma over~$A$.
    Thus $\Br(A)$ is a set of formal brackets of formal brackets, i.e.\ a subset of $\Br(\Br(X))$.
    By convention, the brackets in $\Br(A)$ are denoted by $\langle \cdot, \cdot \rangle$ to distinguish them from the brackets $(\cdot, \cdot)$ in $\Br(X)$. 
    We denote by $\e$ the canonical morphism of magmas from $\Br(\Br(X))$ to $\Br(X)$.
    We also denote by $\e$ its restriction to $\Br(A)$.
    For $t \in \Br(A)$, $|t|_A$ denotes its length in $\Br(A)$ and $|\e(t)|$ its length in $\Br(X)$.
\end{definition}

\begin{example}
    If $X = \{ X_1, X_2, X_3 \}$ and $A = \{ (X_1, X_2), X_3 \}$, then $t = \langle X_3, \langle X_3, (X_1, X_2) \rangle \rangle \in \Br(A)$, $|t|_A = 3$, $\e(t) = (X_3, (X_3, (X_1, X_2)))$ and $|\e(t)| = 4$.
\end{example}

\begin{definition}[Submagma generated by a subset]
    Let $A \subset \Br(X)$. 
    We denote by $\Br_A \subset \Br(X)$ the submagma of $\Br(X)$ generated by $A$, \emph{i.e} the image of the canonical morphism of magmas $\e : \Br(A) \to \Br(X)$.
\end{definition}

\begin{lemma}[Free subset of $\Br(X)$] \label{def:free-subset}
    Let $A \subset \Br(X)$. We say that the subset $A$ is \emph{free} when $A \cap (\Br_A,\Br_A)=\emptyset$. 
    Then the canonical surjection $\e : \Br(A) \rightarrow \Br_A$ is an isomorphism. 
\end{lemma}

\begin{proof}
    By contradiction, assume that there exists a couple $t_1 \neq t_2 \in \Br(A)$ such that $\e(t_1) = \e(t_2)$.
    If $|t_1|_A > 1$ and $|t_2|_A > 1$, since $t_1 \neq t_2$, we can assume by symmetry that $\lambda(t_1) \neq \lambda(t_2)$. 
    Since $\e(t_1) = \e(t_2)$, $\e(\lambda(t_1)) = \e(\lambda(t_2))$, so we have found a couple with the same property but shorter lengths in $\Br(A)$.
    Hence, we can assume that $|t_1|_A = 1$ or $|t_2|_A = 1$.
    One cannot have $|t_1|_A = |t_2|_A = 1$, because the canonical surjection is the identity for trees which are a single leaf.
    By symmetry, assume that $|t_1|_A = 1$ but $|t_2|_A > 1$, then $\e(t_1) \in A$ and $\e(t_2) \in (\Br_A, \Br_A)$, which contradicts the assumption that $A$ is free.
\end{proof}

\subsubsection{Free Lie algebra}

\begin{definition}[Free algebra] \label{def:free.algebra}
    We consider $\mathcal{A}(X)$ the \emph{free associative algebra} generated by $X$ over the field $\K$, i.e.\ the unital associative algebra of polynomials of the noncommutative indeterminates $X$ (see also \cite[Chapter~3, Section~2.7, Definition~2]{zbMATH00194100}).
    The algebra $\mathcal{A}(X)$ is the free vector space over $X^*$, and therefore is a graded algebra:
    \begin{equation}
        \mathcal{A}(X) = \underset{n \in \N}{\bigoplus} \mathcal{A}_n(X),
    \end{equation}
    where $\mathcal{A}_n(X)$ is the finite-dimensional $\K$-vector space spanned by monomials of degree $n$ over $X$ (i.e.\ elements of $X^*$ of length $n$). 
    In particular $\mathcal{A}_0(X) = \K$ and $\mathcal{A}_1(X) = \vect_{\K}(X)$.
\end{definition}

Let us now recall the definition of the main objects that we will be interested in in this paper: free Lie algebras. We refer to the books \cite{MR559927, zbMATH00417855,Serre1992} and the essay \cite{CasselmanFree} for thorough introductions to Lie algebras and free Lie algebras.

\begin{definition}[Free Lie algebra] \label{def:free.lie}
    The algebra $\mathcal A(X)$ is endowed with a natural structure of Lie algebra, the Lie bracket operation being defined by $[a,b] := ab - ba$. 
    This operation satisfies $[a, a] = 0$ and the Jacobi identity $[a,[b,c]]+[c,[a,b]]+[b,[c,a]] = 0$. 
    We consider $\mathcal{L}(X)$, the \emph{free Lie algebra} generated by $X$ over the field $\K$, which is defined as the Lie subalgebra generated by $X$ in $\mathcal A(X)$. 
    It can be seen as the smallest linear subspace of $\mathcal{A}(X)$ containing all elements of $X$ and stable by the Lie bracket (see also \cite[Theorem~0.4]{zbMATH00417855}). 
    The Lie algebra $\mathcal{L}(X)$ is a graded Lie algebra: 
    \begin{equation}
        \mathcal{L}(X) = \underset{n \in \N}{\bigoplus} \mathcal{L}_n(X), \qquad  [\mathcal{L}_m(X),\mathcal{L}_n(X)] \subset \mathcal{L}_{m+n}(X)
    \end{equation}
    where, for each $n \in \N$, we define $\mathcal{L}_n(X) := \mathcal{L}(X) \cap \mathcal{A}_n(X)$ 
\end{definition}

\begin{definition}[Natural evaluation] \label{rk:eval}
    By universal property of the free magma $\Br(X)$, there is an ``evaluation'' mapping $\ev$ : $\Br(X)\to\mathcal{L}(X)$. 
    It is such that $\ev(X_i) := X_i$ for $X_i \in X$, and $\ev(t) := [\ev(\lambda(t)), \ev(\mu(t))]$ when $t \in \Br(X) \setminus X$. 
\end{definition}

\begin{remark}
    The evaluation map $\ev$ is not injective: for example, $(X_0, X_0)$ and $(X_0, (X_1, X_1))$ are two different elements of $\Br(X)$, both evaluated to zero in $\mathcal{L}(X)$. 
\end{remark}

\begin{remark}[Implicit evaluations] \label{rk:conventions}
    For the sake of readability, we will sometimes, when no confusion is possible, omit the evaluation $\ev$ when writing equalities in $\mathcal{L}(X)$.
    In particular, since, by convention, we use parentheses $(\cdot, \cdot)$ as brackets in $\Br(X)$ and $[\cdot, \cdot]$ as brackets in $\mathcal{L}(X)$, we will simply write $[a,b]$ as a shorthand for $[\ev(a), \ev(b)] = \ev((a,b))$ when $a, b \in \Br(X)$.
\end{remark}

\begin{definition}[Iterated bracketing]
    As in $\Br(X)$ (see \cref{def:ad-br}), for $t_1, t_2 \in \mathcal{L}(X)$, we will use the notation $\ad_{t_1}(t_2) := [t_1, t_2]$ for the left bracketing by $t_1$ and $\dad_{t_1}(t_2) := [t_2, t_1]$ for the right bracketing by $t_1$, which allow us to write compactly iterated brackets.
    For each $t \in \mathcal{L}(X)$, the maps $\ad_t(\cdot)$ and $\dad_t(\cdot)$ are derivations on $\mathcal{L}(X)$, so we can use Leibniz' formula.
\end{definition}


\subsection{Hall sets and bases}
\label{ss:def-hall}

This paragraph introduces Hall sets, the associated Hall bases and some of their properties. Hall bases are of paramount importance, notably due to the nice rewriting algorithm (recalled in \cref{sec:rewriting}) and to their deep link with Lazard's elimination process, itself linked with the resolution of formal linear differential equations (as illustrated by \cite{zbMATH04006080}).
They are of course not the only bases of $\mathcal{L}(X)$ (see \cref{s:other-bases} for a short discussion in this direction).

\subsubsection{Definition}

There are different conventions in the literature for the definition of Hall sets.
Indeed, one may decide to swap left and right factors, or swap the order, or both.
We follow Viennot's convention of \cite{MR516004} (also used in control theory by Sussmann \cite{zbMATH04006080}), which is different from the original one by Marshall Hall \cite{zbMATH03059664} or from Reutenauer's in \cite{zbMATH02156381}.
    
\begin{definition}\label{Def2:Laz}
    A \emph{Hall set} is a subset $\B$ of $\Br(X)$ endowed with a total order $<$ such that
    \begin{itemize}
        \item $X \subset \B$,
        \item for all $b_1, b_2 \in \Br(X)$, $(b_1, b_2) \in \B$ iff $b_1, b_2 \in \B$, $b_1 < b_2$ and either $b_2 \in X$ or $\lambda(b_2) \leq b_1$, 
        \item for all $b_1, b_2 \in \B$ such that $(b_1,b_2) \in \B$ then $b_1 < (b_1,b_2)$.
    \end{itemize}
\end{definition}

\begin{remark}
    \label{rk:hall-construction} (See \cite[§5]{Serre1992})
    All Hall sets can be built by induction on the length. 
    One starts with the set $X$ as well as an order on it. 
    To find all Hall elements of length $n$ given those of smaller length, one adds first all $(a, b)$ with $a \in \B$, $|a|=n-1$, $b \in X$ and $a < b$. 
    Then for each bracket $b = (b_1, b_2) \in \B$ of length $2 \leq |b| < n$ one adds all the $(a, b)$ with $a \in \B$ with $|a|=n-|b|$ and $b_1 \leq a < b$. 
    Finally, one inserts the newly generated elements of length $n$ into an ordering, maintaining the condition that $a < (a, b)$. 
\end{remark}

\begin{remark}
    When $c = (a,(b_1,b_2)) \in \B$ then $a$ is ``sandwiched'' between $b_1$ and $c$: $b_1 \leq a < c$. 
    Moreover, iterating the third point of the definition yields $\min(X)=\min(\B)$.
\end{remark}

\cref{Def2:Laz} is due to Viennot, who also proved that the third item is necessary.
It was also known to Shirshov~\cite{zbMATH03234480} (albeit with an opposite convention).
It unifies multiple previous disjoint constructions and narrower definitions, such as the Lyndon basis (see \cref{sec:lyndon}) and the historical length-compatible Hall sets (see \cref{sec:length}).
We refer the interested reader to~\cite{zbMATH00035953} and~\cite{MR516004} for short accounts of the history of Hall sets.

The importance of this unified definition is linked with the following result.

\begin{theorem}[Viennot, \cite{MR516004}]
    \label{thm:viennot}
    Let $\B \subset \Br(X)$ be a Hall set. 
    Then $\ev(\B)$ is a basis of $\mathcal{L}(X)$.
\end{theorem}

\begin{example}
    For instance, with $X=\{X_0,X_1\}$, the elements of length at most $4$ of each Hall set $\B \subset \Br(X)$ with an order $<$ such that $X_0<X_1$ and $X < \Br(X) \setminus X$ are: $X_0$, $X_1$, $(X_0,X_1)$, $(X_0,(X_0,X_1))$, $(X_1,(X_0,X_1))$, $(X_0,(X_0,(X_0,X_1)))$, $(X_1,(X_0,(X_0,X_1)))$, $(X_1,(X_1,(X_0,X_1)))$. 
    But $b := (X_0,(X_1,(X_0,X_1)))$ does not lie in $\B$ because $\lambda(\mu(b))=X_1>X_0$. 
    In fact, the following equality holds in $\mathcal{L}(X)$:
    \begin{equation}[X_0,[X_1,[X_0,X_1]]]=[[X_0,X_1],[X_0,X_1]]+[X_1,[X_0,[X_0,X_1]]]=[X_1,[X_0,[X_0,X_1]]].
    \end{equation}
    This illustrates how \Cref{Def2:Laz} prevents elements from $\Br(X)$, whose evaluations in $\mathcal{L}(X)$ are linked by Jacobi relations, to appear simultaneously in $\B$. 
\end{example}

\subsubsection{Extension of a Hall order}

Constructing a Hall set relies on the construction of an order.
In the sequel, we will often need to construct Hall sets from appropriate totally ordered subsets of $\Br(X)$.
We therefore introduce the following notions and extension result.

\begin{definition}[$\lambda$-stability]
    We say that $A \subset \Br(X)$ is \emph{$\lambda$-stable} when $\lambda(A\setminus X) \subset A$.
\end{definition}

\begin{definition}[Hall order] \label{def:hall-order}
    Let $A$ be a $\lambda$-stable subset of $\Br(X)$. 
    We say that an order $<_A$ on $A$ is a \emph{Hall order} when $<_A$ is a total order on $A$ such that, for every $t \in A \setminus X$, $\lambda(t) <_A t$.
\end{definition}

\begin{proposition} \label{Prop:Prolonger_ordre}
Let $A$ be a $\lambda$-stable non empty subset of $\Br(X)$ and $<_A$ a Hall order on $A$. There exists a Hall order $<$ on $\Br(X)$ that extends $<_A$, i.e.\ for every $t_1, t_2 \in A$, $t_1<t_2$ iff $t_1 <_A t_2$.
\end{proposition}

\begin{proof}
    We consider the set $\mathcal{A}$ of the Hall-ordered $\lambda$-stable subsets $(B,<_B)$ with $B \subset \Br(X)$ such that, $A \subset B$ and for every $t_1,t_2 \in A$, $t_1 <_B t_2$ iff $t_1<_A t_2$.
    The set $\mathcal{A}$ is equipped with the (partial) order $\ll$ defined by $B \ll B'$ if $B \subset B'$ and for every $t_1,t_2 \in B$, $t_1 <_{B} t_2$ iff $t_1 <_{B'} t_2$. 

\step{Step 1: We prove that $\ll$ is an inductive order on $\mathcal{A}$} 
Let $(B_i)_{i\in I}$ be a family of  $\mathcal{A}$ totally ordered for $\ll$.
Let $\widetilde{B}:=\cup_{i\in I} B_i$. For $t_1, t_2 \in \widetilde{B}$, there exists $i \in I$ such that both $t_1, t_2 \in B_i$, then we say that $t_1 <_{\widetilde{B}} t_2$ if $t_1 <_{B_i} t_2$.
Then $(\widetilde{B},<_{\widetilde{B}})$ belongs to $\mathcal{A}$ and $B_i\ll\widetilde{B}$ for every $i \in I$.

\medskip

Thus, by Zorn's lemma, one may consider a maximal element $(B,<)$ in $\mathcal{A}$. 

\step{Step 2: We prove that $B=\Br(X)$} 
By contradiction, assume the existence of $t \in \Br(X) \setminus B$.
Let $p_0:=\max\{ k \in \N ; |\lambda^k(t)| \geq 1 \}$ and $p:=\max\{ k \in \intset{0, p_0} ; \lambda^k(t) \notin B \}$. Then $\lambda^k(t) \notin B$ for $k=0,\cdots,p$ and $\lambda^k(t) \in B$ for $k=p+1,\cdots,p_0$, because $B$ is $\lambda$-stable. 
Let $\overline{B}:=B \cup\{ \lambda^k(t);k=0,\cdots, p\}$ be equipped with the order $<$ that extends the order $<_B$ on $B$ and satisfies
$\tau <\lambda^p(t)<\dots < \lambda(t) < t$ for every $\tau \in B$.
Then $(\overline{B},<) \in \mathcal{A}$ and $B \ll \overline{B}$, which is a contradiction.
\end{proof}

\begin{remark}
    In particular, for every Hall set $(\B,<)$, since $\B$ is $\lambda$-stable, the order $<$, defined on $\B$, can be extended into a Hall order on $\Br(X)$. 
    This means that, although all that is required to construct a Hall set $\B$ is an order on $\B$ itself, all Hall sets can be seen as being constructed from a Hall order on the whole $\Br(X)$.
    This proves a verification left to the reader in \cite[page~85]{zbMATH00417855}.
\end{remark}

\begin{remark} \label{Rk:Zorn}
    Using Zorn's lemma is quite natural in this context, however, it is not necessary as one can also perform the construction by hand.
    Let $\prec$ be a given Hall order on $\Br(X)$, possibly unrelated with $<_A$. 
    Then, one defines an order $<$ on $\Br(X)$ by saying that, for $a, b \in \Br(X)$,
    \begin{itemize}
        \item if $a, b \in A$, $a < b$ iff $a <_A b$,
        \item if $a \in A$ and $b \in \Br(X) \setminus A$, $a < b$,
        \item if $a, b \in \Br(X) \setminus A$, $a < b$ iff $a \prec b$.
    \end{itemize}
    One easily checks that $<$ is indeed a Hall order on $\Br(X)$ using the assumption that $A$ is $\lambda$-stable.
    
    This construction requires the \emph{a priori} knowledge of a Hall order $\prec$ on $\Br(X)$.
    Such an order obviously exists.
    Given a total order $\prec_X$ on $X$, one can extend it to $\Br(X) \setminus X$ using the lexicographic order (see \cref{rk:lexico}) on the triple $(|t|, \lambda(t), \mu(t))$, which is a Hall order because $|\lambda(t)| < |t|$ for each $t \in \Br(X) \setminus X$.
    
    It remains to construct a total order $\prec_X$ on $X$. 
    If $X$ is countable, by definition there exists an injection $f : X \to \N$ and one can define that $x \prec_X x'$ iff $f(x) < f(x')$.
    If $X$ is not countable, the existence of a total order on $X$ is given by the ordering principle, which is strictly weaker than the axiom of choice (see \cite{zbMATH00769562}). 
\end{remark}

\begin{remark} \label{rk:lexico}
    In \cref{Rk:Zorn} and in the sequel, we often construct Hall orders by partitioning the set to be ordered and/or chaining a finite number of preorders.
    For example, given an order $\prec_X$ on $X$, we define the associated ``lexicographic order'' on the triple $(|t|, \lambda(t), \mu(t))$ by $t_1 < t_2$ iff
    \begin{itemize}
        \item either $|t_1| < |t_2|$,
        \item or $|t_1|=|t_2| = 1$ and $t_1 \prec_X t_2$,
        \item or $|t_1|=|t_2|>1$ and $\lambda(t_1) < \lambda(t_2)$,
        \item or $|t_1|=|t_2|>1$ and $\lambda(t_1) = \lambda(t_2)$ and $\mu(t_1) < \mu(t_2)$.
    \end{itemize}
    Such constructions obviously yield reflexive transitive relations, which are moreover total orders provided that enough preorders are combined.
    Although the definition uses recursively the order on $\lambda(t)$ and $\mu(t)$, this recursion makes sense since one can construct the order by induction on the length of the brackets.
    Therefore, in the sequel, when handling such orders of lexicographic nature, we mostly focus on checking that they satisfy the Hall condition $\lambda(t) < t$.
\end{remark}

We conclude this section by an elementary remark on the possibility to construct Hall sets from Hall orders defined only on parts of $\Br(X)$.

\begin{lemma} \label{Lem:G}
    Let $G$ be a $\lambda$-stable subset of $\Br(X)$ with $X \subset G$, endowed with a (total) Hall order such that, for every $b_1<b_2\in G$, $(b_1,b_2)\in G$.
    There exists a unique Hall set $\B \subset G$ associated with the order on $G$.
\end{lemma}

\begin{proof}
    By induction on $\ell \in \N^*$, we construct and determine the sets $\B_\ell \subset G$ of elements of $\mathcal{B}$ with length $\ell$.
    For $\ell=1$, $\mathcal{B}_1=X \subset G$.
    Let $\ell \geq 2$. We assume $\B_j$ constructed for every $j < \ell$. Let 
    $\mathcal{B}_{\ell}:=\cup_{1\leq j \leq \ell-1} \B_{\ell}^j$ where $\B_\ell^1 := \{ (b_1,b_2) ; b_1 \in \B_{\ell-1}, b_2 \in \B_1, b_1<b_2 \}$ and $\B_\ell^j:=\{ (b_1,b_2) ; b_1 \in \B_{\ell-j}, b_2 \in \B_j, \lambda(b_2) \leq b_1 < b_2 \}$ for $j=2,\dots,\ell-1$.
    By the last property of $G$, $\B_\ell \subset G$. 
\end{proof}

\subsubsection{Structure constants} \label{s:structure-constants}

Structure constants of an algebra relative to a basis are the expression of the product of pairs of elements of this basis as linear combinations over it. 
They are for example often used in physics to describe some specific algebras.
We recall this classical notion in the context of Hall bases of free Lie algebras.

\begin{definition}[Structure constants]
    Let $\B \subset \Br(X)$ be a Hall set.
    Using the convention of \cref{rk:conventions}, for any $a, b \in \B$, since $[a, b] \in \mathcal{L}(X)$, it can be written as a finite linear combination over $\B$, say
    \begin{equation} \label{eq:ab-sum-c}
        [a, b] = \sum_{c \in \B} \gamma^c_{a,b} \ev(c),
    \end{equation}
    where the coefficients $\gamma^c_{a,b} \in \K$ and only a finite number of them are non-zero. 
    The set of all coefficients $\gamma^c_{a,b}$ are called the structure constants of $\mathcal{L}(X)$ relative to the basis $\B$. 
\end{definition}

Our choice is to measure the size of the structure constants with the help of the following $\ell^1$ norm of the coordinates of decompositions in the basis.

\begin{definition}[Norm relative to a basis]
    Let $\B \subset \Br(X)$ be a Hall set.
    Any $a \in \mathcal{L}(X)$ can be written as a finite linear combination of basis elements, say
    \begin{equation} \label{eq:a-sum-c}
        a = \sum_{c \in \B} \alpha_a^c \ev(c),
    \end{equation}
    where the coefficients $\alpha_a^c \in \K$ and only a finite number of them are non-zero. We endow $\mathcal{L}(X)$ with the following norm
    \begin{equation}
        \| a \|_\B := \sum_{c \in \B} | \alpha_a^c |.
    \end{equation}
\end{definition}

In particular,
\begin{equation} \label{eq:norm-ab}
    \| [a,b] \|_\B = \sum_{c \in \B} |\gamma^c_{a,b}|,
\end{equation}
which justifies why the statements of \cref{s:intro-results} measure the growth of the structure constants.

\begin{definition}[Support]
    Let $\B \subset \Br(X)$ be a Hall set.
    For $a \in \mathcal{L}(X)$, we will denote by $\supp a$ the set of $c \in \B$ involved in~\eqref{eq:a-sum-c} such that $\alpha_a^c \neq 0$.
    In particular, for $a, b \in \B$, $\supp [a,b]$ denotes the set of $c \in \B$ involved in \eqref{eq:ab-sum-c} such that $\gamma_{a,b}^c \neq 0$.
\end{definition}

\subsubsection{Subsets of a Hall set}

In the sequel, we will use the following elementary results on  subsets of Hall sets.

\begin{lemma} \label{Lem:BRA}
Let $\B \subset \Br(X)$ be a Hall set, $A$ be a finite subset of $\B$.
Then, for every $t \in \B \cap \Br_A$, we have $t \geq \min(A)$.
\end{lemma}
\begin{proof}
The proof is by induction on $\ell(t) := \min\{ |\tau|_A ; \tau \in \Br(A), \e(\tau)=t \} \in \N^*$. If $t \in \B \cap \Br_{A}$ and $\ell(t)=1$, then $t \in A$ thus the conclusion holds. Now, let $t \in \B \cap \Br_A$ such that $\ell(t) \geq 2$. There exists $\tau_1,\tau_2 \in \Br(A)$ such that $t= \e(\langle \tau_1,\tau_2 \rangle)=(\e(\tau_1),\e(\tau_2))$ and $|\tau_1|_A + |\tau_2|_A = \ell(t)$.
Then $\e(\tau_1) \in \B \cap \Br_A$ and $\ell(\e(\tau_1)) \leq |\tau_1|_A < \ell(t)$ thus, by the induction assumption, $\e(\tau_1) \geq \min(A)$. Finally, $t>\e(\tau_1) \geq \min(A)$.
\end{proof}

\begin{lemma} \label{Lem:libre}
    Let $\B \subset \Br(X)$ be a Hall set, $A \subset \B$ free and $\Delta$ a finite subset of $\B \cap \Br_{A}$ such that, for every $t \in \Delta \setminus A$, $\lambda(t) < \min(\Delta)$. 
    Then $\Delta$ is free. 
\end{lemma}

\begin{proof} 
    If $t \in A$, then $t \notin (\Br_A,\Br_A)$ because $A$ is free, thus $t \notin (\Br_\Delta,\Br_\Delta)$ because $\Br_\Delta \subset \Br_A$.
    We assume the existence of $ t \in \Delta \cap (\Br_\Delta,\Br_\Delta)$. 
    Then $t \in \Delta \setminus A$ and $t \in \B$ thus $\lambda(t) \in \B \cap \Br_\Delta$. 
    By \cref{Lem:BRA}, $\lambda(t) \geq \min(\Delta)$, which is a contradiction.
\end{proof}

\section{A rough bound for the decomposition algorithm}
\label{sec:exp-n-n2}

The main goal of this section is to prove \cref{thm:enn2-easy}, which provides a first rough bound of structure constants of $\mathcal{L}(X)$ relative to Hall sets.
We start by an introduction to the classical decomposition algorithm over Hall sets in \cref{sec:rewriting} then proceed to the proof of our first estimate in \cref{sec:rewrite-estimate}.

\subsection{The classical decomposition algorithm} 
\label{sec:rewriting}

Let $\B \subset \Br(X)$ be a Hall set.
For every $a < b \in \B$, the bracket $[a, b]$ can be decomposed in $\mathcal{L}(X)$ on $\B$ using the following recursive algorithm:
\begin{itemize}
    \item if $b \in X$ or $\lambda(b) \leq a$, then $(a, b) \in \B$ so we are done,
    \item otherwise, writing $b = (\lambda(b), \mu(b))$ with $a < \lambda(b) < \mu(b)$, the Jacobi identity yields
    \begin{equation}
        [a, b]  = [[a, \lambda(b)], \mu(b)] + [\lambda(b), [a, \mu(b)]],
    \end{equation}
    then, 
    \begin{itemize}
        \item apply the algorithm to decompose $[a, \lambda(b)]$ as a $\sum \alpha_d \ev(d)$ over a finite subset of $\B$ and, for each $d$, apply the algorithm to decompose $[d, \mu(b)]$ or $- [\mu(b), d]$ on the basis (depending on whether $d < \mu(b)$ or $\mu(b) < d$),
        \item apply the algorithm to decompose $[a, \mu(b)]$ as a $\sum_{d} \alpha_d \ev(d)$ over a finite subset of $\B$ and, for each $d$, apply the algorithm to decompose $[\lambda(b), d]$ or $- [d, \lambda(b)]$ on the basis (depending on whether $\lambda(b) < d$ or $d < \lambda(b)$),
        \item return the sum of all these decompositions.
    \end{itemize}
\end{itemize}
It is not immediate to verify that the recursion  terminates.
One could hope to make a proof by induction on $|a|+|b|$.
Indeed, when decomposing $[a, \lambda(b)]$, one has $|a|+|\lambda(b)| < |a|+|b|$; but this argument is insufficient since, when decomposing $[d, \mu(b)]$, one has $|d|+|\mu(b)| = |a|+|b|$.

Nevertheless, it is well-known that this algorithm does converge, thanks to the properties of Hall sets, see e.g.\ \cite[Section~9]{zbMATH02156381} for a nice exposition (albeit with reversed order conventions), or \cite{zbMATH03217927,zbMATH04136169,zbMATH00417855,zbMATH03133557} for earlier occurrences of this algorithm.
The key point is to prove along the induction that, when $(a,b) \notin \B$, the elements $c \in \supp [a,b]$ satisfy $\lambda(c) > a$. 
This allows for an induction on the couple $(|a|+|b|, \min \{ a, b \})$.

In \cref{ss:rec-theta}, we give a cleaner interpretation of the induction process at stake.

\medskip

In particular, this algorithm yields a constructive proof that $\ev(\B)$ indeed spans $\mathcal{L}(X)$.

\medskip

This algorithm is described in many theoretical works (see all references above) and implemented in most Lie algebraic packages (see e.g.\ \texttt{Hall.\_rewrite\_bracket} in SageMath \cite{sagemath}, \texttt{lie\_basis.\_prod} in CoRoPa's LibAlgebra \cite{libalgebra} or \texttt{phbize} in LTP \cite{zbMATH02237869}).
However, its complexity has, up to our knowledge, not been investigated before, even in papers such as \cite{duleba2019algorithm} where the authors strove to optimize the complexity of the generation of the basis, but without considering the complexity of the decomposition algorithm.

The interpretation of \cref{ss:rec-theta} bounds the size of the associated call stack.
The estimates on the structure constants proved in the next paragraph (and in the sequel of the paper) can be seen as being first steps towards estimating the computational complexity of this algorithm (both with respect to space and time).
Indeed, assuming that representing in memory a bracket $b \in \Br(X)$ requires a space of size $\mathcal{O}(|b|)$, representing the result of the algorithm, i.e.\ the decomposition of $[a,b]$ on the basis, requires a space of size $\mathcal{O}((|a|+|b|) \| [a,b] \|_\B)$, if the algorithm returns a list of (signed) brackets whose sum equals $[a,b]$ in $\mathcal{L}(X)$.
Giving a precise estimate of the time-complexity of the decomposition algorithm is more difficult, since it depends on the time-complexity of the comparison operation, which depends on the specific order defining $\B$.
However, we expect that, given the time-complexity of the comparison operation and the implementation details of the algorithm, adapting the methods we develop in the sequel could lead to bounds on the asymptotic time-complexity of the algorithm.

\subsection{A first naive bound for structure constants}
\label{sec:rewrite-estimate}

Even if $X$ is infinite, any given bracket $[a, b]$ only involves at most $|a|+|b|$ elements of $X$.
Thus, \cref{thm:enn2-easy} of the introduction is a direct consequence of the following estimate.
This estimate is by no means optimal and we will prove much tighter upper bounds in the next sections.
Nevertheless, we believe that it is interesting to remark that naively tracking the double induction on $(|a|+|b|, \min \{ a, b \})$ of this classical algorithm does not yield a very practical estimate.

\begin{theorem} \label{thm:Reutenauer}
    Assume that $X$ is finite.
    Let $\B \subset\Br(X)$ be a Hall set. For every $a < b \in \B$,
    \begin{itemize}
    \item either $(a,b) \in \B$ and then $\left\| [a,b] \right\|_\B=1$,
    \item or $[a,b]=\sum \alpha_c \ev(c)$ where the sum is finite, $\alpha_c \in \Z$, $c \in \B \setminus X$,
    \begin{equation} \label{eq:a<c'}
        a < \lambda(c)
    \end{equation}
    and one has the size estimate
    \begin{equation} \label{eq:2Xn2}
        \left\| [a,b] \right\|_\B = \sum |\alpha_c| \leq 2^{|X|^{\frac{n(n+1)}{2}}},
    \end{equation}
    where $n=|a|+|b|$.
    \end{itemize}
\end{theorem}

\begin{proof}
    The proof consists in applying the decomposition algorithm described in \cref{sec:rewriting}, while keeping track of the number of required iterations. 
    Heuristically, each iteration consists in applying the Jacobi identity, which doubles the number of terms.
    We construct by induction on $n \geq 2$ a non-decreasing sequence $C_n \in [1,+\infty)$ such that the following property holds.
    
    $\mathcal{H}_n$: ``\emph{for every $a < b\in \B$ such that $|a|+|b|=n$, either $(a,b) \in \B \setminus X$, or $[a, b]=\sum \alpha_c \ev(c)$ where the sum is finite, $\alpha_c \in \Z$, $c \in \B$, $\lambda(c)>a$ and, in both cases $\| [a, b] \|_\B \leq C_n$}''.
    
    \medskip \noindent \textbf{Initialization for $n = 2$}.
    Let $a < b \in \B$ with $|a|+|b| = 2$.
    Then $a, b \in X$ and $(a,b) \in \B$ so the conclusion holds with $C_2 := 1$.

    \medskip \noindent \textbf{Induction on $n \geq 3$}.
    The set
    \begin{equation}
        \mathcal{G}_n := \{ a \in \B ; \enskip \exists b \in \B \text{ such that } a < b \text{  and  } |a|+|b|=n  \}
    \end{equation}
    is finite and $r:=|\mathcal{G}_n| \leq |\{ a \in \B ; |a| \leq n-1 \}| \leq |X|^n$ (for a more precise estimate, albeit generically of the same magnitude, one could use Witt's formula \cite{zbMATH03116668}).
    Let $a^1>\dots>a^r$ be the elements of $\mathcal{G}_n$ in decreasing order.
    We prove by induction on $j \in \intset{1, r}$ the following property.
    
    $\mathcal{H}_n'(j)$: ``\emph{for every $b \in \B$ such that $a^j<b$ and $|a^j|+|b|=n$ then either $(a^j,b) \in \B$, or $[a^j,b]=\sum \alpha_c \ev(c)$ where the sum is finite, $\alpha_c \in \Z$, $c \in \B \setminus X$, $\lambda(c)>a^j$ and, in both cases $\| [a^j, b] \|_\B \leq C_n'(j)$ where $C_n'(j) := (2 C_{n-1})^{j-1}$}''.

    \step{Initialization for $j=1$} 
    Let $b \in \B$ such that $a^1<b$ and $|a^1|+|b|=n$. If $b \in X$, then $(a^1,b) \in \B$ and the conclusion holds. From now on, we assume that $b \notin X$. 
    Then 
\begin{equation} \label{b2=[b3,b4]}
    b=(\lambda(b),\mu(b))  \text{  where  } \lambda(b) <  \mu(b) \in \B \text{  and  } \lambda(\mu(b)) \leq \lambda(b) < b.
\end{equation}
If $\lambda(b) \leq a^1$, then $(a^1,b)=(a^1,(\lambda(b),\mu(b))) \in \B$ and the conclusion holds. 
Now, let us prove that, by definition of $a^1$, the situation $a^1<\lambda(b)$ cannot happen. 
Working by contradiction, we assume that $a^1<\lambda(b)$.
By Jacobi's identity
\begin{equation}
    [a^1,b]=[a^1,[\lambda(b),\mu(b)]]=[[a^1,\lambda(b)],\mu(b)]+[\lambda(b),[a^1,\mu(b)]].
\end{equation}
Thus, at least one of the two terms in the right-hand side is nonzero. 
Let us assume it is the first one.
By the induction assumption on $n$, we have $[a^1,\lambda(b)] = \sum \alpha_d \ev(d)$ where the sum is finite and non trivial, $\alpha_d \in \Z$, $d \in \B$ and $d > \lambda(d) \geq a^1$. 
Therefore $[[a^1,\lambda(b)],\mu(b)]=\sum\alpha_d [d,\mu(b)]$ where the sum is finite and non trivial. 
Let $d \in \B \setminus \{\mu(b)\}$ such that $\alpha_d \neq 0$ and $d \neq \mu(b)$. 
Then $d$ and $\mu(b)$ are two distinct elements of $\B$, with total length $n$, and strictly greater than $a^1$, which is in contradiction with the definition of $a^1$.

This concludes the proof for $j=1$ with $C_n'(1) := 1 = (2 C_{n-1})^{0}$.

    \step{Induction on $j$} 
    Let $j \in \intset{2 , r}$. 
    We assume that $\mathcal{H}_n'(k)$ holds for $k \in \intset{1,j-1}$. 
    Let $b \in \B$ such that $a^j<b$ and $|a^j|+|b|=n$. 
    If $b \in X$, then $(a^j,b) \in \B$ and the conclusion holds. 
    From now on, we assume that $b \notin X$. 
    Then \eqref{b2=[b3,b4]} holds.
    If $\lambda(b) \leq a^j$ then $(a^j,b)=(a^j,(\lambda(b),\mu(b))) \in \B$ and the conclusion holds. 
    From now on, we assume that $a^j<\lambda(b)$.
    By Jacobi's identity
    \begin{equation}
        [a^j,b]=[a^j,[\lambda(b),\mu(b)]]=[[a^j,\lambda(b)],\mu(b)]+[\lambda(b),[a^j,\mu(b)]].
    \end{equation}
    Let us decompose both terms.
    
\begin{itemize}
    \item \emph{Study of $[[a^j,\lambda(b)],\mu(b)]$}. 
    By the induction assumption on $n$, we have $[a^j,\lambda(b)]=\sum \alpha_d \ev(d)$ where the sum is finite, $\alpha_d \in \Z$, $d \in \B$, $\lambda(d) \geq a^j$ and $\sum|\alpha_d| \leq C_{|a^j|+|\lambda(b)|} \leq C_{n-1}$ (since the sequence is non-decreasing).
    Then $[[a^j,\lambda(b)],\mu(b)]=\sum \alpha_d [d,\mu(b)]$. 
    For every $d \in \B$ such that $\alpha_d \neq 0$, $\min\{d,\mu(b)\} > a^j$ (because $d>\lambda(d)\geq a^j$ and $\mu(b)>\lambda(b)>a^j$), thus it belongs to $\{a^1,\dotsc,a^{j-1}\}$. 
    Therefore $[d,\mu(b)]=\sum \beta^d_c \ev(c)$ where the sum is finite, $\beta^d_c \in \Z$, $c \in \B$, $\lambda(c) \geq \min\{d,\mu(b)\} > a^j$ 
    (which proves \eqref{eq:a<c'})
    and $\|[d,\mu(b)]\|_\B=\sum|\beta^c_c| \leq C_n'(j-1)$.
    (We implicitly use the fact that $i \mapsto C_n'(i)$ is non-decreasing). 
    Eventually
    \begin{equation}
        \| [[a^j,\lambda(b)],\mu(b)]\|_\B \leq \sum_{d \in \B} |\alpha_d| \sum_{c \in \B} |\beta^d_c| \leq C_{n-1} C_n'(j-1).
    \end{equation}
    
    \item \emph{Study of $[\lambda(b),[a^j,\mu(b)]]$:} Working in the same way, $\| [\lambda(b),[a^j,\mu(b)]]\|_\B \leq  C_{n-1} C_n'(j-1)$.
\end{itemize}
Finally, we get the expected decomposition of $[a^j,b]$ together with
\begin{equation}
    \|[a^j,b] \|_\B \leq 2 C_{n-1} C_n'(j-1) = C_n'(j) = (2 C_{n-1})^{j-1}.
\end{equation}
Thus, $\mathcal{H}_n$ holds with
\begin{equation} \label{eq:rewrite-cn}
    C_n := C_n'(r) = (2 C_{n-1})^{r-1} \leq (2 C_{n-1})^{|X|^n-1}.
\end{equation}

    \medskip \noindent \textbf{Estimate of $C_n$}.
    Our strategy is to find and estimate a sequence $D_n$ for $n \geq 2$ such that $D_2 \geq C_2$ and $D_n = (2 D_{n-1})^{|X|^n-1}$ for $n \geq 3$.
    Assuming temporarily that $D_{n-1} \geq 2^{|X|^n-1}$ (which we check below), the induction relation yields
    \begin{equation}
        D_n \leq D_{n-1}^{|X|^n}.
    \end{equation}
    Taking the logarithm twice and solving the nested recurrence yields
    \begin{equation}
        \ln \ln D_n \leq \ln \ln D_2 + \sum_{k=3}^n k \ln |X|.
    \end{equation}
    So $D_n \leq D_2^{|X|^{\frac{n(n+1)}{2}-3}}$. 
    In particular, choosing $D_2 = 2^{|X|^{\frac{2(2+1)}{2}}} > 1 = C_2$ yields
    \begin{equation}
        D_n \leq 2^{|X|^{\frac{n(n+1)}{2}}}
    \end{equation}
    for every $n \geq 2$. 
    Thus we check \emph{a posteriori} that $D_{n-1} \geq 2^{|X|^n-1}$ for every $n \geq 3$ so the estimate was legitimate.
    
    Then, by construction, one has $C_n \leq D_n$ for every $n \geq 2$ which concludes the proof of \eqref{eq:2Xn2}.
\end{proof}

\section{Iterated brackets over alphabetic subsets}
\label{sec:n!}

The main goal of this section is to prove \cref{thm:n!-easy}, which yields a sharp bound for the norm of brackets of elements known to lie in some specific subsets of a Hall set.
We start by introducing the notion of \emph{alphabetic subsets} in \cref{subsec:alphabetic}.
Then, we prove a precised version of \cref{thm:n!-easy} in \cref{subsec:refinedbound}.
Eventually, we explain in what sense this estimate can be seen as optimal in \cref{subsec:Optim_n!}.

\subsection{Alphabetic subsets}
\label{subsec:alphabetic}

\begin{definition}[Alphabetic subset] \label{def:alphabetlike}
   Let $\B \subset \Br(X)$ be a Hall set and $A \subset \B$. We say that $A$ is \emph{alphabetic} when for all $a_1, a_2 \in A$ such that $a_1 < a_2$, we have $(a_1,a_2) \in \B$.
\end{definition}

For instance, $X$ is alphabetic. 
Every subset of an alphabetic set is also alphabetic. 

\begin{remark}
    The set of indeterminates $X$ is sometimes referred to as the ``alphabet''.
    The fact that subsets of $\B$ satisfying \cref{def:alphabetlike} behave like $X$ explains our choice of terminology.
\end{remark}

A characterization of alphabetic subsets of $\Br(X)$ is given by the following statement.

\begin{lemma} \label{Prop:Ba}
    Let $\B \subset \Br(X)$ be a Hall set.
    For every $a \in \B$, $\B_a := \{a\} \cup \{b \in \B; \enskip (a,b) \in \B\}$ is an alphabetic subset of $\B$. 
    Conversely, for any finite alphabetic subset $A$ of $\B$, there exists $a \in \B$ such that $A \subset \B_a$.
\end{lemma}

\begin{proof}
    Let $b_1<b_2 \in \B_a$. If $b_1=a$ or $b_2 \in X$ then $(b_1,b_2) \in \B$. 
    Otherwise $\lambda(b_2) \leq a < b_1$, because $(a,b_2)$ and $(a,b_1) \in \B$, thus $(b_1,b_2) \in \B$.

    Now let $A \subset \B$ be a finite alphabetic subset. If $A \subset X$, let $a = \min A$, otherwise let $a = \max \{\lambda(b); \enskip b \in A \setminus X \}$.
    The case $A \subset X$ is straightforward.
    Assume $A \not\subset X$, let $a^+ \in A\setminus X$ such that $a = \lambda(a^+)$.
    Let $t \in A$. 
    We must show that $t = a$ or $(a,t) \in \B$.
    \begin{itemize}
        \item If $t = a^+$, then $a = \lambda(a^+) < a^+ = t$ and $\lambda(t) = a \leq a$ so $(a, t) \in \B$.
        \item If $t < a^+$, since $A$ is alphabetic, $(t, a^+) \in \B$ so $\lambda(a^+) = a \leq t$.
        If $t = a$, we are done.
        Else, either $t \in X$ and thus $(a,t) \in \B$ or $t \in A\setminus X$ and, by definition of $a$, $\lambda(t) \leq a$ so $(a,t)\in \B$.
        \item If $t > a^+$, then $t > a$ because $a = \lambda(a^+) < a^+$.
        Moreover, either $t \in X$ and thus $(a,t) \in \B$ or $t \in A\setminus X$ and, by definition of $a$, $\lambda(t) \leq a$ so $(a,t)\in \B$.
    \end{itemize}
    This concludes the proof of the reciprocal statement.
\end{proof}

\begin{example}
    If $X=\{X_0,X_1\}$ and $\B \subset \Br(X)$ is a Hall set with $X_0<X_1$, then the set $\{X_0,(X_0,X_1),(X_0,(X_0,X_1))\}$ is an alphabetic subset of $\B$ because it is a subset of $\B_{X_0}$.
\end{example}

The following result (which will be helpful in the sequel) shows that a free alphabetic subset behaves just like a set of indeterminates in the Lie subalgebra that it generates.

\begin{proposition}\label{Prop:free-lie-alphabetic}
    Let $\B \subset \Br(X)$ be a Hall set and $A \subset \B$ be alphabetic and free. 
    By \cref{def:free-subset}, the canonical surjection $\e$ : $\Br(A) \to \Br_A$ is an isomorphism.
    Then $\B' := \e^{-1}(\B \cap \Br_A)$, endowed with (the restriction of the preimage by $\e$ of) the order on $\B$, is a Hall set of $\Br(A)$.
    Moreover, the canonical map $\mathcal{L}(A) \to \mathcal{L}(X)$ is an isometry with respect to the norms $\|\cdot\|_{\B'}$ and~$\|\cdot\|_\B$.
\end{proposition}
 
\begin{proof}
    First $\B'$ is a totally ordered subset of $\Br(A)$ with the order $b_1 <' b_2$ iff $\e(b_1) < \e(b_2)$ (in $\B$).
    Let us check that the three items of \cref{Def2:Laz} are satisfied.
    \begin{itemize}
        \item $A \subset \B'$. Indeed, $A \subset \B$ and $\e$ is the identity on $A$.
        
        \item For all $b_1, b_2 \in \Br(A)$, $(b_1,b_2) \in \B'$ iff $b_1, b_2 \in \B'$, $b_1 <' b_2$ and $b_2 \in A$ or $\lambda(b_2) \leq' b_1$.
        
        Indeed, let $b_1, b_2 \in \Br(A)$.
        Assume that $(b_1, b_2) \in \B'$.
        Then $\e(b_1,b_2) \in \B$.
        Since $\B$ is a Hall set, $\e(b_1), \e(b_2) \in \B$, $\e(b_1) < \e(b_2)$ and $\e(b_2) \in X$ or $\lambda(\e(b_2)) \leq \e(b_1)$.
        Thus, $b_1, b_2 \in \B'$ and $b_1 <' b_2$.
        Moreover, if $b_2 \notin A$, then $\lambda(b_2) \in \Br(A)$ and $|b_2|_A > 1$ so $|\e(b_2)| > 1$ so $\e(b_2) \notin X$ and $\lambda(\e(b_2)) \in \B$.
        Hence, $\lambda(\e(b_2)) \leq \e(b_1)$ yields $\lambda(b_2) \leq' b_1$.
        
        Conversely, assume that $b_1, b_2 \in \B'$, $b_1 <' b_2$ and $b_2 \in A$ or $\lambda(b_2) \leq' b_1$.
        When $\lambda(b_2) \leq' b_1$, we obtain that $(b_1, b_2) \in \B'$ using the fact that $\B$ is a Hall set.
        When $b_2 \in A$, let $a := \lambda^k(b_1)$ where $k \in \N$ is chosen such that $a \in A$ (iterated left factor, up to falling in $A$).
        Since $b_1 \in \B'$, $\e(b_1) \in \B$ and $\e(a) = \lambda^k(\e(b_1)) \in \B$.
        Moreover, by the Hall order property, $a = \e(a) \leq \e(b_1) < \e(b_2) = b_2$ (with equality if and only if $k = 0$).
        Since $A$ is alphabetic, $(a, b_2) \in \B$, so $\lambda(b_2) \leq a$.
        Hence $\lambda(b_2) \leq \e(b_1)$ and $\e(b_1, b_2) \in \B$ so $(b_1, b_2) \in \B'$.
        
        \item For all $b_1, b_2 \in \B'$ such that $(b_1,b_2) \in \B'$, one has $b_1 <' (b_1,b_2)$. 
        
        Indeed, since $(b_1, b_2) \in \B'$, $\e(b_1, b_2) \in \B$, so $\e(b_1) < \e(b_1, b_2)$. 
    \end{itemize}
    This shows that $\B'$ is indeed a Hall set over $A$.
    
    By the universal property of the free Lie algebra $\mathcal{L}(A)$, there is a canonical morphism of Lie algebras $h$ : $\mathcal{L}(A) \to \mathcal{L}(X)$ (which maps any element of $A$ to itself in $\mathcal{L}(X)$). 
    Since $\B'$ is a basis of $\mathcal{L}(A)$ and $h(\B') \subset \B$ is a linearly independent subset of $\mathcal{L}(X)$, this map is injective. 
    By definition, its image is the Lie subalgebra of $\mathcal{L}(X)$ generated by $A$, which is the linear subspace generated by $\B'$. 
    Let $x = \sum_{b \in \B'} \alpha_b \ev(b)$, then by definition $h(x) = \sum_{b \in \B} \alpha_b \ev(b)$ (where $\alpha_b = 0$ when $b \notin \B \cap \Br_A$). 
    This shows that this map is indeed an isometry.
\end{proof}

The following stability property of alphabetic subsets of $\B$ is a key point of this article.

\begin{lemma}\label{lem:alphabetic}
Let $\B \subset \Br(X)$ be a Hall set and $A \subset \B$ be alphabetic. Assume $a_0 = \min A$ exists, then for all $a \in A$, $a \neq a_0$, the set $A \cup \{ (a_0,a) \}$ is alphabetic.
\end{lemma}

\begin{proof}
Let $a \in A$ such that $a \neq a_0$. Then since $A$ is alphabetic, $(a_0,a) \in \B$. Hence $A \cup \{ (a_0,a) \} \subset \B$.

Let $a_1, a_2 \in A \cup \{ (a_0,a) \}$ with $a_1<a_2$. If $a_1, a_2 \in A$, then $(a_1, a_2) \in \B$, so we may assume that either $a_1$ or $a_2$ is equal to $(a_0,a)$. If $a_1 < a_2 = (a_0,a)$, then the minimality of $a_0$ implies that $a_1 \geq a_0$, so that $(a_1, a_2) \in \B$. Now if  $(a_0,a) = a_1 < a_2$, then since $A$ is alphabetic, either $a_2 \in X$ or $a_2 \notin X$ and $a_0 \geq \lambda(a_2)$. In both cases, since $a_0 < (a_0,a)$, we have $(a_1, a_2) \in \B$.
\end{proof}

\subsection{Bound for iterated brackets over alphabetic subsets}
\label{subsec:refinedbound}

\begin{definition}[Multisets, i.e.\ sets with multiplicity]
By convention, we use blackboard bold font to name them and bag brackets to define them, e.g.\ $\mathbb{A} := \lbag 2, 2, 3 \rbag$.
Multiplicity matters, so, $\mathbb{A} \neq \lbag 2, 3 \rbag$. 
But order does not matter, so $\mathbb{A} = \lbag 2, 3, 2 \rbag$.
For such multisets, we define their cardinal as the sum of their multiplicities (so $|\mathbb{A}| = 3$) and their support as the underlying set (so $\supp \mathbb{A} = \{ 2, 3 \}$). If $\mathbb{A}_1, \mathbb{A}_2$ are two multisets, $\mathbb{A}_1 + \mathbb{A}_2$ denotes their sum, i.e.\ the multiset in which the multiplicities are the sum of those within $\mathbb{A}_1$ and $\mathbb{A}_2$. 
We can also define the difference $\mathbb{A}_1 - \mathbb{A}_2$ (if the multiplicities in $\mathbb{A}_1$ are greater than those in $\mathbb{A}_2$). More formally, a multiset whose support is $A$ is a map $\mathbb{A} : A\to \N$. The cardinal of $\mathbb{A}$ is $\sum_{a \in A} \mathbb{A}(a)$, while the sum (resp.\ the difference) of the multisets is the sum (resp.\ the difference, when it exists) of the maps. 
\end{definition}


\begin{definition}[Trees on a multiset supported in $\B$]
    Let $\B \subset \Br(X)$ be a Hall set and $\mathbb{A}$ be a multiset with $A := \supp \mathbb{A} \subset \B$.
    We denote by $\Tr(\mathbb{A})$ the subset of $\Br(A)$ (and thus of $\Br(\B)$) whose elements are brackets of the elements of $\mathbb{A}$, involved according to their multiplicity in $\mathbb{A}$. 
    Thus, for $t \in \Tr(\mathbb{A})$, $|t|_A=|t|_\B=|\mathbb{A}|$.
\end{definition}

\begin{definition}[Leaves]
    For $A \subset \B$ and $t \in \Br(A)$, $L_A(t)$ (resp.\ $\mathbb{L}_A(t)$) denotes the set (resp.\ multiset) of leaves of $t$ in $A$. 
    In particular $t \in \Tr( \mathbb{L}_{A}(t) )$ and $\supp \mathbb{L}_A(t) = L_A(t)$.
\end{definition}

\begin{theorem} \label{Prop:n!}
    Let $\B \subset \Br(X)$ be a Hall set and $t \in \Br(\B)$ such that $L_\B(t)$ is alphabetic. Then
    \begin{equation}
        \| \ev(\e(t)) \|_\B \leq (|t|_\B-1)!
    \end{equation}
    and $\supp \ev(\e(t)) \subset \e(\Tr(\mathbb{L}_\B(t)))$: the supporting elements are obtained by creating new brackets by combining the leaves (with their multiplicity) of $t$.
\end{theorem} 

\begin{proof}
    The proof is by induction on $n=|t|_\B$ and obvious when $n = 1$. Now let $n\geq 2$, assume that the result holds for all $k<n$, and prove that it holds for $n$.
    
    Let $a_0$ be the smallest element of $L_\B(t)$. 
    There exists $k \in \{1,\dots,n-1\}$,  $a_{i_1}, \dots, a_{i_k} \in \mathbb{L}_\B(t)$ and $v \in \Tr(\lbag a_{i_1}, \dots, a_{i_k} \rbag)$ such that $\langle a_0,v \rangle$ or $\langle v,a_0 \rangle$ is a sub-tree of $t$ and either $v=a_0$ or $a_0$ does not appear in $v$. 
    In the first case, $\| \e(t) \|_\B=0$, so we may assume that $a_0$ does not appear in $v$. Using the iterated Jacobi identity, we can write $[a_0,\e(v)] = \sum_{j=1}^k \ev(\e(v_{i_j}))$, where $v_{i_j} \in \Br(\B)$ is the bracket obtained from $v$ by replacing $a_{i_j}$ by $\langle a_0,a_{i_j} \rangle$. 
    Then, by bilinearity of $[\cdot,\cdot]$, we get $\ev(\e(t)) = \sum_{j=1}^k \ev(\e(t_j))$, where for all $j \in \intset{1,k}$, $t_{j} \in \Br(\B)$ and $\mathbb{L}_\B(t_j)= \mathbb{L}_\B(t) + \lbag (a_0,a_{i_j}) \rbag - \lbag a_0, a_{i_j} \rbag $
    By \cref{lem:alphabetic}, $L_\B(t_j)$ is alphabetic.
    by the induction hypothesis, for all $j \in \intset{1,k}$, we have $\|\ev(\e( t_j ))\|_\B \leq (n-2)!$ and $\supp \ev(\e(t_j)) \subset \e ( \Tr( \mathbb{L}_\B(t_j) ) ) \subset \e ( \Tr( \mathbb{L}_\B(t)  ) )$.
    The expression $\ev(\e(t)) = \sum_{j=1}^k \ev(\e(t_j ))$ and the fact that $k <n$ then show that $\|\ev(\e(t))\|_\B \leq (n-1)!$.
\end{proof}

\subsection{Optimality case}
\label{subsec:Optim_n!}

\begin{remark}
    For $b \in \Br(A)$, $L_A(b)$ denotes the set of leaves in $A$ of the tree $b$. 
    When $A$ is a free subset, this notion can be transported on $\Br_A$, because $\e:\Br(A) \rightarrow \Br_A$ is an isomorphism.
    Setting $L_A(\e(b)) := L_A(b)$ for $b \in \Br(A)$ extends $L_A$ to $\Br_A$.
\end{remark}

\begin{proposition} \label{Prop:sature_n!}
    Let $\B \subset \Br(X)$ be a Hall set, $n\in\N^*$, $a_1<\dots<a_{n} \in \B$. We assume that
    \begin{itemize}
    \item $A :=\{a_1,\dots,a_{n}\}$ is a free alphabetic subset of $\B$,
    \item for all $b_1, b_2 \in \B \cap \Br_A$ with $L_A(b_1) \cap L_A(b_2) = \emptyset$ then $b_1<b_2$ iff $\max L_A(b_1) < \max L_A(b_2)$.
    \end{itemize}
    Then
    \begin{equation}
        \| [ \dotsb [ a_{n},a_{n-1} ], \dotsc , a_1 ] \|_\B = (n-1)!.
    \end{equation}
\end{proposition}

\begin{proof}
    We proceed by induction on $n \in \N^*$. 
    The conclusion holds for $n=1$.
    Let $n \geq 2$. 
    We assume the property proved up to $(n-1)$. 
    Let $A=\{a_1<\dots<a_n\}$ be a free alphabetic subset of $\B$ such that the order of $\B$ is compatible with $\max L_A$(in the sense of the second point above). 
    Let $w \in \Br(A)$ be defined by $w:=\langle \dotsb \langle a_{n},a_{n-1} \rangle, \dots,a_1 \rangle$.

    For any $k \in \intset{2,n}$, we have $(a_1,a_k) \in \B$ because $A$ is alphabetic and, by compatibility of the Hall set order $<$ with $\max L_A$, $a_{k-1}<(a_1,a_k)<a_{k+1}$. 
    Thus $a_2 < \dots < a_{k-1} < (a_1,a_k) < a_{k+1} < \dots < a_n$.
    Let $\mathbb{A}_k:=\lbag a_2 , \dots , a_{k-1} , (a_1,a_k) , a_{k+1} , \dots , a_n  \rbag$ and $A_k=\supp(\mathbb{A}_k)$.
    By \cref{lem:alphabetic}, $A_k$ is alphabetic.
    By \cref{Lem:libre}, $A_k$ is free. 
    Indeed, $A_k \subset \B \cap \Br_A$ and $a_1=\lambda((a_1,a_k))<\min(A_k)$.

    \step{Step 1: We prove the compatibility of the order $<$ with $\max L_{A_k}$}
    Let $b_1, b_2 \in \B \cap \Br_{A_k}$ with $L_{A_k}(b_1) \cap L_{A_k}(b_2) = \emptyset$. 
    We assume $\max L_{A_k}(b_1) < \max L_{A_k}(b_2)$. 
    Let us prove that $b_1<b_2$.

    We have $b_1, b_2 \in \B \cap \Br_A$ and $L_A(b_1) \cap L_A(b_2)=\emptyset$. 
    Moreover, $\max L_A(b_j) = \max L_{A_k}(b_j)$ except when $\max L_{A_k}(b_j)=(a_1,a_k)$ and then $\max L_{A}(b_j)=a_k$. 
    In any case, $\max L_A(b_1) < \max L_A(b_2)$ thus $b_1<b_2$ because $(a_1, a_k)$ is inserted at the position of $a_k$ thanks to the order on $A$.
    
    \step{Step 2: We apply the induction assumption} 
    Let $w_k \in \Br(A_k)$ be defined by
    \begin{equation}
        w_k:=\langle\dotsb \langle\langle\langle\langle a_{n},a_{n-1}\rangle,\dots a_{k+1}\rangle, ( a_1,a_k) \rangle, a_{k-1} \rangle \dots ,a_2 \rangle .
    \end{equation}
    Then $\| \ev(\e(w_k)) \|_\B = (n-2)!$.
    

    \step{Step 3: We prove that the sets $\B \cap \e(\Tr(\mathbb{A}_{k}))$ for $k \in \intset{2,n}$ are two by two disjoint} 
    Let $j \neq k \in \intset{2,n}$, $B_j \in \Tr(\mathbb{A}_{j})$ and $B_k \in \Tr(\mathbb{A}_{k})$ such that $b_j=\e(B_j)$ and $b_k = \e (B_k)$ belong to $\B$.
    Then $B_j$ and $B_k$ are two different elements of $\Br(\Delta)$ where 
    \begin{equation}
        \Delta:=A_j \cup A_k=\{a_2,\dotsc,a_n,(a_1,a_j),(a_1,a_k)\}.
    \end{equation}
    By \cref{Lem:libre}, $\Delta$ is free because $\Delta \subset \B \cap \Br_A$, $A$ is alphabetic and $a_1<\min(\Delta)$.
    Thus $\e$ is injective on $\Br(\Delta)$ and $b_j \neq b_k$.

    \step{Step 4: Conclusion}
    Iterating Jacobi's identity proves that $- \ev(\e(w)) = \ev(\e(w_2)) + \dotsb + \ev(\e(w_n))$. 
    By Steps 2 and 3, $\| \ev(\e(w)) \|_\B = (n-1)!$.
\end{proof}

\begin{corollary} \label{Cor:optn!}
    Let $n \geq 2$ and $X=\{X_1,\dots,X_n\}$. 
    There exists a Hall set $\B \subset \Br(X)$ such that 
    \begin{equation} \label{eq:cor-optn!}
        \| [\dotsb[X_n,X_{n-1}],\dotsc,X_1] \|_\B = (n-1)!.
    \end{equation}
\end{corollary}

\begin{proof}
    For $b \in \Br(X)$, we use the notation $L(b) \subset X$ to denote the set of leaves involved in $b$.

    The strategy consists in constructing a Hall order $<$ on $\Br(X)$ such that $X_1<\dots<X_n$ and, for every $b_1, b_2 \in \Br(X)$ with $L(b_1) \cap L(b_2)=\emptyset$, then $b_1<b_2$ iff $\max L(b_1) < \max L(b_2)$.
    Then, considering the Hall set $\B$ of $\Br(X)$ associated with this order and applying \cref{Prop:sature_n!} to the free subset $A:=\{X_1,\dots,X_n\}$ yields \eqref{eq:cor-optn!}.

    Let $\prec$ be any Hall order on $\Br(X)$. 
    Such an order does exist: see the second paragraph of \cref{Rk:Zorn} or \cref{Prop:Prolonger_ordre}. 
    We will use $\prec$ as an arbitrary order to compare brackets for which we have no particular requirement. 
    For $b_1 \neq b_2 \in \Br(X)$, we write $b_1<b_2$ when
    \begin{itemize}
        \item either $b_1=X_j$ and $b_2=X_k$ with $j<k$ (i.e.\ $X_1<\dots<X_n$),
        \item or $\max L(b_1) < \max L(b_2)$,
        \item or $\max L(b_1) = \max L(b_2)$ and $b_1 \prec b_2$.
    \end{itemize}
    Since $<$ is defined as the lexicographic order on the couple $(\max L, \prec)$, it is a total order on $\Br(X)$ which satisfies the desired properties.
    
    It remains to check that $<$ is indeed a Hall order.
    Let $t \in \Br(X)$ with $|t| \geq 2$. 
    If $\max L(\lambda(t))<\max L(t)$ then $\lambda(t)<t$. 
    Otherwise, $\max L(\lambda(t)) =\max L(t)$ and $\lambda(t) \prec t$ since $\prec$ is a Hall order, thus $\lambda(t)<t$.
\end{proof}

\section{A refined bound stemming from brackets structure}
\label{Sec:refined_general_bound}

The main goal of this section is to prove \cref{thm:en1-easy}, which provides a sharp bound for the structure constants of $\mathcal{L}(X)$ relative to Hall sets in the general case.
We start by introducing the notion of \emph{relative folding} in \cref{subset:Treeb1/b2} and estimate its length in \cref{ss:theta}.
Then, we prove a precised version of \cref{thm:en1-easy} in \cref{subsec:bound_2}.
We explain in what sense this estimate can be seen as optimal in \cref{subsec:optim_2}.
Eventually, we prove in \cref{ss:rec-theta} that the \emph{length of the relative folding} provides a natural strictly decreasing indexation of the classical recursive decomposition algorithm of \cref{sec:rewriting}.

\subsection{Relative folding}
\label{subset:Treeb1/b2}

Let $\B \subset \Br(X)$ be a Hall set. 
We define a notion of \emph{relative folding} which will be of paramount importance for tracking, during the execution of the decomposition algorithm of \cref{sec:rewriting}, which brackets fall directly in the basis and which must be split in two using Jacobi's identity.


\begin{definition} \label{def:folding}
   For $a < b \in \B$, the \emph{relative folding} of $b$ with respect to $a$  is the tree $\rf{a}{b} \in \Br(\B)$ defined by induction by
   \begin{equation}
        \rf{a}{b} := \begin{cases}
        b & \text{ when } (a,b) \in \B, \\
        \langle \rf{a}{\lambda(b)},\rf{a}{\mu(b)} \rangle & \text{ otherwise}, 
       \end{cases}
   \end{equation}
   which makes sense as, when $(a,b) \notin \B$, then $b \notin X$ and $a < \lambda(b) < \mu(b)$.
\end{definition}

\begin{lemma} \label{Prop:Ta(b)_feuille_b}
    For each $c \in L_\B(\rf{a}{b})$, $(a,c) \in \B$.
    Moreover, $b=\e( \rf{a}{b} )$.
\end{lemma}

\begin{proof}
    These are immediate consequences of \cref{def:folding}.
\end{proof}

\begin{proposition} \label{Prop:Ta(b)_alphabetic}
    For $a < b \in \B$, $\{a\} \cup L_\B( \rf{a}{b} )$ is an alphabetic subset of $\B$. 
\end{proposition}

\begin{proof}
    By \cref{Prop:Ta(b)_feuille_b}, the considered set is a subset of $\B_a$, which is alphabetic by  \cref{Prop:Ba}.
\end{proof}

\begin{definition}
   For $a < b \in \B$, we define $\theta_a(b) := |\rf{a}{b}|_\B \in \N^*$.
\end{definition}

\begin{example}\label{ex:tree_fold}
    Assume that $\B \subset \Br(\{ X_0, X_1 \})$ is a Hall set such that the tree of \cref{ex:tree}, $t := (((X_0,X_1),((X_0,X_1),X_1)),((X_0,X_1),X_1)))$ belongs to $\B$.
    Then $a := (X_0,X_1) \in \B$ and $\rf{a}{t} = \langle b_1, b_2 \rangle$, where $b_1 = ((X_0,X_1),((X_0,X_1),X_1))$ and $b_2 = ((X_0,X_1),X_1))$.
    So $\theta_a(t) = 2$. This is illustrated by the following trees:
    \begin{equation}
        t = {\Tree [. [. [. $X_0$ $X_1$ ] [. [. $X_0$ $X_1$ ] $X_1$ ] ] [. [. $X_0$ $X_1$ ] $X_1$ ] ]}
    \end{equation}
    and
    \begin{equation}
        \rf{a}{t} = {\Tree [. $((X_0,X_1),((X_0,X_1),X_1))$ $((X_0,X_1),X_1))$ ]}
    \end{equation}
\end{example}

\subsection{Properties of the length of the folding}
\label{ss:theta}

In order to achieve our goal of proving estimates depending on $|b|$, we prove in this paragraph bounds relating $\theta_a(b)$ and $|b|$, along with other elementary remarks on $\theta$.
These results will be used also in the sequel of the paper for other purposes. We start with the following elementary result.

\begin{lemma}\label{lem:some_elements_hall}
    Let $\B \subset \Br(X)$ be a Hall set and $a,b \in \B$ such that $(a,b) \in \B$. 
    There exists a unique $r=r(a,b)\in\N^* \cup \{+\infty\}$ such that, for each $n \in \intset{0,r}$, $\dad_{b}^{n}(a) \in \B$ and, for each $n > r$, $\ad_{b}^{n-r} \dad_{b}^{r}(a) \in\B$. 
    Moreover, for each $n \in \intset{0,r-1}$, $\dad_{b}^{n}(a)<b$ and, if $r$ is finite, $b<\dad_b^r(a)$.
\end{lemma}

\begin{proof}
    Assume that $\{\dad_{b}^n(a);n \in \N\} \not\subset \B$.
    Let $r$ be the smallest integer such that $\dad_{b}^{r+1}(a) \notin \B$. 
    By definition, for all $n \in \intset{0,r}$, $\dad_{b}^n(a) \in \B$, thus $\dad_b^n(a)<b$ for $n \in \intset{0,r-1}$.
    Assume $\dad_{b}^r(a) < b$, then by hypothesis $b \notin X$ (otherwise $\dad_{b}^{r+1}(a) \in \B$), and since $\dad_{b}^r(a) > \lambda(\dad_{b}^r(a)) =  \dad_{b}^{r-1}(a) \geq \lambda(b)$ because $\dad_{b}^r(a) \in \B$, one sees that $(\dad_{b}^r(a),b) \in \B$ which is a contradiction. Hence, $b < \dad_{b}^r(a)$, and $b>\dad_{b}^{r-1}(a) = \lambda(\dad_{b}^r(a))$, so that $(b,\dad_{b}^r(a)) \in \B$, and  $\ad_b^p(\dad_{b}^r(a)) \in \B$ for all $p \in \N$.
\end{proof}

\begin{proposition} \label{Prop:theta/length}
    Let $\B \subset \Br(X)$ be a Hall set.
    \begin{enumerate}
    
    \item If $|X| \geq 3$ then, for every $a<b \in \B$, $\theta_a(b) \leq |b|$ and equality holds for arbitrarily long $b \in \B$.

    \item If $|X|=2$ then, for every $a<b \in\B$ with $|b|\geq 2$, $\theta_a(b) \leq |b|-1$ and equality holds for arbitrarily long $b \in \B$.
    \end{enumerate}
\end{proposition}

\begin{proof}
Let us prove the items successively.
\begin{enumerate}
    \item By Lemma \cref{Prop:Ta(b)_feuille_b}, $|b|$ is the sum of the lengths of the leaves of $\rf{a}{b}$, that are in number $\theta_a(b)$ and each leaf has length at least one $1$, one must have $|b| \geq \theta_a(b)$.
    
    If $X$ has at least 3 elements $X_1<X_2<X_3$ then, $\ad_{X_2}^p(X_3) \in \B$ for every $p \in \N$. 
    We prove by induction on $p \in \N$ that  $\theta_{X_1}(\ad_{X_2}^p(X_3))=p+1$.
    For $p=0$, $\theta_{X_1}(X_3)=1$ because $(X_1,X_3) \in \B$. For $p \geq 2$, $\theta_{X_1}(\ad_{X_2}^p(X_3))=\theta_{X_1}(X_2)+\theta_{X_1}(\ad_{X_2}^{p-1}(X_3))=1+p$.

    \item  We assume $X=\{X_0,X_1\}$ with $X_0<X_1$. Then, $X_0$ is the minimal element of $\B$. We prove the estimate by induction on $|b| \geq 2$. If $|b|=2$ then $b=(X_0,X_1)$ and $X_0 \leq a$ thus $(a,b) \in \B$ and $\theta_a(b)=1 = |b|-1$. If $|b| \geq 3$, then $|\lambda(b)| \geq 2$ or  $|\mu(b)| \geq 2$, thus, using the first statement and the induction assumption $\theta_a(b)=\theta_a(\lambda(b))+\theta_a(\mu(b))\leq |\lambda(b)|+|\mu(b)|-1 =|b|-1$.  
    
    Let $b_n$ be the unique element of $\B$ containing $X_0$ exactly once and $X_1$ exactly $n$ times, whose form is given in \cref{lem:some_elements_hall} and depends on $r(X_0,X_1)$.
    Then $|b_n| = n + 1$ and we easily get $\theta_{X_0}(b_n)=n$ by induction on $n \in \N^*$. \qedhere
\end{enumerate}
\end{proof}

\begin{lemma} \label{p:theta-dec}
    Let $\B \subset \Br(X)$ be a Hall set.
    For all $a \leq \tilde{a} < b \in \B$, one has $\theta_{\tilde a}(b) \leq \theta_{a}(b)$.
\end{lemma}

\begin{proof}
    We proceed by induction on $n := \theta_{a}(b)$.
    
    \step{Initialization for $n = 1$}
    Either $b \in X$, or $|b| \geq 2$ and $\lambda(b) \leq a \leq \tilde{a}$. 
    In both cases, $\theta_{\tilde{a}}(b) = 1$. 
    
    \step{Inductive step for $n \geq 2$}
    Let $a \leq \tilde a < b$ such that $\theta_a(b) = n$. 
    If $\theta_{\tilde{a}}(b) = 1$, the inequality holds.
    Otherwise, $a \leq \tilde{a} < \lambda(b)$, so $\theta_a(b) = \theta_a(\lambda(b)) + \theta_a(\mu(b))$ and $\theta_{\tilde{a}}(b) = \theta_{\tilde{a}}(\lambda(b)) + \theta_{\tilde{a}}(\mu(b))$, and thus the estimate follows from $\theta_{\tilde a}(\lambda(b)) \leq \theta_{a}(\lambda(b))$ and $\theta_{\tilde a}(\mu(b)) \leq \theta_{a}(\mu(b))$.
\end{proof}

\subsection{Bound for a bracket of two basis elements}
\label{subsec:bound_2}

\subsubsection{Warm-up version and structure of the supporting basis elements}
\label{sss:warm-up}

We start with a warm-up version of our main estimate \cref{thm:en1-easy}.
Let $\B \subset \Br(X)$ be a Hall set and $a < b \in \B$.
By \cref{Prop:Ta(b)_feuille_b}, $[a,b] = \ev(\e(t))$ where $t = \langle a, \rf{a}{b} \rangle$ is a bracket of length $|t|_\B = \theta_a(b)+1$ over $\{ a \} \cup L_\B(\rf{a}{b})$ which is an alphabetic subset of $\B$ by \cref{Prop:Ta(b)_alphabetic}. 
By \cref{Prop:n!} and the first item of \cref{Prop:theta/length},
\begin{equation} \label{eq:theta!}
   \|[a, b]\|_\B 
   \leq \theta_{a}(b) ! 
   \leq |b|!.
\end{equation}
Moreover, \cref{Prop:n!} also proves that $\supp [a,b] \subset \e(\Tr(\mathbb{A}))$ where $\mathbb{A} := \lbag a \rbag + \mathbb{L}_\B(\rf{a}{b})$.
This is a strong information on the structure of these supporting elements: they are obtained by creating new brackets with a single $a$ and the leaves (with their multiplicity) of $\rf{a}{b}$.

Estimate \eqref{eq:theta!} is nice because it is straight-forward and much better than \eqref{eq:2Xn2}.
However, it is not sharp.
This boils down to the fact that, contrary to \cref{Prop:n!} where no assumption is made on the structure of the considered bracket, here, $\rf{a}{b}$ models $b$, an element of the Hall set $\B$.
So there is a little more structure in $\langle a, \rf{a}{b} \rangle$ than the mere fact that its leaves form an alphabetic subset of~$\B$.
We exploit this remark in the next paragraph to tighten the bound from $\theta_a(b)!$ down to $\lfloor e (\theta_a(b)-1)! \rfloor$. 

\subsubsection{Sharp version}
 
In this paragraph, we prove an enhanced, sharp version of \cref{thm:en1-easy}.
We start by introducing the following operation, linked with the Jacobi identity, which will be helpful for the proof.

\begin{definition}[Jacobi distribution] \label{def:J}
    Let $\B \subset \Br(X)$ be a Hall set.
    Let $t \in \Br(\B)$ and $\ell_1, \ell_2 \in \B$ such that $(\ell_1, \ell_2) \in \B$, where $\ell_1$ is a localized leaf of $t$ and $\ell_2$ is a localized leaf of the sibling $t_1$ of $\ell_1$ in $t$ (see \cref{rk:localized-leaf-sibling}).
    Let $t_2$ be the tree constructed from $t_1$ where $\ell_2$ has been replaced by $(\ell_1, \ell_2)$.
    We denote by $J(t,\ell_1,\ell_2)$ the tree constructed from $t$, where the subtree $\langle \ell_1, t_1 \rangle$ (or $\langle t_1, \ell_1 \rangle$) has been replaced by $t_2$.
    In particular, one has $|J(t,\ell_1,\ell_2)|_\B = |t|_\B - 1$.
\end{definition}

\begin{example}
    Let $t$ be defined as
    \begin{equation}
        \Tree [. [. $\ell_3$ [. $\ell_1$ $\ell_2$ ] ] [. $\ell_4$ [. [. $\ell_5$ [. $\ell_6$ $\ell_7$ ] ] $\ell_8$ ] ] ]
    \end{equation}
    Then $J(t, \ell_4, \ell_6)$ is the tree
    \begin{equation}
        \Tree [. [. $\ell_3$ [. $\ell_1$ $\ell_2$ ] ]  [. [. $\ell_5$ [. $(\ell_4,\ell_6)$ $\ell_7$ ] ] $\ell_8$ ] ]
    \end{equation}
\end{example}
    
\begin{remark} \label{rk:localized-leaf-sibling}
    In \cref{def:J}, ``localized leaf'' means that the position of the leaf within the tree matters.
    There will be no ambiguity when using this notation in the sequel concerning which occurrence of $\ell_1$ and $\ell_2$ we wish to modify. 
    The ``sibling'' of a localized leaf $\ell$ in a tree $t$ denotes the other subtree sharing the same parent as $\ell$.
    
    This wording could be formalized using notions coming from graph theory and labeled trees, but we consider that, in our context, a full formalization would make the comprehension harder.
\end{remark}

We can now prove the following estimate, which of course implies \cref{thm:en1-easy} since $\theta_a(b) \leq |b|$ by the first item of \cref{Prop:theta/length}.
In the particular case $|X| = 2$, the refined estimate $\theta_a(b) \leq |b| - 1$ (second item of \cref{Prop:theta/length}) yields an even smaller upper bound.

\begin{theorem}
    \label{thm:en1}
    Let $\B \subset \Br(X)$ be a Hall set. 
    For any $a < b \in \B$,
    \begin{equation} \label{eq:e-theta}
        \| [a, b] \|_\B \leq \lfloor e (\theta_a(b) - 1)! \rfloor.
    \end{equation}
    Moreover, $\supp [a,b] \subset \e(  \Tr( \mathbb{A}))$ where $\mathbb{A}=\lbag a \rbag +\mathbb{L}_\B(\rf{a}{b})$.
\end{theorem}

\begin{proof}
    The part concerning the structure of supporting  basis elements has already been proved in \cref{sss:warm-up}. 
    So we only need to prove the size estimate.
    
    Let $n := \theta_a(b)$.
    When $n = 1$, $(a,b) \in \B$ so $\|[a,b]\|_\B = 1$.
    When $n = 2$, $(a,b) = (a,(b_1, b_2))$ where $b_1, b_2 \in \B$ and the set $\{ a, b_1, b_2 \}$ is alphabetic.
    Hence, by \cref{Prop:n!}, $\|[a,b]\|_\B \leq (3-1)!=2 = \lfloor e 1! \rfloor$. 
    We now assume that $n \geq 3$.
    
    \step{Strategy} 
    The proof is inspired by the method deployed in the proof of \cref{Prop:n!}: seeing $[a,b]$ as $\langle a, \rf{a}{b} \rangle$, a tree of $\Br(\B)$ over an alphabetic subset of $\B$, one iteratively selects its minimal leaf, distributes it on the leaves of its sibling using the Jacobi identity, remarks that the leaves of the new tree are an alphabetic subset of $\B$ and that the new tree is strictly shorter, and continues the process.
    As announced in \cref{sss:warm-up}, this sharper estimate comes from the additional structure of $T_a(b)$ inherited from the fact that $b \in \B$.
    Thus, the key features of the following argument are to identify this additional structure, prove that it is preserved during the iterative process, and eventually use it to tighten the bound from $n!$ down to $\lfloor e (n-1)! \rfloor$. 
    
    \step{Structure of $\rf{a}{b}$}
    Let $b_1, \dotsc, b_n \in \B$ be the (not necessarily distinct) leaves of $\rf{a}{b}$.
    By \cref{Prop:Ta(b)_feuille_b}, recall that $b = \e(\rf{a}{b})$.
    Let us prove that there exist $i \neq j \in \intset{1,n}$ such that $\langle b_i, b_j \rangle$ is a subtree of $\rf{a}{b}$ and $b_i$ is minimal among $b_1, \dotsc, b_n$.
    Let $b_k$ be a minimal leaf of maximal depth. 
    If its sibling is also a leaf $b_{k'}$, then, since $b \in \B$, $b_k < b_{k'}$ and the couple $(i,j) := (k,k')$ is a valid choice.
    If its left sibling is a tree $w$, $\langle w, b_k \rangle$ is a subtree of $\rf{a}{b}$.
    But, since $\e(\langle w, b_k \rangle) \in \B$, $\e(w) < b_k$ so there exists a strictly smaller leaf within $w$, which contradicts the minimality of $b_k$.
    If its right sibling is a tree $w$, $\langle b_k, w \rangle$ is a subtree of $\rf{a}{b}$.
    But, since $\e(\langle b_k, w \rangle) \in \B$, $\lambda(\e(w)) \leq b_k$ so one can find a leaf smaller or equal to $b_k$ within $w$ of greater depth.
    Eventually, up to reindexing, we can assume that $\rf{a}{b}$ contains $\langle b_1, b_2 \rangle$ as a subtree, where $b_1 \leq b_j$ for every $j \in \intset{1,n}$.

    For example, this yields the following structure, with $b_1$ minimal:
    \begin{equation} \label{ex:tab-b1min}
       \rf{a}{b} = \Tree [. [. $b_3$ [. $b_1$ $b_2$ ] ] [. $b_4$ [. [. $b_5$ [. $b_6$ $b_7$ ] ] $b_8$ ] ] ] 
    \end{equation}
    For this tree structure, we knew from the start that $b_1$ or $b_6$ was minimal (and we chose $b_1$ up to re-indexing).
    Indeed, using repeatedly the axioms of a Hall set and the fact that $b \in \B$ yields: $b_1 < b_2$, $b_1 \leq b_3$, $b_6 < b_7$, $b_6 \leq b_5$, $b_5 < (b_5, (b_6, b_7)) < b_8$ and $(b_5, (b_6, b_7)) \leq b_4$.
    
    \step{Paths within $\rf{a}{b}$}
    Let $p \in \intset{0,n-1}$. 
    We say that $\pi = (\pi_1, \dotsc, \pi_p) \in \intset{2,n}^{p}$ is a path within $\rf{a}{b}$ when, for each $i \in \intset{1,p-1}$, the sibling of $b_{\pi_i}$ is not a leaf and contains $b_{\pi_{i+1}}$ as a (deeper) leaf.
    For $\pi = (\pi_1, \dotsc, \pi_p)$ a path of length $p \geq 1$, we define $\pi' = (\pi_1, \dotsc, \pi_{p-1})$ which is a path of length $(p-1)$. 
    
    For the tree given in \eqref{ex:tab-b1min}, examples of such paths are $\emptyset$, $(3)$, $(4, 6)$ or $(4, 5, 7)$.
    However, $(3, 6)$ is not a path for this definition because $b_6$ is not a leaf of the sibling of $b_3$ (which is $\langle b_1, b_2 \rangle$).
    
    There are at most $\binom{n-1}{p}$ such paths of length $p$. 
    Indeed, given a subset of $p$ elements of $\intset{2,n}$, there is at most a single permutation of its elements such that the property ``each next index must be a leaf of the sibling tree'' is satisfied.
    
    \step{Iterative construction}
    We intend to define by induction on $p \in \intset{0,n-1}$, a set $\Pi_p$ of admissible paths of length $p$ and, for each $\pi \in \Pi_p$, an element $a_\pi \in \B$ and a bracket $B_\pi \in \Br(\B)$ such that
    \begin{itemize}
        \item $\mathbb{L}_\B(B_\pi) = \lbag a_\pi \rbag + \lbag b_j ; \enskip j \in \intset{1,n} \setminus \pi \rbag$ and in particular 
        $|B_\pi|_\B = n - p + 1$,
        \item $\mathbb{L}_\B(B_\pi)$ is alphabetic,
        \item the right-sibling of $b_1$ in $B_\pi$ is a leaf,
        \item if $p \geq 1$, the sibling of $a_\pi$ in $B_\pi$ is the sibling of $b_{\pi_p}$ in $\rf{a}{b}$,
        \item $\min \{ b_1, a_\pi \} = \min \mathbb{L}_\B(B_\pi)$.
    \end{itemize}
    For $p = 0$, we let $\Pi_0 := \{ \emptyset \}$, $a_\emptyset := a$ and $B_\emptyset := \langle a, \rf{a}{b} \rangle$ and all properties are satisfied thanks to the remarks made above concerning the structure of $\rf{a}{b}$.
    
    Let $p \in \intset{1,n-1}$. We assume that we have constructed $\Pi_{p-1}$ and for $\pi \in \Pi_{p-1}$, $a_\pi$ and $B_\pi$ with the claimed properties.
    We now define $\Pi_{p}$ as the set of paths $\pi$ of length $p$ such that $\pi' \in \Pi_{p-1}$ and $a_{\pi'} = \min \mathbb{L}_\B(B_{\pi'}) < b_1$. Then $a_{\pi'} < b$ for every $b \in \lbag b_1,\dots,b_n \rbag$.
    For such a path $\pi \in \Pi_p$, we define
    \begin{itemize}
        \item $a_\pi := (a_{\pi'}, b_{\pi_p})$,
        \item $B_{\pi} := J(B_{\pi'}, a_{\pi'}, b_{\pi_p})$.
    \end{itemize}
    Then $a_\pi \in \B$ because $\mathbb{L}_\B(B_{\pi'})$ is alphabetic and $a_{\pi'}< b_{\pi_p}$. The first, fourth and fifth claimed properties above are obviously satisfied. The second one is a consequence of \cref{lem:alphabetic} because $a_{\pi'} \neq b_{\pi_p}$. We also have the third one because the right-sibling of $b_1$ in $B_{\pi}$ is either the right sibling of $b_1$ in $B_{\pi'}$, let us call it $\tilde{b}\in \B$, or $(a_{\pi'},\tilde{b})$ and in this case $\tilde{b} = b_{\pi_p}$ thus 
    $(a_{\pi'},\tilde{b})=(a_{\pi'}, b_{\pi_p})=a_{\pi} \in \B$.
    See \cref{sec:ex-bpi} for an example of the construction of $B_\pi$.
    
	\step{Jacobi identity and size estimate}
    We now prove that, for each $p \in \intset{0,n-1}$ and $\pi \in \Pi_p$, there holds
    \begin{equation} \label{eq:ebpi}
        \| \e(B_\pi) \| \leq (n-p-1)! + \sum_{\bar \pi \in S(\pi)} \| \e(B_{\bar \pi}) \|,
    \end{equation}
    where 
    \begin{equation}
        S(\pi) := \{ \bar \pi \in \Pi_{p+1} ; \enskip \bar \pi' = \pi \},
    \end{equation}
    where, by convention, $\Pi_{n} := \emptyset$.
    Let $p \in \intset{0,n-1}$ and $\pi \in \Pi_p$. 
    Then either $b_1$ or $a_\pi$ is a minimal element of $L_\B(B_\pi)$.
    \begin{itemize}
    	\item Assume that $b_1$ is a minimal element of the leaves of $B_\pi$.
    	Let $\ell \in \B$ be its sibling leaf ($\ell$ is either some $b_k$ for $k \neq 1$ or $a_\pi$).
    	Then $(b_1, \ell) \in \B$. 
    	Let $t := J(B_\pi, b_1, \ell)$.
    	Then $\e(t) = \e(B_\pi)$ and $|t|_\B = |B_\pi|_\B - 1 = n - p$.
    	By \cref{lem:alphabetic}, $L_\B(t)$ is still alphabetic.
    	By \cref{Prop:n!}, $\| \e(t) \| \leq (n-p-1)!$.
    	
    	\item Otherwise, $a_\pi$ is a minimal element of the leaves of $B_\pi$.
    	\begin{itemize}
    		\item Assume that the sibling of $a_\pi$ in $B_\pi$ is a leaf $b_k$.
    		Then $(a_\pi, b_k) \in \B$.
    		Let $t := J(B_\pi, a_\pi, b_k)$.
    		Then $\ev(\e(t)) = \pm \ev(\e(B_\pi))$ in $\mathcal{L}(X)$ and $|t|_\B = |B_\pi|_\B - 1 = n - p$.
            By \cref{lem:alphabetic}, $L_\B(t)$ is still alphabetic.
    	    By \cref{Prop:n!}, $\| \e(t) \| \leq (n-p-1)!$.
    		
    		\item Otherwise, the sibling of $a_\pi$ is a tree $w$. Two cases occur.
    		\begin{itemize}
    			\item First, if $b_1$ is not a leaf of $w$, using the iterated Jacobi identity, we have in $\mathcal{L}(X)$, $\ev(\e(B_\pi)) = \sum_{\bar \pi \in S(\pi)} \pm     			\ev(\e(B_{\bar \pi}))$ so $\| \e(B_\pi) \| \leq \sum_{\bar \pi \in S(\pi)} \|\e(B_{\bar \pi})\|$.

    			\item Second, if $b_1$ is a leaf of $w$, using the iterated Jacobi identity, we have in $\mathcal{L}(X)$, $\ev(\e(B_\pi)) = \pm \ev(\e(t_\pi)) + \sum_{\bar \pi \in S(\pi)} \pm \ev(\e(B_{\bar \pi}))$ so $\| \e(B_\pi) \| \leq \| \e(t_1) \| +\sum_{\bar \pi \in S(\pi)} \|\e(B_{\bar \pi})\|$, where $t_\pi := J(B_\pi, a_\pi, b_1)$ and $|t_\pi|_\B = |B_\pi|_\B-1 = n-p$ so $\| \e(t_\pi) \| \leq (n-p-1)!$
    		\end{itemize} 
    	\end{itemize}
    \end{itemize}
    
    \step{Final estimate}
    Iterating \eqref{eq:ebpi} proves that
	\begin{equation} \label{sequence_ineq}
	    \begin{split}
		\| \e(B_\emptyset) \| & \leq
		\sum_{p=0}^{n-1} \sum_{\pi \in \Pi_p} (n-p-1)! 
		\leq \sum_{p=0}^{n-1} \binom{n-1}{p} (n-p-1)! 
		= \sum_{p=0}^{n-1} \frac{(n-1)!}{p!}
		= \lfloor e (n-1)! \rfloor
	    \end{split}
	\end{equation}
	where the last equality holds for $n \geq 2$ (see e.g.\ \href{https://oeis.org/A000522}{sequence A000522 in OEIS}).
\end{proof}

In the next paragraph, we prove the optimality of estimate \eqref{eq:e-theta}, by forcing both inequalities in \eqref{sequence_ineq} to be equalities: we construct a Hall set and candidates $a, b$ for which any admissible path is in the last configuration (i.e.\ $a_{\pi}$ is a minimal element of $\mathbb{L}_\B(B_\pi)$,
the sibling of $a_{\pi}$ in $B_{\pi}$ is a tree and $b_1$ is a leaf of this tree), and the number of the admissible paths of length $p$ is exactly $\binom{n-1}{p}$.

\subsection{Optimality case}
\label{subsec:optim_2}

Let $n \geq 2$. 
We prove that \eqref{eq:e-theta} is sharp with $\theta_a(b) = n$.
Since we chose $b$ such that $|b|=n$, our construction proves that the length-based estimate \eqref{eq:en1-easy} of the introduction is also sharp.
We start by constructing an order and the associated Hall set, then we exhibit brackets for which the estimate is saturated.

\bigskip

Let $\overline X := \{ X_2, \dotsc, X_n \}$ and $X := \{ X_0, X_1 \} \cup \overline X$.
Let
\begin{equation} \label{eq:Pi}
    \Pi := \bigcup_{0 \leq p \leq n-1}\{\pi = (\pi_1, \dotsc, \pi_p) \in \intset{2,n}^p;\pi_1<\cdots<\pi_p\},
\end{equation}
the set of multi-indexes in increasing order. 
For $\pi \in \Pi$, we define
\begin{itemize}
    \item $a_\pi := (\dotsb((X_0, X_{\pi_1}), X_{\pi_2}), \dotsc X_{\pi_p})$,
    \item $y_\pi := (a_\pi, X_1)$.
\end{itemize}
In particular, $a_\emptyset = X_0$ and $y_\emptyset = (X_0, X_1)$.
By \cref{def:free-subset}, $\Br(\overline X)$ is isomorphic to $\Br_{\overline X}$, the submagma of $\Br(X)$ generated by $\overline X$.
Let $\Br_\Pi$ denote the subset of $\Br(X)$ made of brackets involving exactly one $y_\pi$ and any additional number of indeterminates of $\overline X$.
($\Br_\Pi$ is well-defined, since an equivalent characterization is that it is subset of trees $t \in \Br(X)$ involving $X_0$ and $X_1$ exactly once, within which the sibling of $X_1$ is of the form $a_\pi$ for some $\pi \in \Pi$ and all other leaves are in~$\overline{X}$).

\begin{lemma} \label{p:order-sharp-en1}
    There exists a total Hall order $<$ on $\Br(X)$ such that
    \begin{align}
        & && X_0 < X_1 < X_2 < \dotsb < X_{n}, \label{OO1} \\
        & \forall j \in \intset{2, n-1}, && (\dotsb(X_1,X_n)\dots,X_{j+1}) < X_j \label{OO1bis} \\
        & \forall \pi \in \Pi, &&  a_\pi < X_1, \label{OO2} \\
        & \forall c_1, c_2 \in \Br(\overline X), &&  \max L_{\overline X}(c_1) < \max L_{\overline X}(c_2) \Rightarrow c_1 < c_2, \label{OO3}\\
        & \forall c_1 \in \Br(\overline X), \forall c_2 \in \Br_\Pi, &&  
        c_1 < c_2. \label{OO4}
    \end{align}
\end{lemma}

\begin{proof}
    Our strategy consists in defining a Hall order on a $\lambda$-stable subset $A$ of $\Br(X)$, ensuring the desired conditions, and then extend it to $\Br(X)$ using \cref{Prop:Prolonger_ordre}.
    Let $A := A_1 \cup A_2 \cup A_3 \cup A_4$, where $A_1, A_2, A_3, A_4$ are the pairwise disjoint sets
    \begin{equation} \label{eq:a1a2a3a4}
        \begin{aligned}
         & A_1 := \{ a_\pi ; \enskip \pi \in \Pi \}, \quad
         && A_2 := \{ X_1 \} \cup \{ (\dotsb(X_1,X_n)\dots,X_j) ; j \in \intset{2, n} \}, \\
         & A_3 := \Br(\overline{X}), 
         && A_4 := \Br_\Pi.
        \end{aligned}
    \end{equation}
    Let $\prec$ be any Hall order on $\Br(X)$.
    We will use $\prec$ as an arbitrary order to compare brackets for which we have no particular requirement.
    
    \step{Step 1: We define a total order $<$ on $A$ in the following way}
    \begin{itemize}
        \item (O1) If $c_1 \in A_i$ and $c_2 \in A_j$ with $i \neq j$, $c_1 < c_2$ if and only if $i < j$ (i.e.\ $A_1<A_2<A_3<A_4$)
        \item (O2) If $c_1 \neq c_2 \in A_1$ (resp. $A_2$, $A_4$) then $c_1 < c_2$ if and only if $c_1 \prec c_2$.
        \item (O3) If $c_1 \neq c_2 \in A_3 = \Br(\overline X)$, $c_1 < c_2$ if and only if $\max L_{\overline X}(c_1) < \max L_{\overline X}(c_2)$ or $\max L_{\overline X}(c_1) = \max L_{\overline X}(c_2)$ and $c_1 \prec c_2$, where the maximum is computed for the natural ordering $X_2 < \dotsb < X_n$ of $\overline{X}$.
    \end{itemize}
    One easily checks that the relation $<$ is indeed transitive (due to its lexicographic nature) and a total order on $A$ (thanks to the use of $\prec$ as a last resort to break equalities).
    
    
    \step{Step 2: We check that $A$ is $\lambda$-stable and that $<$ is a Hall order on $A$, i.e.\ for any $c \in A \setminus X$, $\lambda(c) \in A$ and $\lambda(c)<c$}
    \begin{itemize}
        
        \item Let $c \in A_1 \setminus X = A_1 \setminus \{ X_0 \}$. 
        Then $c = a_\pi$ with $\pi \neq \emptyset$ so $\lambda(c) = a_{\pi'} \in A_1$.
        Hence, since $\prec$ is a Hall order, we have $\lambda(c) \prec c$, and by definition of $<$ on $A_1$, $\lambda(c) < c$.
        
        \item Let $c \in A_2 \setminus X = A_2 \setminus \{ X_1 \}$.
        Then $\lambda(c) \in A_2$, so $\lambda(c) \prec c$ so $\lambda(c) < c$.
        
        \item Let $c \in A_3 \setminus X = \Br(\overline X) \setminus \overline{X}$.
        Then $\lambda(c) \in \Br(\overline X)$.
        Moreover $\max L_{\overline X}(\lambda(c)) \leq \max L_{\overline X}(c)$ because each leaf of $\lambda(c)$ is also a leaf of $c$.
        So, by definition of $<$ on $A_3$, either $\max L_{\overline X}(\lambda(c)) < \max L_{\overline X}(c)$ and $\lambda(c) < c$, or $\max L_{\overline X}(\lambda(c)) = \max L_{\overline X}(c)$ and $\lambda(c) < c$ if and only if $\lambda(c) \prec c$, which is the case because $\prec$ is a Hall order on $\Br(X)$.
        
        \item Let $c \in A_4 = \Br_\Pi$.
        First, if $c = y_\pi$ for some $\pi \in \Pi$, then $\lambda(y_\pi) = a_\pi \in A_1$ and $\lambda(c) < c$ by (O1).
        Then, 
        \begin{itemize}
            \item either $\lambda(c) \in \Br(\overline X) = A_3$, in which case $\lambda(c) < c$ by (O1),
            \item or $\lambda(c) \in \Br_\Pi$, 
            in which case $\lambda(c) < c$ if and only if $\lambda(c) \prec c$ (see (O2)), which is the case since $\prec$ is a Hall order on $\Br(X)$.
        \end{itemize} 
        
    \end{itemize}
    
    \step{Step 3: Conclusion} 
    By \cref{Prop:Prolonger_ordre}, $<$ can be extended as a Hall order on $\Br(X)$. 
    Then (O1) ensures \eqref{OO1}, \eqref{OO1bis}, \eqref{OO2} and \eqref{OO4} while (O3) ensures \eqref{OO3}, which concludes the proof.
    The purpose of (O2) is to fill in the unconstrained comparisons to build a total Hall order on $A$.
\end{proof}

\begin{proposition} \label{p:en1-optimal}
    Let $<$ be the order constructed in \cref{p:order-sharp-en1} and $\B \subset \Br(X)$ the associated Hall set. 
    Then $b :=(\dotsb((X_1,X_n),X_{n-1}), \dotsc,X_2)\in \B$ and $\| [X_0,b] \|_\B = \lfloor e (n-1)! \rfloor$.
\end{proposition}

\begin{proof}
    \emph{Step 1: We prove that $b \in \B$.} 
    By \eqref{OO1}, $(X_1,X_n) \in \B$. 
    Using \eqref{OO1bis}, we prove by induction on $j \in \{ n-1 , \dots, 2 \}$ that $(\dotsb,(X_1,X_n) \dotsc,X_j) \in \B$. The case $j=2$ gives $b\in \B$.

    \step{Step 2: We prove that, for any $\pi \in \Pi$, then $a_\pi, y_\pi \in \B$} 
    By definition of $y_\pi$ and \eqref{OO2}, it is sufficient to prove that $a_{\pi} \in \B$, which can be obtained by induction on the length of $\pi$, because $a_{\pi}=(a_{\pi'},X_{\pi_p})$ and $a_{\pi'}<X_{\pi_p}$ by \eqref{OO2} and \eqref{OO1}.

    \bigskip

    From now on, we use the same vocabulary as in the proof of \cref{thm:en1}, with $a \leftarrow X_0$ and $b_j \leftarrow X_j$ for $j \in \intset{1,n}$. 

    \step{Step 3: Any $\pi \in \Pi$ is an admissible path}
    The proof is by induction on its length $p \in \intset{0, n-1}$ and trivially holds for $p=0$. 
    When $p \geq 1$ then $\pi'$ is an admissible path by the induction assumption and $\mathbb{L}_\B(B_{\pi'}) = \lbag a_{\pi'} , X_1 \rbag + \lbag X_j ; j \in \intset{2,n} \setminus \pi \rbag$ thus $a_{\pi'}=\min \mathbb{L}_\B(B_{\pi'}) < X_1$ by \eqref{OO2}.

    \bigskip
    
    For $\pi \in \Pi$, we define $t_\pi, \widetilde{t}_\pi \in \Br(\B)$ by
    \begin{align}
        t_\pi & = \langle \dotsb \langle y_{\pi},X_{j_{n-p-1}} \rangle \dots  , X_{j_1} \rangle, \\
        \widetilde{t}_\pi & = \langle \dotsb \langle \langle X_1 , a_\pi \rangle , X_{j_{n-p-1}} \rangle \dots , X_{j_1} \rangle,
    \end{align}
    where $\{j_1<\dots<j_{n-p-1}\} \cup \{\pi_1<\dots<\pi_p\}$ is a partition of $\intset{2, n}$. 
    Then $\ev (\e ( \widetilde{t}_\pi )) = -\ev ( \e ( t_\pi ))$.

    \step{Step 4: We prove the following equality in $\mathcal{L}(X)$}
    \begin{equation} \label{eq:ab-eitpi}
        [a,b] = \sum_{\pi \in \Pi, \pi_p \leq n-1} \pm \ev(\e(t_\pi))
        + \sum_{\pi \in \Pi, \pi_p = n} \pm \ev(\e(\widetilde{t}_\pi)) = \sum_{\pi \in \Pi} \pm \ev(\e(t_\pi)).
    \end{equation}
    Let $\pi \in \Pi$.
    If $p = |\pi| \geq 1$ and $\pi_p = n$, then $B_\pi$ is of the form $\widetilde{t}_\pi$ so $\e(B_\pi) = \e(\widetilde{t}_\pi)$.
    Otherwise, the sibling of $a_{\pi}$ in $B_{\pi}$ is a non-trivial tree having $X_1$ as a leaf.
    Therefore (see the last step in the proof of \cref{thm:en1}), we have in $\mathcal{L}(X)$, $\ev(\e(B_\pi)) = \pm \ev(\e(t_\pi)) + \sum_{\bar \pi \in S(\pi)} \pm \ev(\e(B_{\bar \pi}))$.
    By iterating this relation starting from $B_\emptyset$, we obtain the conclusion \eqref{eq:ab-eitpi}.

    \step{Step 5: We prove that, for any $\pi \in \Pi$, $\| \e(t_\pi) \|_\B = (n-p-1)!$ and $\supp \ev(\e(t_\pi)) \subset \e(\Tr(\mathbb{L}_\B(t_\pi)))$} 
    Let $\pi \in \Pi$. 
    We already know that the set $\Delta := L_\B(t_\pi) = \{ X_{j_1} , \dots , X_{j_{n-p-1}} , y_{\pi} \}$ is alphabetic. Moreover, $ X_{j_1} < \dots < X_{j_{n-p-1}} < y_\pi$ by \eqref{OO1} and \eqref{OO4}. 
    By \cref{Lem:libre}, $\Delta$ is free because $\Delta \subset \B \cap \Br_X$ and $a_\pi=\lambda(y_\pi)<\min(\Delta)$ (see \eqref{OO1} and \eqref{OO2}).

    Let us check that, for every $c_1,c_2 \in \B \cap \Br_\Delta$ with $L_\Delta(c_1) \cap L_\Delta(c_2)=\emptyset$, then $c_1<c_2$ iff $\max L_\Delta(c_1) < \max L_\Delta(c_2)$. 
    If $c_1, c_2 \in \Br(\{X_{j_1},\dots,X_{j_{n-p-1}}\})$ then \eqref{OO3} gives the conclusion. 
    Otherwise, $y_\pi$ is a leaf of (exactly) one of the two elements $c_1, c_2$ so \eqref{OO4} gives the conclusion.
    
    Hence, the conclusion of this step follows from \cref{Prop:sature_n!} (and \cref{Prop:n!} for the structural part).

    \step{Step 6: We prove that the sets $\B \cap \e(\Tr(\mathbb{L}_\B)(t_\pi))$, for $\pi \in \Pi$, are two by two disjoint} 
    Let $\pi \neq \pi^\flat \in \Pi$, $c \in \B \cap \e (\Tr(\mathbb{L}_\B)(t_\pi))$ and $c^\flat \in \B \cap \e (\Tr(\mathbb{L}_\B)(t_{\pi^\flat}))$. 
    Then $c=\e(B)$ and $c^\flat=\e(B^\flat)$, where $B$ and $B^\flat$ are two different elements of $\Br(\Delta)$ where $\Delta := \{ y_\pi , y_{\pi^\flat} \} \cup \{ X_j ; \enskip j \in \intset{2, n} \}$.
    By \cref{Lem:libre} (with $A = X$), $\Delta$ is free because $\Delta \subset \B \cap \Br_{X}$, $a_\pi = \lambda(y_\pi) < \min(\Delta)$ and $a_{\pi^\flat} = \lambda(y_{\pi^\flat}) < \min (\Delta)$ (see \eqref{OO1} and \eqref{OO2}). 
    Thus $\e:\Br(\Delta) \rightarrow \Br_\Delta$ is injective and $c \neq c^\flat$.

    \bigskip

    We deduce from Steps 4, 5 and 6 that
    \begin{equation}
        \| [X_0,b] \|_\B = \sum_{\pi \in \Pi} \| \e(t_\pi) \|_\B = \sum_{p=0}^{n-1} \binom{n-1}{p} (n-p-1)! = \lfloor e (n-1)! \rfloor,
    \end{equation}
    which concludes the proof.
\end{proof}

The example constructed above is stated in a finite setting.
However, straightforward adaptations lead to the following consequence for infinite sets of indeterminates.

\begin{corollary} \label{cor:x-infini-sature}
    Assume that $X$ is infinite.
    There exists a Hall set $\B \subset \Br(X)$ such that, for every $n \geq 2$, there exists $a < b \in \B$ with $|a| = 1$, $\theta_a(b) = |b| = n$ and $\|[a,b]\|_\B = \lfloor e(n-1)! \rfloor$.
\end{corollary}

\begin{proof}
    We can assume that $X$ contains indeterminates labeled as $X_i$ for $i \in \N$.
    We then perform the following slight modifications to the construction detailed above.
    First, $\Pi$ defined in \eqref{eq:Pi} is changed to $\Pi := \bigcup_{p \in \N} \{\pi = (\pi_1, \dotsc, \pi_p) ; 2 \leq \pi_1<\cdots<\pi_p\}$.
    Then, we prove that there exists a total Hall order on $\Br(X)$ satisfying the all conditions of \cref{p:order-sharp-en1} (where \eqref{OO1} and \eqref{OO1bis} are to be understood as holding for every $n \geq 2$).
    This solely requires to modify the set $A_2$ of \eqref{eq:a1a2a3a4} into $A_2 := \{ X_1 \} \cup \{ (\dotsb(X_1,X_n)\dots,X_j) ; j \in \intset{2, n} ; n \geq 2 \}$.
    Then the result follows by considering the Hall set associated with this order and the brackets of \cref{p:en1-optimal}.
\end{proof}

\subsection{Recursion depth of the decomposition algorithm}
\label{ss:rec-theta}

In this paragraph, we prove that $\theta_a(b)$ is strictly decreasing along the historical recursive decomposition algorithm described in \cref{sec:rewriting}, within any Hall set $\B \subset \Br(X)$.

Hence, from a computational point of view, although this algorithm is classically seen as an induction on the couple $(|a|+|b|, \min\{a,b\})$, this shows that the size of the associated call stack (number of simultaneous nested executions) is bounded by $\theta_a(b) \leq |b|$. 
This is not straightforward, since it is not guaranteed that $|b|$ or $|a|+|b|$ decrease along the iterations.

From a theoretical point of view, we will use this property to prove results relying on the recursive decomposition algorithm by induction on $\theta_a(b)$ for some particular Hall sets.

We start with a technical lemma which states a kind of stability of $\theta_a(b)$ with respect to $b$ within a certain class (monotony with respect to $a$ is covered in \cref{p:theta-dec}).

\begin{lemma} \label{lem:theta-dec2/support-absurde}
    Let $a < b \in \B$ and $c \in \supp [a,b]$.
    Then $\theta_a(c) = \theta_a(b)$ and $\lambda(c) \neq b$.
\end{lemma}

\begin{proof}
    Let $a < b \in \B$. 
    We start by proving two elementary claims, which lead to the result.
    
    \step{Step 1: For each $t \in \Br(L_\B(T_a(b)))$ with $\e(t) \in \B$, one has $\theta_a(\e(t)) = |t|_\B$}
    
    We proceed by induction on $|t|_\B$.
    For each such $t$, by \cref{Lem:BRA}, $a < \e(t)$.
    When $|t|_\B = 1$, $t$ is of the form $b_i$ for some $b_i \in L_\B(T_a(b))$ so $(a, \e(t)) = (a, b_i) \in \B$ hence $\theta_a(\e(t)) = 1$.
    When $|t|_\B > 1$, $\lambda(t) \in \Br(L_\B(T_a(b)))$ so $a < \lambda(\e(t))$.
    Hence $\theta_a(\e(t)) = \theta_a(\e(\lambda(t))) + \theta_a(\e(\mu(t)))$ and the conclusion follows by induction.
    
    \step{Step 2: For each $t \in \Br(\{a\} \cup L_\B(T_a(b)))$ with $\e(t) \in \B$, $|t|_\B \geq 2$ involving $a$ exactly once, $\theta_a(\e(t)) = |t|_\B - 1$}
    
    We also proceed by induction on $|t|_\B$.
    When $|t|_\B = 2$, $t$ is of the form $\langle a, b_i \rangle$ for some $b_i \in L_\B(\rf{a}{b})$, so $\e(t) = (a, b_i)$ and $\lambda(\e(t)) = a$, hence $\theta_a(\e(t)) = 1$.
    When $|t|_\B > 2$, one cannot have $\lambda(t) = a$. 
    Indeed, if $\lambda(t) = a$, since $|\mu(t)|_\B \geq 2$, $\lambda(\mu(t)) \in \Br(L_\B(\rf{a}{b}))$ so $a < \e(\lambda(\mu(t)))$, which contradicts the fact that $\e(t) \in \B$.
    Hence, $\theta_a(\e(t)) = \theta_a(\e(\lambda(t))) + \theta_a(\e(\mu(t)))$, where $\lambda(t)$ and $\mu(t)$ satisfy respectively the hypotheses of Step 1 and Step 2 (or the converse).
    Thus the conclusion follows by induction.

    \step{Step 3: Conclusion}
    Let $c \in \supp [a,b]$. 
    By \cref{thm:en1}, $c \in \e(\Tr(\mathbb{A}))$ where $\mathbb{A} = \lbag a \rbag + \mathbb{L}_\B(T_a(b))$.
    Hence, by the previous argument, $\theta_a(c) = (\theta_a(b)+1)-1 = \theta_a(b)$.
    By contradiction, if $\lambda(c) = b$, then one would have $\theta_a(c) = \theta_a(\lambda(c)) + \theta_a(\mu(c)) > \theta_a(b)$. 
\end{proof}

\begin{proposition} \label{p:rec-theta}
    Let $\B \subset \Br(X)$ be a Hall set.
    Let $\texttt{Rewrite}(\cdot, \cdot)$ denote the algorithm described in \cref{sec:rewriting} which associates to each couple $a < b \in \B$ the decomposition of $[a,b]$ on $\B$.
    For every $a < b \in \B$, calling $\texttt{Rewrite}(a,b)$ requires a call stack of size at most $\theta_a(b)$.
\end{proposition}

\begin{proof}
    We proceed by induction on $n := \theta_a(b)$.
    First, when $n = \theta_a(b) = 1$, $(a,b) \in \B$ so there is no nested call.
    Let $n \geq 2$ and $a < b \in \B$ such that $\theta_a(b) = n$.
    We need to check that $\texttt{Rewrite}(a,b)$ only calls $\texttt{Rewrite}(\tilde{a},\tilde{b})$ with arguments $\tilde{a} < \tilde{b}$ such that $\theta_{\tilde{a}}(\tilde{b}) < \theta_a(b)$. 
    Since $\theta_a(b) > 1$, $b \notin X$ and $a < \lambda(b) < \mu(b)$. 
    Jacobi's identity yields
    \begin{equation}
        [a, b]  = [[a, \lambda(b)], \mu(b)] + [\lambda(b), [a, \mu(b)]].
    \end{equation}
    We study both terms separately, but with the same method.
    \begin{itemize}
        \item \emph{First term}. Calling $\texttt{Rewrite}(a,\lambda(b))$, we obtain $[a,\lambda(b)] = \sum \alpha_d \ev(d)$.
        Moreover, one has $\theta_a(\lambda(b)) < \theta_a(b)$ (by definition, because $\theta_a(\mu(b)) \geq 1$).
        For each $d$ in this decomposition, 
        \begin{itemize}
            \item If $d < \mu(b)$, we call $\texttt{Rewrite}(d,\mu(b))$.
            By \eqref{eq:a<c'}, $a \leq \lambda(d) < d$.
            Hence, by \cref{p:theta-dec}, $\theta_d(\mu(b)) \leq \theta_a(\mu(b)) < \theta_a(b)$.
            
            \item If $\mu(b) < d$, we call  $\texttt{Rewrite}(\mu(b),d)$.
            By \cref{p:theta-dec}, $\theta_{\mu(b)}(d) \leq \theta_a(d)$.
            Moreover, since $d \in \supp [a,\lambda(b)]$, by \cref{lem:theta-dec2/support-absurde}, $\theta_a(d) \leq \theta_a(\lambda(b))$.
            Hence $\theta_{\mu(b)}(d) < \theta_a(b)$.
        \end{itemize}
        
        \item \emph{Second term}. Calling $\texttt{Rewrite}(a,\mu(b))$, we obtain $[a,\mu(b)] = \sum \alpha_d \ev(d)$.
        Moreover, one has $\theta_a(\mu(b)) < \theta_a(b)$ (by definition, because $\theta_a(\lambda(b)) \geq 1$).
        For each $d$ in this decomposition, 
        \begin{itemize}
            \item If $d < \lambda(b)$, we call $\texttt{Rewrite}(d,\lambda(b))$.
            By \eqref{eq:a<c'}, $a \leq \lambda(d) < d$.
            Hence, by \cref{p:theta-dec}, $\theta_d(\lambda(b)) \leq \theta_a(\lambda(b)) < \theta_a(b)$.
            
            \item If $\lambda(b) < d$, we call  $\texttt{Rewrite}(\lambda(b),d)$.
            By \cref{p:theta-dec}, $\theta_{\lambda(b)}(d) \leq \theta_a(d)$.
            Moreover, since $d \in \supp [a,\mu(b)]$, by \cref{lem:theta-dec2/support-absurde}, $\theta_a(d) \leq \theta_a(\mu(b))$.
            Hence $\theta_{\lambda(b)}(d) < \theta_a(b)$.
        \end{itemize}
    \end{itemize}
    This concludes the proof.
\end{proof}

\begin{remark}
    With this approach, one could also estimate the size of the structure constants. 
    By induction on $n := \theta_a(b) \geq 1$, one constructs a sequence $C(n)$ such that $\|[a,b]\|_\B \leq C(\theta_a(b))$.
    If the worst cases happen simultaneously, the previous proof shows that one can take $C(n) := 2 C(n-1)^2$.
    Hence $C(n) := 2^{2^{n-1}-1}$ is an admissible bound.
    This yields an estimate of the form
    \begin{equation}
        \| [a,b] \|_\B \leq 2^{2^{\theta_a(b)-1}-1},
    \end{equation}
    which, since $\theta_a(b) \leq |b|$, is of course much better than \eqref{eq:2Xn2}, but also very far from the previous optimal estimate \eqref{eq:e-theta}, which stems from a tighter tracking of the unfavorable cases. 
\end{remark}

\section{Systematic geometric lower bounds}
\label{sec:geom-lower}

The main goal of this section is to prove \cref{thm:lower-easy} which yields a geometric lower bound for the structure constants relative to any Hall set.
In \cref{subsec:prel}, we gather elementary results concerning systematic basis elements used in the next sections. 
We then prove \cref{thm:lower-easy} in \cref{subsec:gen_case} when $|X| \geq 3$ (which is straightforward and easily checked to be optimal) and in \cref{subsec:2letter_case} when $|X| = 2$ (which relies on more involved computations).
The optimality of the lower bound in this latter case is discussed in \cref{s:fibo}.
Eventually, we prove in \cref{sec:lower-theta} a systematic $\theta$-based geometric lower-bound.

To measure the worst-case growth of the structure constants of $\mathcal{L}(X)$ relative to a given Hall set $\B \subset \Br(X)$, we will use the following quantity for $n \geq 2$:
\begin{equation} \label{eq:beta-n}
    \beta_n(\B):=\sup \left\{  \|[a,b]\|_\B ; \enskip a<b\in\B , |a|+|b| = n \right\} .
\end{equation}

\subsection{Some linearly independent elements of any Hall set}
\label{subsec:prel}

We identify some elements which belong to a Hall set without knowing the underlying order.

\begin{proposition}\label{prop:some_elements_hall}
    Let $\B \subset \Br(X)$ be a Hall set.
    \begin{itemize}
    \item If $\{X_0,X_1\} \subset X$, $b_0 \in \B$ with $b_0 < X_1$ and $r:=r(b_0,X_1)$ is defined by \cref{lem:some_elements_hall}, then the following elements of $\Br(X)$ belong to $\B$: 
    \begin{enumerate}
        \item $b_k:=\dad_{X_1}^k(b_0)$ for $1 \leq k \leq r$, 
        \item $\ad_{X_1}^n(b_r)$ for $n \in \N$, 
        \item $w_i := (b_i,b_{i+1})=\ad_{b_i}^2(X_1)$ for $0 \leq i \leq r-1$, 
        \item $(\ad_{X_1}^n(b_r),\ad_{X_1}^m(b_r))$ for $m,n \in \N$ such that $\ad_{X_1}^n(b_r) < \ad_{X_1}^m(b_r)$,
        \item for $0 \leq i \leq r-1$, if $w_i < X_1$, $\dad_{X_1}^n(w_i)$  for $n \leq q_i := r(w_i, X_1)$ and $\ad_{X_1}^{n-q_i} \dad_{X_1}^{q_i} (w_i)$ for $n > q_i$, otherwise $\ad_{X_1}^n (w_i)$ for all $n \geq 0$.
    \end{enumerate}
    \item If $\{X_0,X_1,X_2\} \subset X$ with $X_0<X_1<X_2$, denoting by $b_i=\dad_{X_1}^k(X_0)$ and $r:=r(X_0,X_1)$, the following elements of $\Br(X)$ lie in $\B$: 
    \begin{enumerate} 
        \item $\ad_{X_1}^n(X_2)$ for $n \in \N$,
        \item $\ad_{X_1}^n(b_i,X_2)$ for $0 \leq i \leq r-1$, $n\in\N$, 
        \item either $(\ad_{X_1}^n(b_r),\ad_{X_1}^m(X_2))$ or $(\ad_{X_1}^m(X_2),\ad_{X_1}^n(b_r))$ when $m,n \in \N$.
    \end{enumerate}
\end{itemize}
\end{proposition}

\begin{proof}
    Let us prove both parts.
    \begin{itemize}
        \item 
        Items $(1)$ and $(2)$ are \cref{lem:some_elements_hall}. 
        Item $(3)$ holds because, for $0 \leq i \leq r-1$, $b_i<X_1$. 
    
        For item $(4)$, let $m,n \in \N$ such that $\ad_{X_1}^m(b_r)<\ad_{X_1}^n(b_r)$. 
        If $n \geq 1$, then $\lambda(\ad_{X_1}^n(b_r)) = X_1  \leq \ad_{X_1}^m(b_r)$, either by Hall order property when $m \geq 1$, or because $b_r>X_1$ when $m=0$, thus $(\ad_{X_1}^n(b_r),\ad_{X_1}^m(b_r)) \in \B$.
        If $n=0$, then $m \geq 1$ thus $\lambda(\ad_{X_1}^n(b_r))=\lambda(b_r)=b_{r-1}\leq X_1 \leq \ad_{X_1}^m(b_r)$ which gives the conclusion.
    
        For item $(5)$, if $w_i < X_1$, it suffices to apply \cref{lem:some_elements_hall}.
        Of $X_1 < w_i$, since $\lambda(w_i) = b_i < X_1$ because $i < r$, so $(X_1, w_i) \in \B$ and the same holds for $\ad_{X_1}^n(w_i)$ for all $n \geq 0$.
    
        \item 
        Since $X_1<X_2$, item $(1)$ is clear. 
        Concerning item (2), for $0\leq i \leq r-1$, we have $b_i<X_1<X_2$ thus $(b_i,X_2) \in \B$ and $(X_1,(b_i,X_2))\in\B$, which gives the conclusion. For item (3), let $m,n\in\N$. We assume $\ad_{X_1}^n(b_r) < \ad_{X_1}^m(X_2)$. If $m=0$ then $(\ad_{X_1}^n(b_r),\ad_{X_1}^m(X_2))=(\ad_{X_1}^n(b_r),X_2) \in \B$.
        If $m\geq 1$, then $\lambda(\ad_{X_1}^m(X_2))=X_1 \leq \ad_{X_1}^n(b_r)$, by Hall order property when $n\geq 1$, because $b_r>X_1$ when $n=0$, thus $(\ad_{X_1}^n(b_r),\ad_{X_1}^m(X_2))\in\B$.
        The same reasoning gives the conclusion when $\ad_{X_1}^m(X_2) < \ad_{X_1}^n(b_r)$.
        \qedhere
    \end{itemize}
\end{proof}

\subsection{The general case}
\label{subsec:gen_case}

We prove an optimal systematic lower bound when $|X| \geq 3$.

\begin{theorem} \label{thm:X3-2^n}
    Let $\B \subset \Br(X)$ be a Hall set.
    Assume that $|X| \geq 3$ and let $X_0 < X_1 < X_2 \in X$.
    For every $n \in \N$, $\|[X_0,\ad_{X_1}^n(X_2)] \|_\B \geq 2^n$ with equality if $r(X_0, X_1) \in \{ 1, + \infty \}$.
    Thus $\beta_{n+2}(\B) \geq 2^n$.
\end{theorem}

\begin{proof}
    Let $r:=r(X_0,X_1)$ be defined in \cref{lem:some_elements_hall}. 
    By \cref{p:ad-2} (applied to the derivation $\ad_{X_1}$ on $\mathcal{L}(X)$, $\nu \leftarrow n$, $b_1 \leftarrow X_0$ and $b_2 \leftarrow X_2$),
    \begin{equation} \label{eq:X0X1nX2-Jac}
        [X_0,\ad_{X_1}^n(X_2)]
        = \sum_{k=0}^{n} (-1)^k \binom{n}{k} \ad_{X_1}^{n-k} [\ad_{X_1}^k(X_0),X_2].
    \end{equation}
    
    \step{Case $n < r$}
    Then \eqref{eq:X0X1nX2-Jac} implies that
    \begin{equation}
        \begin{split}
            [X_0,\ad_{X_1}^n(X_2)]
            = \sum_{k=0}^{n} \binom{n}{k}  \ad_{X_1}^{n-k} [\dad_{X_1}^k(X_0),X_2].
        \end{split}
    \end{equation}
    By item (2) of the second point of \cref{prop:some_elements_hall}, the elements of the right-hand side are (evaluations of) distinct elements of $\B$. 
    Thus, $\|[X_0,\ad_{X_1}^n(X_2)]\|_\B=\sum_{k=0}^{n}  \binom{n}{k} = 2^n$.
    In particular, when $r=+\infty$ this equality holds for every $n\in\N$.
    
    \step{Case $r$ finite and $n \geq r$} 
    Starting from \eqref{eq:X0X1nX2-Jac}, using successively Jacobi's formula, the index change $s=k+j$ and a re-arrangement, we obtain
    \begin{equation} \label{eq:X0X1nX2-JacR}
        \begin{split}
        [X_0,\ad_{X_1}^n(X_2)] & - \sum_{k=0}^{r-1} \binom{n}{k} \ad_{X_1}^{n-k} [\dad_{X_1}^k(X_0),X_2]
        \\ & = \sum_{k=r}^{n} (-1)^k \binom{n}{k} \sum_{j=0}^{n-k} \binom{n-k}{j} [\ad_{X_1}^{k+j}(X_0),\ad_{X_1}^{n-k-j}(X_2)]
        \\ & = \sum_{s=r}^{n} \alpha_s     [\ad_{X_1}^{s-r} \dad_{X_1}^r (X_0),\ad_{X_1}^{n-s}(X_2)],
        \end{split}
    \end{equation}
    where, using \cite[(1.5)]{zbMATH03380631} for the last equality,
    \begin{equation} \label{alphas:valeur}
        \alpha_s 
        := (-1)^r \sum_{k=r}^{s} (-1)^k \binom{n}{k} \binom{n-k}{s-k} 
        = (-1)^r \binom{n}{s} \sum_{k=r}^{s} (-1)^k \binom{s}{k} 
        = \binom{n}{s}\binom{s-1}{r-1}.
    \end{equation}
    By item (3) of the second point of \cref{prop:some_elements_hall}, the elements in the right-hand side of \eqref{eq:X0X1nX2-JacR} are, up to a sign which depends on $\B$, (evaluations of) distinct elements of $\B$.
    Thus $\| [X_0,\ad_{X_1}^n(X_2)] \|_\B = \sum_{k=0}^{r-1} \binom{n}{k} +\sum_{s=r}^{n} \alpha_s \geq 2^n$, because $\alpha_s\geq\binom{n}{s}$, with equalities when $r = 1$.
\end{proof}

\begin{remark} 
    This lower bound is optimal, in the sense that there exist Hall sets $\B \subset \Br(X)$ such that $\beta_{n+2}(\B) \leq 2^n$.
    In particular, this is the case of the length-compatible and Lyndon bases for which \cref{thm:length-easy} and \cref{thm:lyndon-easy} prove that $\| [a, b] \|_\B \leq 2^{|b|-1} \leq 2^{|a|+|b|-2}$.
    Hence, for these bases $\beta_{n+2}(\B) = 2^n$ when $|X| \geq 3$.
\end{remark}

\subsection{The two-indeterminates case}
\label{subsec:2letter_case}

We prove a systematic lower bound when $|X| = 2$ in \cref{thm:LB2}.
The construction requires more complex brackets, for which we start by giving elementary formulas.

\begin{proposition}\label{prop:geometric_general_formulas}
    Let $a,b \in \mathcal{L}(X)$. 
    For $n \in \N$, let $b_n := \ad_{a}^n(b)$. 
    For $i,j \in \N$ with $i \leq j$, let $w_{i,j} := [b_i, b_j]$ and $w_i := w_{i,i+1} = [b_i, b_{i+1}]$. 
    The following relations hold for all $i \in \N$ and $n \in \N^*$ 
    \begin{align} \label{wi,i+n}
    w_{i,i+n} & = \sum_{p=0}^{\lfloor \frac{n-1}{2}\rfloor} (-1)^p\binom{n-1-p}{p} \ad_a^{n-1-2p}(w_{i+p}), \\
     \label{adan(wi)}
    \ad_a^{n}(w_i) & = \sum_{p=0}^{\lfloor\frac{n-1}{2}\rfloor} \left[ \binom{n}{p} - \binom{n}{p-1}\right]w_{i+p,i+1+n-p}.
    \end{align}
    The following relation holds for all $n \in \N$ and $r \in \intset{1, \lfloor\frac{n-1}{2}\rfloor}$,
    \begin{multline} \label{eq:w0n}
        w_{0,n} = \sum_{p=0}^{r-1} (-1)^p \binom{n-1-p}{p} \ad_a^{n-1-2p}(w_{p}) \\ + \sum_{s=r}^{\lfloor\frac{n-1}{2}\rfloor}\left[\sum_{p=r}^s (-1)^p\binom{n-1-p}{p}\left( \binom{n-1-2p}{s-p} - \binom{n-1-2p}{s-p-1}\right) \right]w_{s,n-s}.
    \end{multline}
\end{proposition}
\begin{proof}
    The validity for every $i \in \N$ of \eqref{wi,i+n} and \eqref{adan(wi)} is proved by induction on $n\in\N^*$. 
    For $n=1$, they hold because $w_{i,i+1}=w_i$.  
    Let us prove the result for $(n+1)$ knowing it for $n$. 

    For \eqref{wi,i+n}, we use the Jacobi relation $w_{i,i+n+1}=[b_i,\ad_{a}(b_{i+n})]=\ad_{a}(w_{i,i+n})-w_{i+1,i+n}$, inject the induction assumption, re-organize terms and conclude with $\binom{n-1-p}{p}+\binom{n-1-p}{p-1}=\binom{n-p}{p}$.

    For \eqref{adan(wi)}, we start from the induction assumption, use the Jacobi relation
    \begin{equation}
        \ad_{a}(w_{i+p,i+1+n-p})=w_{i+p+1,i+1+n-p}+w_{i+p,i+2+n-p},
    \end{equation}
    reorganize terms and conclude with $\binom{n}{p}-\binom{n}{p-1}+\binom{n}{p-1}-\binom{n}{p-2}=\binom{n+1}{p}-\binom{n+1}{p-1}$.

    Finally, to prove \eqref{eq:w0n}, we start from \eqref{wi,i+n} with $i=0$, then split the sum in two sums over $p \in \intset{0, r-1}$ and $p \in \intset{r,  \lfloor \frac{n-1}{2} \rfloor}$. 
    In each term of the second sum, we incorporate \eqref{adan(wi)} and re-organize terms.
\end{proof}

We now prove the main theorem of this paragraph.
We include here an upper bound, see \eqref{eq:X0bn-r1-2n} below, which is used in the proof of \cref{p:length-sharp}.
We prove this upper bound here because it concerns the same bracket as the one which we use to prove systematic lower bounds.

\begin{theorem} \label{thm:LB2}
    Let $X=\{X_0,X_1\}$, $\B \subset \Br(X)$ be a Hall set such that $X_0<X_1$, $r:=r(X_0,X_1)$ defined by \cref{lem:some_elements_hall}, $b_n:=\dad_{X_1}^n(X_0)$ for $n \in \intset{0,r}$ and $b_n:=\ad_{X_1}^{n-r}(\dad_{X_1}^r(X_0))$ for $n>r$.
    Then, for every $n \in \N^*$,
    \begin{align}
        \label{eq:X0bn-Fn}
        & \| [X_0, b_{n}] \|_\B = F_n && \text{if } n<2r+1, \\
        \label{eq:X0bn-FnCr}
        & \| [X_0, b_{n}] \|_\B \geq \max\{ F_n, C(r) n^{r-\frac{5}{2}} 2^{n} \} && \text{if } n \geq 2r+1, \\
        \label{eq:X0bn-r1-2n}
        & \| [X_0, b_{n}] \|_\B \leq 2^{n-2} && \text{if } r = 1 \text{ and } n \geq 2.
    \end{align}
    where $(F_\nu)_{\nu \in \N}$ denote the $0$-based Fibonacci numbers.
    In particular $\beta_{n+2}(\B) \geq F_n$.
\end{theorem}

\begin{proof}
    We use the notations of \cref{prop:geometric_general_formulas} with $b_0 = X_0$.
    If $n<2r+1$, then by item (5) of the first point of \cref{prop:some_elements_hall}, the right-hand side elements of \eqref{wi,i+n} with $i=0$ are, up to sign, evaluations of distinct elements of $\B$, so we have
    \begin{equation}
        \|[b_0,b_n]\|_\B = \sum_{p=0}^{\lfloor \frac{n-1}{2} \rfloor} \binom{n-1-p}{p} = F_n,
    \end{equation}
    which proves \eqref{eq:X0bn-Fn}.
    Now, we assume that $r$ is finite and $n \geq 2r+1$. 
    By items (4) and (5) of the first point of \cref{prop:some_elements_hall}, the right-hand side elements of \eqref{eq:w0n} are, up to sign, evaluations of distinct elements of $\B$, so we have
    \begin{equation} \label{eq:norm-b0-b0n}
       \| [b_0,b_n] \|_\B = \sum_{p=0}^{r-1} \binom{n-1-p}{p}
        + \sum_{s=r}^{\lfloor \frac{n-1}{2} \rfloor} A^r_s(n)
    \end{equation}
    where
    \begin{equation} \label{eq:def-arsn}
    	A^r_s(n) := \left| 
    	\sum_{p=r}^s (-1)^p \binom{n-1-p}{p} \left(
    	\binom{n-1-2p}{s-p} - \binom{n-1-2p}{s-p-1}
    	\right)
    	\right|.
    \end{equation}
    Then \eqref{eq:norm-b0-b0n} and \cref{p:arsn-fibo} prove that
    \begin{equation}
        \|[b_0,b_n]\|_\B \geq \sum_{p=0}^{\lfloor \frac{n-1}{2} \rfloor} \binom{n-1-p}{p} = F_n.
    \end{equation}
    Moreover \eqref{eq:norm-b0-b0n} and \cref{p:arsn-2n} prove that
    \begin{equation}
        \|[b_0,b_n]\|_\B \geq A^r_{\lfloor \frac{n-1}{2} \rfloor}(n)
        \geq C_r n^{r - \frac 5 2}2^n.
    \end{equation}
    which concludes the proof of \eqref{eq:X0bn-FnCr}.

    \medskip
    
    Eventually, the last estimate \eqref{eq:X0bn-r1-2n} follows from \eqref{eq:norm-b0-b0n} with $r = 1$ and \cref{p:sum-a1sn}.
\end{proof}

\begin{remark}
    This estimate is optimal, in the sense that there exists a Hall set $\B \subset \Br(X)$ such that $\beta_{n+2}(\B) = F_n$ for all $n \geq 0$.
    We construct such a Hall set in \cref{s:fibo}.
    Up to our knowledge, this Hall set has not been studied before. 
    When $|X| = 2$, neither the length-compatible nor the Lyndon basis satisfy $\beta_{n+2}(\B) = F_n$ (see \cref{s:fibo-motivation} for more insight on this topic, and \cref{rk:length-beta} and \cref{rk:lyndon-beta} for the computation of the exact values in the case of length-compatible and Lyndon bases).
\end{remark}

\subsection{A \texorpdfstring{$\theta$}{theta}-based lower bound}
\label{sec:lower-theta}

The previous paragraphs prove that the length-based lower bound is geometric in any Hall set. 
We prove that this also holds for the $\theta$-based lower bound, with a geometric ratio of 2, even in the two-indeterminates case.

\begin{proposition} \label{prop:lower-theta}
    Let $\B \subset \Br(X)$ be a Hall set with $|X| \geq 2$. 
    For every $\theta \in \N^*$ there exist $a < b \in \B$ such that $\theta_a(b) = \theta$ and $\|[a,b]\|_\B \geq 2^{\theta-1}$.
\end{proposition}

\begin{proof}
    Let $X_0, X_1 \in X$ with $X_0<X_1$. 
    Now let $\mathcal{X} := \{x_0,x_1,x_2\} = \{(X_0,X_1), (X_0,(X_0,X_1)), X_1\}$, numbered such that $x_0<x_1<x_2$. 
    The set $\mathcal{X}$ is a free alphabetic subset of $\B$. 
    Let $\theta \in \N^*$, and let $b_\theta := \ad_{x_1}^{\theta-1}(x_2)$. Since $x_0<x_1<x_2$, $x_0< b_\theta \in \B$, and a straightforward induction on $\theta$ shows that $\theta_{x_0}(b_\theta) = \theta$. 
    By \cref{Prop:free-lie-alphabetic}, the canonical map $\mathcal{L}(\mathcal{X}) \to \mathcal{L}(X)$ is an isometry with respect to the norms relative to $\B \cap \Br_\mathcal{X}$ and $\B$ respectively. 
    The fact that $\|[x_0,b_\theta]\|_\B \geq 2^{\theta-1}$ is then a direct consequence of \cref{thm:X3-2^n} applied in $\mathcal{L}(\mathcal{X})$.
\end{proof}

\begin{remark}
    As we will see in forthcoming sections, this lower bound of $2^{\theta-1}$ for the $\theta$-based estimate is attained in the case of the length-compatible Hall sets (see \cref{Thm:lgth-comp-main}) and in the case of the Lyndon basis (see \cref{thm:lyndon-bound}). 
    However, none of these Hall sets attains the lower bound for the length-based estimate in the case $|X|=2$.

    When $|X|=2$, we will construct in \cref{s:fibo} another Hall set which does not attain the lower bound for the $\theta$-based estimate, but attains the lower bound for the length-based estimate.
\end{remark}

\section{Length-compatible Hall sets}
\label{sec:length}

The main goal of this section is to prove \cref{thm:length-easy}, which yields a sharp bound of structure constants relative to length-compatible Hall sets.

\begin{definition}[Length-compatible Hall sets]
    We say that $\B$ is a \emph{length-compatible Hall set} when its order is length-compatible, i.e.\ satisfies 
    \begin{equation} \label{eq:order-length}
        \forall a, b \in \B, \quad a < b \Rightarrow |a| \leq |b|.
    \end{equation}
\end{definition}

Since \eqref{eq:order-length} does not impose the order between elements of the same length, which then determines whether some brackets belong or not the Hall set, there are many different length-compatible Hall sets, even for a fixed $X$.
The lower and upper bounds we prove in this section are valid for all length-compatible Hall sets without distinction, highlighting that, with respect to our focus on worst-case growth of structure constants, they behave similarly.

We start with a short terminology discussion in \cref{sec:length-historic} and preliminary results concerning structure constants for such bases in \cref{subsec:Preliminary_lengthcomp}.
We then prove an enhanced version of \cref{thm:length-easy} in \cref{subsec:Proof_lengthcomp}.
Eventually, we discuss the optimality of length-focused estimates in \cref{subsec:optim_lengthcomp}.

\subsection{Naming length-compatible Hall sets}
\label{sec:length-historic}

Historically, Hall sets were introduced by Marshall Hall in \cite{zbMATH03059664}, from ideas implicit in the work~\cite{zbMATH03010343} of Philip Hall concerning group theory.
In Marshall Hall's definition, the third condition in \cref{Def2:Laz} was replaced by the stronger condition \eqref{eq:order-length}.
Since one always has $|a| < |a| + |b|$, condition \eqref{eq:order-length} implies that $a < (a,b)$ so it is indeed stronger.

\begin{remark}
    Some authors refer to such length-compatible sets as ``\emph{Philip Hall bases}'' (e.g.\ \cite{zbMATH04006080}) or ``\emph{Philip Hall families}'' (e.g.\ \cite{zbMATH00426005, Serre1992}) or sometimes even simply ``\emph{Hall bases}'' (e.g.\ \cite{kawski2000calculating}, in which the generalized ones are called ``\emph{Hall-Viennot bases}'') or ``\emph{Hall sets}'' (e.g.\ Bourbaki's choice in \cite{zbMATH00194100}, anterior to Viennot's work).
    
    We follow Viennot's definition (which seems to be the modern convention, also used in \cite{zbMATH00035953, zbMATH00417855}) and use the name ``\emph{Hall set}'' to refer to \cref{Def2:Laz}.
    To avoid confusion, we introduce the non-standard but more descriptive name ``\emph{length-compatible Hall sets}'' instead of relying on the subtle distinction between ``\emph{Hall sets}'' and ``\emph{Philip Hall sets}'' (moreover, historically speaking, both Philip's and Marshall's definitions involved length-compatibility).
\end{remark}

Both in Philip Hall's description of his ``\emph{commutator collecting process}'' and in Marshall Hall's definition of his ``\emph{standard monomials}'', it is explicit that the order between elements of the same length can be chosen arbitrarily.

Nevertheless, it seems that the mathematical literature and the software community has somehow progressively defined the most natural ``\emph{Philip Hall basis of $\mathcal{L}(X)$}'' as the Hall set on $X$ whose order is given by the lexicographic order on the triple $(|b|, \lambda(b), \mu(b))$ (which is indeed well-defined by induction on the length).
This order appears explicitly in \cite{meier1952note} associated with the name ``\emph{basic commutator}'', in  \cite{simpson1966hall} with the name ``\emph{natural ordering}'' or in \cite{duleba1997checking} with the name ``\emph{optimal Philip Hall basis}'' (this work indeed claims that this particular choice allows to write an algorithm generating the basis up to some fixed length while minimizing the number of required bracket comparisons).
This precise choice of length-compatible order also appears implicitly in many Lie algebraic packages, where it is linked with implementation choices (see e.g.\ \texttt{Hall.\_generate\_hall\_set} in SageMath \cite{sagemath} or \texttt{hall\_basis.\_growup} in CoRoPa's LibAlgebra \cite{libalgebra}).
Some packages however make other choices (see e.g.\ \texttt{phb} in LTP \cite{zbMATH02237869}).

\subsection{Preliminary results} \label{subsec:Preliminary_lengthcomp}

In this paragraph, $\B \subset \Br(X)$ denotes a length-compatible Hall set.

\begin{lemma}\label{Lem:length-comp-x2}
    Let $a < b \in \B$ such that $|b|<2|a|$. 
    Then $(a,b) \in \B$.
\end{lemma}

\begin{proof}
    If $b \in X$, then $(a,b) \in \B$. 
    Otherwise, we can write $b = (a_1, b_1)$ with $a_1 <  b_1 \in \B$. In particular, using \eqref{eq:order-length}, $|a_1| \leq |b_1|$ thus $|a_1| \leq \frac{1}{2} |b| < |a|$ which shows that $a_1<a$ and $(a,b) \in \B$.
\end{proof}

\begin{lemma}\label{Lem:length-comp-x3}
    Let $b = (a_1,b_1) \in \B$, $a \in \B$ be such that $a<a_1$ and $h \in \supp [a,a_1]$.
    Then, either $(h,b_1)\in\B$, or $(b_1,h)\in\B$ or $h=b_1$.
    In each case $[h,b_1] \in \pm \ev(\B)\cup \{0\}$.
\end{lemma}

\begin{proof} 
    We assume $h \neq b_1$. 
    
    \step{First case: $h < b_1$} 
    If $b_1 \in X$, then $(h,b_1) \in \B$.
    Otherwise, we may write $b_1 = (a_2,b_2)$ with $a_2<b_2 \in \B$ and $a_2 \leq a_1$. 
    Then $|h| = |a|+|a_1| > |a_2|$, so $h > a_2$ and $(h,b_1) \in \B$.
    
    \step{Second case: $b_1 < h$} 
    We have $|h| = |a| + |a_1| \leq 2|a_1| \leq 2|b_1|$. 
    If this inequality is strict, then $(b_1, h) \in \B$ by \cref{Lem:length-comp-x2}. 
    If equality holds, then $|a| = |a_1|$, so $(a,a_1) \in \B$ and $h$ must be $(a,a_1)$. 
    In this case, since $b_1 > a_1 > a$, we have again $(b_1,h) \in \B$.
\end{proof}


\begin{proposition} \label{Prop:estim_Lengthcomp}
    Let $i \in\N^*$, $b=(a_1,(\dotsc,(a_i,b_i)\dotsb)) \in \B$ and $a \in \B$ such that $a<a_i$. 
    Then
    \begin{equation} \label{eq:a1aibi}
        [a_1,[\dotsc,[[a,a_i],b_i]\dotsb]]=\sum \alpha_c \ev(c)
    \end{equation}
    where the sum is finite, $\alpha_c \in \Z$, $\sum |\alpha_c| \leq 2^{i-1} \|[a,a_i]\|_\B$ and the $c \in \B$ satisfy either $\lambda(c) = a_1$ or $c = (c',c'')$ with $c' \in X$ or $\lambda(c') \leq a_1$ and $c'' \in X$ or $\lambda(c'') \leq a_1$.
\end{proposition}

\begin{proof}
    The proof is by induction on $i \in \N^*$.

    \step{Initialization for $i=1$} 
    Let $b=(a_1,b_1) \in \B$ and $a \in \B$ such that $a<a_1$. 
    We introduce the decomposition of $[a,a_1]$ on $\B$: $[a,a_1]=\sum_{h \in \B} \beta_h \ev(h)$ with $\beta_h \in \K$. 
    Then 
    \begin{equation}
        [[a,a_1],b_1]=\sum_{h \in \B} \beta_h [h,b_1].
    \end{equation}
    By \cref{Lem:length-comp-x3}, for every $h \in \B$ such that $\beta_h \neq 0$, then $[h,b_1] \in \pm \ev(\B) \cup \{0\}$. 
    Thus the above expression is the expansion of $[[a,a_1],b_1]$ on $\B$, up to the sign of the basis elements. 
    Therefore,
    \begin{equation}
        \| [[a,a_1],b_1] \|_\B \leq \sum_{h \in \B} \left| \beta_h \right| = \| [a,a_1] \|_\B.
    \end{equation}
    Moreover, let us prove the desired structural properties on the supporting basis elements. 
    \begin{itemize}
        \item Since $b \in \B$, $b_1 \in X$ or $\lambda(b_1) \leq a_1$.
        \item If $|a|=|a_1|$, then $(a,a_1) \in \B$ so $h = (a, a_1)$ and $\lambda(h) = a < a_1$. Otherwise, $|h| = |a|+|a_1| < 2 |a_1|$ so $|\lambda(h)| < |a_1|$ so $\lambda(h) < a_1$.
    \end{itemize}
    
    \step{Induction}
    Let $i \geq 2$.
    We assume the statement holds until $(i-1)$. 
    Let $b=(a_1,(\dots(a_i,b_i)\dots)) \in \B$ and $a \in \B$ such that $a<a_i$. 
    Then $a_1 \geq \dotsb \geq a_i> a$ thus $|a_1| \geq \dotsb \geq |a_i| \geq |a|$. 
    By the induction hypothesis,
    \begin{equation}
        [a_2,[\dotsc,[[a,a_i],b_i]\dotsb]]=\sum \alpha_c \ev(c),
    \end{equation}
    where the sum is finite, $\alpha_c \in \Z$, $\sum |\alpha_c| \leq 2^{i-2} \|[a,a_i]\|_\B$ and the $c \in \B$ satisfy either $\lambda(c) = a_2$ or $c = (c', c'')$ with $c' \in X$ or $\lambda(c') \leq a_2$ and $c'' \in X$ or $\lambda(c'') \leq a_2$.
Then, 
\begin{equation}
    [a_1,[a_2,[\dotsc,[[a,a_i],b_i]\dotsb]]] = \sum \alpha_c  [a_1,c]
\end{equation}
and, by the triangular inequality,
\begin{equation}
    \left\| [a_1,[a_2,[\dotsc,[[a,a_i],b_i]\dotsb]]] \right\|_\B \leq \sum |\alpha_c|  \left\| [a_1,c] \right\|_\B.
\end{equation}
To conclude, it is thus sufficient to prove that, for every $c \in \B$ such that $\alpha_c \neq 0$, then $\|[a_1,c]\|_\B \leq 2$ and the expansion of $[a_1,c]$ on $\B$ only involves brackets of the desired form. 

First, for each such $c$, since $|c|=|a|+|b|-|a_1| > |b|-|a_1| \geq |a_1|$ (because $b = (a_1, \mu(b)) \in \B$) one has $a_1 < c$.
We must know consider three cases.
\begin{itemize}
    \item If $\lambda(c) = a_2$, then $(a_1, c) \in \B$ directly because $a_2 \leq a_1$.
    \item If $\lambda(c) \leq a_1$, then $(a_1, c) \in \B$ directly also.
    \item If $c = (c', c'')$ with $a_1 < c' < c''$ and ($c' \in X$ or $\lambda(c') \leq a_2$) and ($c'' \in X$ or $\lambda(c'') \leq a_2$), we use Jacobi's identity to write $[a_1,c]=[[a_1,c'],c'']+[c',[a_1,c'']]$.
    Using the pieces of information in the previous sentence suffices to ensure that these both brackets, up to sign, either vanish or belong to the basis and are of the desired form.
\end{itemize}
In particular, in all cases $\|[a_1,c]\|_\B \leq 2$, which concludes the proof.
\end{proof}

\begin{remark} \label{rk:form-c-length}
    The same proof also yields more information on the structure of elements $c \in \B$ involved in \eqref{eq:a1aibi}.
    Indeed, they are of the form
    \begin{equation}
        c = \Big(a_1, \Big(\dotsc, \Big(a_p, \Big( 
        \quad \big( a_{i_1}, (\dotsc, (a_{i_r}, v)) \big) \quad , 
        \quad \big( a_{j_1}, (\dotsc, (a_{j_s}, w)) \big) \quad 
        \Big) \Big) \dotsb \Big)\Big)
    \end{equation}
    where $\lbag v, w \rbag = \lbag h, b_i \rbag$ for some $h \in \B$ in the $\supp [a, a_i]$, $p \in \intset{0, i-1}$ and the sets $\{i_1 <  \dotsc < i_r\}$, $\{j_1 < \dotsc < j_s\}$ are a partition of $\intset{p+1, i-1}$.
\end{remark}

\subsection{Proof of the main upper bound} \label{subsec:Proof_lengthcomp}

We now prove the following refined version of \cref{thm:length-easy}.

\begin{theorem}\label{Thm:lgth-comp-main}
    Let $\B \subset \Br(X)$ be a length-compatible Hall set.
    For all $a<b \in \B$,
    \begin{equation} \label{Estim:lgth-comp}
        \|[a,b]\|_\B \leq 2^{\theta_{a}(b)-1} \leq 2^{\left\lfloor \frac{|b|}{|a|} \right\rfloor-1}.
    \end{equation}
\end{theorem}

\begin{proof}
    The second estimate is a consequence of the first one. 
    Indeed, any leaf $c$ of $\rf{a}{b}$ satisfies $|c| \geq |a|$ because $(a,c) \in \B$, and these leaves are in number $\theta_a(b)$, thus $|b| \geq \theta_a(b)|a|$. Since $\theta_a(b)$ is an integer, we can take the integral part of $|b|/|a|$.
    
    We prove the first estimate by induction on $q := \theta_a(b) \in \N^*$. 
    
    The initialization for $q=1$ is given by \cref{Lem:length-comp-x2}.
    
    Let $q \geq 2$. 
    We assume the property proved until $(q-1)$.
    Let $a<b \in \B$ be such that $\theta_a(b)=q$. 
    We define two sequences $(a_i)_{i \in \intset{1 , i_*}}$, $(b_i)_{i \in \intset{1 , i_*}}$ of $\B$ by the following relations: for every $i \in \intset{1 , i_*}$,
    \begin{equation}
        b=(a_1,(\dotsc,(a_i,b_i)\dotsb)),
    \end{equation}
    i.e.\ $a_i:=\lambda(\mu^{i-1}(b))$, $b_i:=\mu^{i}(b)$.
    The integer $i_*$ is the smallest value of $i$ for which $b_i \in X$.
    Then $a_1 \geq \dotsb \geq a_{i_*}$ and $a_i < b_i$ for every $i \in \intset{1 , i_*}$. 
    We define
    \begin{equation}
        k := \begin{cases}
            \min\{ i \in \intset{1, i_*-1} ; \enskip a_{i+1}< a \} \text{ if this set is not empty}, \\
            i_* \text{ otherwise.}
        \end{cases}
    \end{equation}
    Then $(a,b_k) \in \B$. 
    Indeed, if $k \in \intset{1, i_*-1}$ then $a_{k+1} < a \leq a_k < b_k$ so $(a,b_k)=(a,(a_{k+1},b_{k+1})) \in \B$. If $k=i_*$ then $a<a_{i_*}<b_{i_*}$ and $b_{i_*} \in X$ thus $(a,b_{i_*}) \in \B$.

    Iterating Jacobi's identity, we get the expression
    \begin{equation}
        \label{Eqn:iterated-jacobi}
        [a,b] = [a_1,[ \dotsc,[a_k,[a,b_k]]\dotsb]] + \sum_{i=1}^{k} [a_1,[\dotsc,[[a,a_i],b_{i}]\dotsb]].
    \end{equation}
    The first term belongs to $\ev(\B)$. 
    Indeed $(a_k,(a,b_k)) \in \B$ because $a_k < b_k < (a,b_k)$. 
    Moreover, for every $i \in \intset{1 , k-1}$, $a_i<b_i=(a_{i+1}(\dotsc,(a_k,b_k)\dotsb))<(a_{i+1},(\dotsc,(a_k,(a,b_k)) \dotsb))$ because the order is length-compatible and $a_i \geq a_{i+1}$ thus $(a_i,(\dotsc,((a,a_k),b_k) \dotsb)) \in \B$.

    Then, by \cref{Prop:estim_Lengthcomp} and the induction assumption,
    \begin{equation}
        \| [a,b] \|_\B \leq 1 + \sum_{i=1}^{k} 2^{i-1} 2^{\theta_a(a_i)-1}.
    \end{equation}
    Moreover, by definition of the map $\theta_a$, we have $\sum_{i=1}^k \theta_a(a_i)= \theta_a(b) - \theta_a(b_k) \leq \theta_a(b)-1$.
    Finally, we get the conclusion by the following lemma.
\end{proof}

\begin{lemma}\label{Lem:sum2ki}
    Let $k \in \N^*$, $q_1, \dotsc, q_{k} \in \N^*$ and $Q := q_1+\dotsb+q_k $. 
    Then
    \begin{equation}
        \sum_{i=1}^{k} 2^{i-1} 2^{q_i-1} \leq 2^Q -1.
    \end{equation}
\end{lemma}

\begin{proof}
    The proof is by induction on $k \in \N^*$.
    If $k=1$, then $2^{q_1 - 1} \leq 2^{q_1}-1 = 2^Q - 1$ holds.
    Let $k \geq 2$. 
    We assume the result holds until $(k-1)$.
    Let $q_1,\dotsc,q_k$ and $Q$ as in the statement.
    Then
    \begin{equation}
        \begin{split}
            \sum_{i=1}^{k} 2^{i-1} 2^{q_i-1} 
            & = 2^{q_1-1} + 2 \sum_{i=1}^{k-1} 2^{i-1} 2^{q_{i+1}-1} \\
            & \leq 2^{q_1-1} + 2 \left(2^{Q-q_1}-1\right) \\
            & = 2^{Q} - 1 - (2^{q_1-1} - 1)(2^{Q-q_1+1}-1) \leq 2^Q - 1,
        \end{split}
    \end{equation}
    which concludes the proof and shows that equality holds if and only if $q_i = 1$ for $i \in \intset{1,k-1}$.
\end{proof}

\begin{remark}
    The same proof also yields more information on the structure of the elements of $\supp [a,b]$: either $(a_1, (\dotsc, (a_k,(a,b_k))\dotsb))$ or of the form described in \cref{rk:form-c-length} for $i \in \intset{1,k}$.
    In any case, they satisfy the same structural property as in \cref{Prop:estim_Lengthcomp}, i.e.\ either $\lambda(c) = \lambda(b)$ or $c = (c',c'')$ with ($c'\in X$ or $\lambda(c') \leq \lambda(b)$) and ($c'' \in X$ or $\lambda(c'') \leq \lambda(b)$).
\end{remark}

\subsection{Optimality cases} 
\label{subsec:optim_lengthcomp}

We now investigate the optimality of the estimate proved in the previous paragraph. 
We start with the following elementary bound on $\theta_a(b)$, which should be seen as a refinement of \cref{Prop:theta/length} in the case of length-compatible Hall sets when $|X| = 2$.

\begin{lemma} \label{p:theta-b2}
    Let $X = \{ X_0, X_1 \}$.
    Let $\B \subset \Br(X)$ be a length-compatible Hall set such that $X_0 < X_1$.
    Then, for every $a<b \in \B$ with $|b|\geq 3$,
    \begin{equation}
        \theta_a(b) 
        \begin{cases}
            = |b|-1 \text{ if } a=X_0 \text{ and } b=\ad_{X_1}^{n}(X_0,X_1) \text{ for some } n\in\N, \\
            \leq |b|-2 \text{ otherwise. }
        \end{cases}
    \end{equation}
\end{lemma}

\begin{proof}
    One already knows that $\theta_a(b) \leq |b|-1$ by the second item of \cref{Prop:theta/length}.
    Let $a<b \in \B$ with $|b|\geq 3$ such that $\theta_a(b)=|b|-1$. 
    Then $|b|-1=\theta_a(b) \leq \frac{|b|}{|a|}$. 
    Indeed, any leaf $c$ of $\rf{a}{b}$ satisfies $|c| \geq |a|$ because $(a,c) \in \B$, and these leaves are in number $\theta_a(b)$, thus $|b| \geq \theta_a(b)|a|$. 
    Therefore, $|a| \leq \frac{|b|}{|b|-1} \leq \frac{3}{2}$ and $a \in X$.
    Necessarily, $\rf{a}{b}$ contains $\theta_a(b)-1=|b|-2$ leaves labeled by an indeterminate strictly greater than $a$, and one leaf labeled by the only basis element with length two, i.e.\ $(X_0,X_1)$. 
    In conclusion, $a=X_0$, and $b=\ad_{X_1}^{n}(X_0,X_1)$ for some $n \in \N$, for which one checks that $a < b \in \B$ and $\theta_a(b) = n + 1 = |b| - 1$.
\end{proof}

\begin{proposition} \label{p:length-sharp}
    Let $\B \subset \Br(X)$ a length-compatible Hall set.
    \begin{enumerate}
        \item If $|X|\geq 3$, for every $a<b \in \B$, $\|[a,b]\|_\B \leq 2^{|b|-1}$.
        \item If $X \supset \{ X_0, X_1, X_2 \}$ with $X_0 < X_1 < X_2$, for every $n\in \N$, $\ad_{X_1}^n (X_2) \in \B$ and 
        \begin{equation}
            \| [X_0, \ad_{X_1}^n(X_2)] \|_\B = 2^n
        \end{equation}
        so the previous estimate is optimal.
        \item If $|X|=2$, for every $a<b \in \B$ with $|b|\geq 3$, $\|[a,b]\|_\B \leq 2^{|b|-3}$.
        \item If $X=\{X_0,X_1\}$ with $X_0<X_1$, for every $n \in \N$, $\ad_{X_1}^n\ad_{X_0}^2(X_1) \in \B$ and
        \begin{equation}
            \|[X_0,\ad_{X_1}^n\ad_{X_0}^2(X_1)]\|_\B=2^n
        \end{equation}
        so the previous estimate is optimal.
    \end{enumerate}
\end{proposition}

\begin{proof}
    We proceed step by step.
    Some computations are postponed to \cref{sec:geom-lower} since they can be carried out in more generality, with any Hall order.
\begin{enumerate}
    \item This follows from \cref{Thm:lgth-comp-main} and the first item of \cref{Prop:theta/length}.
    
    \item This example is detailed in \cref{thm:X3-2^n} in full generality.
    
    \item Using \cref{Thm:lgth-comp-main} and \cref{p:theta-b2}, we obtain the expected bound on $\|[a,b]\|_\B$, except when $a=X_0$ and $b=\ad_{X_1}^{n-1}(X_0,X_1)$ for some $n \geq 2$.
    Then, \eqref{eq:X0bn-r1-2n} yields $\|[a,b]\|_\B \leq 2^{n-2} = 2^{|b|-3}$.
    
    \item For any $j \in \N$, $b_{j}:=\ad_{X_1}^{j}\ad_{X_0}^2(X_1)$ belongs to $\B$ because $X_0<X_1$ and the order is length-compatible. By iterating the Jacobi relation, we obtain
    \begin{equation}
        \begin{split}
            [X_0,\ad_{X_1}^n\ad_{X_0}^2(X_1)] 
       & = \sum_{k=0}^{n-1} \ad_{X_1}^k [[X_0,X_1],b_{n-k-1}]  + \ad_{X_1}^n \ad_{X_0}^3(X_1)
       \\ & = \sum_{k=0}^{n-1}\sum_{j=0}^{k} \binom{k}{j} [\ad_{X_1}^j[X_0,X_1], b_{n-j-1} ]  + \ad_{X_1}^n \ad_{X_0}^3(X_1)
       \\ & = \sum_{j=0}^{n-1} \alpha_j [\ad_{X_1}^j[X_0,X_1], b_{n-j-1} ]  + \ad_{X_1}^n \ad_{X_0}^3(X_1)
        \end{split}
    \end{equation}
    where $\alpha_j:=\sum_{k=j}^{n-1} \binom{k}{j}$. 
    Note that $\ad_{X_1}^n \ad_{X_0}^3(X_1) \in \B$ because $X_0<X_1$.
    Let us prove that $\pm [\ad_{X_1}^j[X_0,X_1], b_{n-j-1} ]\in\ev(\B)$ for every $j \in \intset{0, n-1}$.

    If $\ad_{X_1}^j(X_0,X_1) < b_{n-j-1}$ then $(\ad_{X_1}^j(X_0,X_1), b_{n-j-1} ) \in \B$ because $\lambda(b_{n-j-1})$ is either $X_1$ or $X_0$ thus $\lambda(b_{n-j-1}) \leq \ad_{X_1}^j(X_0,X_1)$ because the order is length-compatible.

    If $b_{n-j-1} < \ad_{X_1}^j(X_0,X_1)$ then $(b_{n-j-1},\ad_{X_1}^j(X_0,X_1)) \in \B$ because $\lambda(\ad_{X_1}^j(X_0,X_1))$ is either $X_1$ or $X_0$ thus $\leq b_{n-j-1}$ because the order is length-compatible.

    Eventually,
    \begin{equation}
        \| [X_0,\ad_{X_1}^n\ad_{X_0}^2(X_1)] \|_\B = \sum_{j=0}^{n-1} \alpha_j + 1 = \sum_{k=0}^{n-1} 2^k + 1 = 2^n.
    \end{equation}
\end{enumerate}
    This concludes the proof of the optimality of the length-based estimates.
\end{proof}

Interpreted in terms of the symmetric quantity $\beta_n(\B)$ defined in \eqref{eq:beta-n}, these examples yield the following consequences.

\begin{corollary} \label{rk:length-beta}
    Let $\B \subset \Br(X)$ be length-compatible Hall set.
    When $|X| \geq 3$, $\beta_n(\B) = 2^{n-2}$ for every $n \geq 2$.
    When $|X| = 2$, $\beta_n(\B) = \max \{ 1, 2^{n-4} \}$ for every $n \geq 2$.
\end{corollary}

\section{Lyndon basis}
\label{sec:lyndon}

The main goal of this section is to prove \cref{thm:lyndon-easy}, which yields a sharp bound of the structure constants relative to the Lyndon basis.
We start with definitions and a short introduction to the Lyndon basis in \cref{ss:lyndon-intro}, then prove a refined version of \cref{thm:lyndon-easy} in \cref{ss:lyndon-bound}.
Eventually, we investigate the optimality of this estimate in \cref{ss:lyndon-optimal}.

\subsection{Definitions and preliminary remarks}
\label{ss:lyndon-intro}

In this section, $X$ is totally ordered and $X^*$ (recall \cref{def:free.monoid}) is endowed with the induced lexicographic order. 
We study the case of a classical Hall basis of $\mathcal{L}(X)$ indexed by Lyndon words in $X^*$, the Lyndon basis. 
This basis is sometimes referred to as the ``\emph{Chen-Fox-Lyndon basis}'' due to the important related results proved in \cite{zbMATH03229405}, or as the ``\emph{Shirshov basis}'' due to the work \cite{zbMATH03131962}.
As in \cite{zbMATH00417855}, we choose the name ``\emph{Lyndon basis}'' for brevity and to highlight the source work \cite{zbMATH03092594} where Lyndon introduced ``\emph{standard sequences}'' (which are now named Lyndon words).

For further details on the combinatorics of Lyndon words and their relations with Hall sets, the reader can refer to~\cite[Chapter 5]{zbMATH01024080}, \cite[Section 5.1]{zbMATH00417855}.

\begin{definition}[Length, prefixes and suffixes]
    We use the following notions for elements of $X^*$:
    \begin{itemize}
        \item If $u \in X^*$, its length $|u|$ is the length of the corresponding sequence.
        \item If $u,v \in X^*$, we say that $u$ is a \emph{prefix} of $v$ if there exists $w \in X^*$ such that $v = uw$.
        \item If $u,v \in X^*$, we say that $u$ is a \emph{suffix} of $v$ if there exists $w \in X^*$ such that $v = wu$.
    \end{itemize}
\end{definition}

\begin{definition}[Lyndon word]
    A word $w \in X^*$ is a Lyndon word if either $w \in X$, or for all $u,v \in X^*$ such that $w = uv$, $w <vu$.
    Denote by $\Lyn(X)$ the set of Lyndon words in $X^*$.
\end{definition}

 As a consequence of this definition, if $u,v\in \Lyn(X)$ are such that $u<v$, then $uv<v$. 
 Every Lyndon word that is not a letter can be written as the concatenation of two shorter Lyndon words. 
 Such a factorization is not unique in general, but we can single out one of them: the \emph{standard factorization} of $w$ is the factorization $w=uv$ with $u,v \in \Lyn(X)$ such that $u$ has maximal length. 
 Standard factorizations allow us to recursively define a map $\br$: $\Lyn(X) \to \Br(X)$ by mapping each letter to itself and a Lyndon word $w$ to $\br(w) = (\br(u),\br(v))$, where $w=uv$ is the standard factorization of $w$. 
 Endow $\Br(X)$ with the preorder given by the lexicographic order on the foliage of trees, that we will call the \emph{Lyndon order} (this order is not a total order, but it does not matter as we only intend to use its restriction to $\br{\Lyn(X)}$, which is a total order).

\begin{definition}[Lyndon basis]
   The subset $\B=\br(\Lyn(X))$ is a Hall set (see \cite[Theorem 5.1]{zbMATH00417855}).
   The associated basis of $\mathcal{L}(X)$ is  called the Lyndon basis.
\end{definition}

In this section, if $u \in \Lyn(X)$, extending the convention for left and right factors of the elements of $\Br(X)$, we will denote by $\lambda(u)$ the maximal strict Lyndon prefix of $u$, and by $\mu(u)$ the suffix of $u$ such that $u = \lambda(u)\mu(u)$. 
If $u \in X$, we use the convention that $\lambda(u)$ is the empty word $\varepsilon$.

Let us recall some useful properties of the lexicographic order:
\begin{itemize}
    \item If $m,m',m'' \in X^*$ and $m'<m''$, then $mm' < mm''$.
    \item If $m,m',m'' \in X^*$ and $m<m'<mm''$, then $m$ is a prefix of $m'$.
\end{itemize}

\begin{lemma} \label{p:lyndon-order}
    Let $b_1, b_2, b_3 \in \B$. Then:
	\begin{enumerate}
		\item $\lambda(b_1) < b_1 < \mu(b_1)$,
		\item if $b_1 < b_2 \text{ and } |b_1| > |b_2| \text{ and } b_1 \in \Lambda(b_3)$, then $b_3 < b_2$,
		\item if $b_1 < b_2 < b_3 \text{ and } b_1 \in \Lambda(b_3)$, then $b_1 \in \Lambda(b_2)$,
		\item if $b_2 < b_3$, then $(b_1, b_2) < (b_1, b_3)$.
    \end{enumerate}
\end{lemma}

\begin{proof}
	These are straightforward consequences of \eqref{eq:LAMBDA}, the previously mentioned properties of the lexicographic order, and specific properties of Lyndon words.
\end{proof}

\subsection{Proof of the main upper bound}
\label{ss:lyndon-bound}

We now prove the following refined version of \cref{thm:lyndon-easy}.

\begin{theorem} \label{thm:lyndon-bound}
	Let $\B \subset \Br(X)$ be the Hall set of the Lyndon basis.
	For all $a < b \in \B$,
	\begin{equation}
		\| [a, b] \|_\B \leq 2^{\theta_a(b)-1}
	\end{equation}
	and for each $c \in \supp [a,b]$, $a \in \Lambda(c)$ (so, in particular, $a < c$), $\lambda(c) \leq \max\{a, \lambda(b)\}$, and $c < b$.
\end{theorem}

\begin{proof}
	We proceed by induction on $n := \theta_a(b) \geq 1$ and execute a refined version of the decomposition algorithm \cref{sec:rewriting}.
	
	\step{Initialization for $n = 1$}
	By definition of $\theta_a(b) = 1$ implies that $(a,b) \in \B$ so is its own decomposition. 
	Let $c := (a, b)$.
	Then one has all the desired properties, since $\lambda(c) = a$, $a \in \Lambda(c)$, $\lambda(c) \leq \max \{ a, \lambda(b) \}$ and $c < b$ by item (1) of \cref{p:lyndon-order}.
	Moreover, $\| [a, b] \|_\B = 1 = 2^{1-1}$.
	
	\step{Induction}
	Assume that the result is proved up to some $n \in \N^*$.
	Let $a < b \in \B$ such that $\theta_a(b) = n+1$.
	In particular, $b = (\lambda(b), \mu(b))$ with $a < \lambda(b) < \mu(b)$ (otherwise $\theta_a(b) = 1$).
	By Jacobi's identity,
	\begin{equation}
		[a, b] = [[a, \lambda(b)], \mu(b)] + [\lambda(b), [a, \mu(b)]]
	\end{equation}
	and we now treat both terms separately.
	
	\begin{itemize}
		\item \emph{Study of $[[a, \lambda(b)], \mu(b)]$.}
		Since $\theta_a(\lambda(b)) \leq \theta_a(b)-1 \leq n$, the induction hypothesis yields
		\begin{equation}
			[a, \lambda(b)] = \sum \alpha_d \ev(d)
		\end{equation}
		with $d \in \B$, $a \in \Lambda(d)$, $\lambda(d) \leq \max \{ a, \lambda^2(b) \}$, $d < \lambda(b)$ and $\sum |\alpha_d| \leq 2^{\theta_a(\lambda(b))-1}$.

		For each $d$ in this sum, $a < d < \lambda(b) < \mu(b)$ and, by \cref{p:theta-dec}, $\theta_d(\mu(b)) \leq \theta_a(\mu(b)) \leq n$, so the induction hypothesis applied to $[d, \mu(b)]$ yields $[d, \mu(b)] = \sum \beta_c \ev(c)$ where $d \in \Lambda(c)$, $\lambda(c) \leq \max \{ d, \lambda(\mu(b)) \}$, $c < \mu(b)$ and $\sum |\beta_c| \leq 2^{\theta_a(\mu(b))-1}$.
		
		Since $a \in \Lambda(d)$ and $d \in \Lambda(c)$, $a \in \Lambda(c)$.
		Since $b \in \B$, $\lambda(\mu(b)) \leq \lambda(b)$ and, since $d < \lambda(b)$, $\lambda(c) \leq \max \{d, \lambda(\mu(b))\}$ implies that $\lambda(c) \leq \lambda(b) = \max \{ a, \lambda(b) \}$.
		Moreover, as $d < \lambda(b)$, $|d| > |\lambda(b)|$ and $d \in \Lambda(c)$, item (2) of \cref{p:lyndon-order} implies that $c < \lambda(b) < b$.
		
		This proves that the support of $[[a,\lambda(b)],\mu(b)]$ has the desired properties and the size estimate $\|[[a,\lambda(b)],\mu(b)]\|_\B \leq 2^{\theta_a(\lambda(b))-1} 2^{\theta_a(\mu(b))-1} = 2^{\theta_a(b)-2}$.
		
		\item \emph{Study of $[\lambda(b), [a, \mu(b)]]$.}
		Since $\theta_a(\mu(b)) \leq n$, the induction hypothesis yields
		\begin{equation}
			[a, \mu(b)] = \sum \alpha_d \ev(d)
		\end{equation}
		with $d \in \B$, $a \in \Lambda(d)$, $\lambda(d) \leq \max \{ a, \lambda(\mu(b)) \}$, $d < \mu(b)$ and $\sum |\alpha_d| \leq 2^{\theta_a(\mu(b))-1}$.
		
		Let $d \in \B$ be part of this sum. 
		Unlike the previous case, we do not know how $d$ and $\lambda(b)$ compare, so we treat both cases separately.
		
		\begin{itemize}
			\item If $d < \lambda(b)$, since by \cref{p:theta-dec}, $\theta_d(\lambda(b)) \leq \theta_a(\lambda(b)) \leq n$, the induction hypothesis applied to $[d, \lambda(b)]$ yields $[d, \lambda(b)] = \sum \beta_c \ev(c)$ where $d \in \Lambda(c)$, $\lambda(c) \leq \max \{ d, \lambda^2(b) \}$, $c < \lambda(b)$ and $\sum |\beta_c| \leq 2^{\theta_a(\lambda(b))-1}$.
			Since $a \in \Lambda(d)$ and $d \in \Lambda(c)$, $a \in \Lambda(c)$.
			Also, $\lambda(c) \leq \max \{d, \lambda^2(b)\} < \lambda(b) = \max \{ a, \lambda(b) \}$ and $c < \lambda(b) < b$.
			So $c$ has the required properties and $\|[d,\lambda(b)]\|_\B \leq 2^{\theta_a(\lambda(b))-1}$.
			
			\item  If $\lambda(b)<d$, since $\lambda(d) \leq \max\{a,\lambda(\mu(b))\} \leq \lambda(b)$, we have $c := (\lambda(b),d) \in \B$.
			Since $a < \lambda(b) < d$ and $a \in \Lambda(d)$, item (3) of \cref{p:lyndon-order} yields $a \in \Lambda(\lambda(b)) \subset \Lambda(c)$.
			Obviously $\lambda(c) = \lambda(b) \leq \max \{ a, \lambda(b) \} = \lambda(b)$.
			Finally, since $d < \mu(b)$, item (4) of \cref{p:lyndon-order} yields $c = (\lambda(b), d) < (\lambda(b), \mu(b)) = b$. 
			So $c$ has the required properties and $\| [\lambda(b), d] \|_\B = 1$.
		\end{itemize}
		
		This proves that the support of $[[a,\lambda(b)],\mu(b)]$ has the desired properties and the size estimate $\|[[a,\lambda(b)],\mu(b)]\|_\B \leq 2^{\theta_a(\mu(b))-1} \max \{ 2^{\theta_a(\lambda(b))-1} , 1 \} = 2^{\theta_a(b)-2}$.
	\end{itemize}
	Combining both studies and summing the estimates concludes the proof.
\end{proof}

\subsection{Optimality cases}
\label{ss:lyndon-optimal}

\cref{thm:lyndon-bound} yields the following optimal length-based estimates.

\begin{proposition} \label{p:lyndon-sharp}
    Let $\B \subset \Br(X)$ the Hall set of the Lyndon basis.
    \begin{enumerate}
        \item If $|X|\geq 3$, for every $a<b \in \B$, $\|[a,b]\|_\B \leq 2^{|b|-1}$.
        \item If $X \supset \{ X_0, X_1, X_2 \}$ with $X_0 < X_1 < X_2$, for every $n\in \N$, $\ad_{X_1}^n (X_2) \in \B$ and 
        \begin{equation}
            \| [X_0, \ad_{X_1}^n(X_2)] \|_\B = 2^n
        \end{equation}
        so the previous estimate is optimal.
        \item If $|X|=2$, for every $a<b \in \B$ with $|b|\geq 2$, $\|[a,b]\|_\B \leq 2^{|b|-2}$.
        \item If $X=\{X_0,X_1\}$ with $X_0<X_1$, for every $n \in \N^*$, $\ad_{X_0}^2(X_1) < \dad_{X_1}^n(X_0) \in \B$ and
        \begin{equation} \label{eq:lyndon-n4-n1}
            \|[\ad_{X_0}^2(X_1), \dad_{X_1}^n(X_0)]\|_\B=2^{n-1}
        \end{equation}
        so the previous estimate is optimal.
    \end{enumerate}
\end{proposition}

\begin{proof}
    We proceed step by step.
    Some computations are postponed to \cref{sec:geom-lower} since they can be carried out in more generality, with any Hall order.
\begin{enumerate}
    \item This follows from \cref{thm:lyndon-bound} and the first item of \cref{Prop:theta/length}.
    
    \item This example is detailed in \cref{thm:X3-2^n} in full generality.
    
    \item This follows from \cref{thm:lyndon-bound} and the second item of \cref{Prop:theta/length}.
    
    \item 
    Let $a := \ad_{X_0}^2(X_1)$.
    By \cref{p:ad-2} (applied to the derivation $\dad_{X_1}$ on $\mathcal{L}(X)$, $\nu \leftarrow n-1$, $b_1 \leftarrow a$ and $b_2 \leftarrow [X_0, X_1]$),
    \begin{equation}
        \begin{split}
            [ a, \dad_{X_1}^n(X_0) ]
            & = [ a, \dad_{X_1}^{n-1} ([X_0,X_1]) ] \\
            & = \sum_{i=0}^{n-1} (-1)^i \binom{n-1}{i} \dad_{X_1}^{n-1-i} [ \dad_{X_1}^i (a), [X_0, X_1] ].
        \end{split}
    \end{equation}
    One checks that the elements of the right-hand side are (evaluations of) distinct basis elements, so the norm estimate readily follows.
\end{enumerate}
    This concludes the proof of the optimality of the Lyndon basis estimates.
\end{proof}

Interpreted in terms of the symmetric quantity $\beta_n(\B)$ defined in \eqref{eq:beta-n}, these examples yield the following consequences.

\begin{corollary} \label{cor:lyndon-beta-X3}
	Let $\B \subset \Br(X)$ be the Hall set of the Lyndon basis over $X$ with $|X| \geq 3$.
	Then, for every $n \geq 2$, $\beta_n(\B) = 2^{n-2}$.
\end{corollary}
	
\begin{proposition} \label{rk:lyndon-beta}
	Let $\B \subset \Br(X)$ be the Hall set of the Lyndon basis over $X$ with $|X| = 2$.
	Then, for every $n \geq 2$, $\beta_n(\B) = \max \{ 1, F_{n-2}, 2^{n-5} \}$.
\end{proposition}

\begin{proof}
	Let $n \in \N$.
	Let $a < b \in \B$ with $|a|+|b| = n$.
	We separate cases depending on $|a|$.
	For $m \in \N$, we introduce the bracket $A_m := \dad_{X_1}^m(X_0) \in \B$.

    \step{Case $|a| \geq 3$}
	By the third item of \cref{p:lyndon-sharp}, $\|[a,b]\|_\B \leq 2^{|b|-2} \leq 2^{n-5}$.
	
	\step{Case $|a| = 2$}
	Then $a = (X_0, X_1)$.
	By \cref{Prop:theta/length}, $\theta_a(b) \leq |b|-1$.
	\begin{itemize}
		\item If $\theta_a(b) = |b|-1$, then $b = \dad_{X_1}^{m}(X_0, X_1) = A_{m+1}$ with $m = n-4 \geq 1$.
		Formula \eqref{wi,i+n} (applied with $a \gets X_1$, $b \gets (X_0, X_1)$, $n \gets m$) yields
		\begin{equation}
			[A_1, A_{m+1}] =
			\sum_{p=0}^{\lfloor \frac{m-1}{2} \rfloor} (-1)^p \binom{m-1-p}{p} \dad_{X_1}^{m-1-2p} [A_{p+1}, A_{p+2}].
		\end{equation} 
		This is indeed the decomposition of $[a,b]$ on the basis, so $\|[a,b]\|_\B = F_{m} = F_{n-4}$.
		
		\item Otherwise, $\theta_a(b) \leq |b|-2$, so, by \cref{thm:lyndon-bound}, $\|[a,b]\|_\B \leq 2^{\theta_a(b)-1} \leq 2^{|b|-3} = 2^{n-5}$.
	\end{itemize}
		
	\step{Case $|a| = 1$}
	Then $a = X_0$.
	By \cref{Prop:theta/length}, $\theta_a(b) \leq |b|-1$.
	\begin{itemize}
		\item If $\theta_a(b) = |b| - 1$, then $b = \dad_{X_1}^m(X_0) = A_m$ with $m = n - 2 \geq 1$.
		Formula \eqref{wi,i+n} (applied with $a \gets X_1$, $b \gets X_0$, $n \gets m$) yields
		\begin{equation} \label{eq:lyndon-wm}
			[X_0, A_m] = 
			\sum_{p=0}^{\lfloor \frac{m-1}{2} \rfloor} (-1)^p \binom{m-1-p}{p} \dad_{X_1}^{m-1-2p} [A_p, A_{p+1}].
		\end{equation} 
		This is indeed the decomposition of $[a,b]$ on the basis, so $\|[a,b]\|_\B = F_{m} = F_{n-2}$.
		
		\item If $\theta_a(b) = |b| - 2$, $b$ is of one of the following forms.
		\begin{itemize}
		    \item \emph{Subcase $b = \dad_{X_1}^m\ad_{X_0}^2(X_1) =: S_m$ with $m = n - 4 \geq 0$}.
		    One has $[X_0, S_0] = \ad_{X_0}^3(X_1) \in \B$ so $\|[X_0,S_0]\|_\B = 1$.
		    Moreover, by Jacobi's identity,
		    \begin{equation}
		        [X_0, S_{m+1}] = [X_0, [S_m, X_1]]
		        = [[X_0, S_m], X_1] + [S_m, [X_0, X_1]].
		    \end{equation}
		    Since $(S_m, (X_0, X_1)) \in \B$, a straightforward induction on $m$ proves that $\|[X_0, S_m]\|_\B \leq m+1$.
		    Hence $\|[a,b]\|_\B \leq n-3 \leq F_{n-2}$.
		    
		    \item \emph{Subcase $b = \dad_{X_1}^q (H_p)$ with $p \geq 1$, $\in \N$, $2p+q+4 = n$, where $H_p := (A_p, A_{p+1})$.}
    		We start with the case $q = 0$.
    		Using Jacobi's identity,
    		\begin{equation}
    		    [X_0,H_p] = [[X_0,A_p],A_{p+1}] - [[X_0,A_{p+1}],A_p].
    		\end{equation}
    		If $d \in \supp [X_0, A_p]$ (respectively $d \in \supp [X_0, A_{p+1}]$), then, by \eqref{eq:lyndon-wm}, $d < A_{p+1}$ (resp.\ $d < A_p$) and $\theta_d(A_{p+1}) \leq |A_{p+1}|-1=p+1$ (resp.\ $\theta_d(A_p) \leq p$).
    		Hence, by \cref{thm:lyndon-bound},
    		\begin{equation} \label{eq:2pfp}
    		    \| [X_0, H_p] \|_\B 
    		    \leq F_p 2^p + F_{p+1} 2^{p-1}
    		    \leq 2^{p-1} F_{p+3},
    		\end{equation}
    		by \eqref{eq:fibo-2+2}.
    		We now proceed by induction on $q \in \N$.
    		By Jacobi's identity,
    		\begin{equation}
    		    [X_0, \dad_{X_1}^{q+1}(H_p)]
    		    = [[X_0, \dad_{X_1}^{q}(H_p)], X_1]
    		    - [[X_0, X_1], \dad_{X_1}^q(H_p)].
    		\end{equation}
    		Since $\theta_{(X_0,X_1)}(\dad_{X_1}^q(H_p)) = 2p-1+q$, by \cref{thm:lyndon-bound} and induction,
    		\begin{equation}
    		    \| [X_0, \dad_{X_1}^{q}(H_p)] \|_\B 
    		    \leq \| [X_0, H_p] \|_\B 
    		    + \sum_{r=0}^{q-1} 2^{(2p-1+r)-1}
    		    \leq 2^{p-1} F_{p+3} + 2^{2p-2} (2^q-1)
    		\end{equation}
    		using \eqref{eq:2pfp}.
    		Hence, for $n$ large enough, $\| [X_0, \dad_{X_1}^{q}(H_p)] \|_\B < 2^{n-5}$.
    		For small values of $p$ and $q$, one checks that the right-hand side is indeed bounded by $\max \{ 2^{n-5}, F_{n-2} \}$.
		\end{itemize}
		
		\item Otherwise, $\theta_a(b) \leq |b|-3$, so, by \cref{thm:lyndon-bound}, $\|[a,b]\|_\B \leq 2^{\theta_a(b)-1} \leq 2^{|b|-4} = 2^{n-5}$.
	\end{itemize}
	This concludes the proof of the upper bounds.
	The lower bounds come from the fourth item of \cref{p:lyndon-sharp} (see \eqref{eq:lyndon-n4-n1}) and from the case $\| [X_0, A_{n-2}] \|_\B = F_{n-2}$ (see \cref{eq:lyndon-wm}).
\end{proof}

\section{A minimal Hall set}
\label{s:fibo}

The main goal of this section is to prove \cref{thm:fibo-easy}, which illustrates the optimality of the lower bound of \cref{thm:LB2} when $|X| = 2$ by exhibiting a Hall set $\B \subset \Br(X)$ with $|X| = 2$ for which the growth of the structure constants has the magnitude of the Fibonacci sequence, i.e.\ such that $\beta_{n+2}(\B) = F_n$.
We define a quite natural Hall set to answer this question, which seems new.

We give some motivations for the properties that we are looking for in \cref{s:fibo-motivation}.
We define the Hall set we will consider in \cref{s:fibo-def} and prove elementary structural properties in \cref{s:fibo-elementary}.
We then start by a $\theta$-based size estimate for a particular family of right-nested brackets in \cref{s:fibo-right-quick}.
In \cref{s:fibo-general}, we then deduce from it a general size estimate, valid for all brackets, which distinguishes the role of the maximal indeterminate and allows to prove a slightly weaker version of \cref{thm:fibo-easy} (see \cref{cor:fibo-light}).
In \cref{s:fibo-refined}, we prove a refined version of our general estimate which concludes the proof of \cref{thm:fibo-easy}.

Eventually, in \cref{s:fibo-theta}, we investigate an independent question concerning the optimal $\theta$-based estimate for our Hall set, which, quite surprisingly, turns out to be larger than the $\theta$-based estimates for length-compatible Hall sets and the Lyndon basis.

As in \cref{subsec:2letter_case}, throughout all this section, $(F_\nu)_{\nu \in \N}$ denote the $0$-based Fibonacci numbers.

\subsection{Motivation}
\label{s:fibo-motivation}

Let $X := \{X_0,X_1\}$ and let $\B\subset \Br(X)$ be a Hall set with $X_0<X_1$. 
We start by deriving necessary conditions for $\beta_n(\B)$ to have a geometric growth with a rate smaller than $2$.

\begin{lemma}\label{lem:CN-fibonacci}
    Let $X := \{X_0,X_1\}$ and let $\B\subset \Br(X)$ be a Hall set with $X_0<X_1$. 
    With the notation of \eqref{eq:beta-n}, assume that 
    \begin{equation} \label{eq:beta-2n}
        \lim_{n\to\infty} \frac{\beta_n(\B)}{2^n} = 0.
    \end{equation}
    Then:
    \begin{itemize}
        \item $\forall a \in \B$ such that $a<X_1$, $r(a,X_1) = + \infty$,
        \item $X_1 = \max\B$,
        \item $\forall a,b \in \B\setminus\{X_1\}$, $\exists n \in \N$ such that $\dad_{X_1}^n(a) > b$.
    \end{itemize}
\end{lemma}

\begin{proof}
    We prove each item successively by contradiction.
    
    \step{First item}
    Assume that there exists $a \in \B$ such that $a<X_1$ and $r(a,X_1)=r<+\infty$. 
    For $n \geq r$, let $b_n := \ad_{X_1}^{n-r}(\dad_{X_1}^r(a)) \in \mathcal B$ (see \cref{lem:some_elements_hall}). 
    Then the lower bound \eqref{eq:X0bn-FnCr} of \cref{thm:LB2}, combined with \cref{Prop:free-lie-alphabetic} (which is allowed since $\{a,X_1\}$ is alphabetic and free) shows that $2^{-n} \|[a,b_n]\|_\B$ does not go to zero as $n \to +\infty$, which contradicts \eqref{eq:beta-2n} since $|a|+|b_n| = 2|a| + n$.

    \step{Second item}
    Assume that there exists $b \in \B$ such that $X_1<b$. We may choose $b$ such that $|b|$ is minimal (note that $|b| \geq 2$ since $X_0<X_1$). 
    In this case, by minimality of $b$, $\lambda(b) < \mu(b) \leq X_1$. 
    Let $a := \lambda(b)$. 
    By hypothesis, $a<X_1<b$, and since $r(a,X_1) = + \infty$, for all $n \in \N$, $\dad_{X_1}^n(a)<X_1<b$. This shows in particular that for all $n \in \N$, $(\dad_{X_1}^n(a),b) \in \mathcal B$, because $\lambda(b) = a < \dad_{X_1}^n(a)$. 
    Let $b_n := \ad_{X_1}^n(b) \in \mathcal B$.
    By \cref{p:ad-2}:
    \begin{equation}
        [a,b_n] = \sum_{k=0}^n \binom{n}{k} \ad_{X_1}^{n-k}([\dad_{X_1}^{k}(a),b]),
    \end{equation}
    and the previous discussion shows that for all $k,n \in \N$, $\ad_{X_1}^{n-k}([\dad_{X_1}^{k}(a),b]) \in \ev(\B) $, so that $\|[a,b_n]\|_\B = 2^n$, with $|a| + |b_n| = n + |a| + |b|$, which contradicts \eqref{eq:beta-2n}.

    \step{Third item}
    Let $a_0 \in \B\setminus\{X_1\}$ and assume that there exists $b\in \B$ with $b \neq X_1$ such that for all $n \in \N$, $\dad_{X_1}^n(a_0) <b$. 
    We may assume that $|b|$ is minimal (by assumption, $|b| \geq 2$). 
    By minimality of $b$ and since $\lambda(b) \neq X_1$, there exists $k \geq 0$ such that $\dad_{X_1}^{k}(a_0)\geq \lambda(b)$. 
    Let $k \in \N$ be minimal such that $\dad_{X_1}^{k}(a_0)\geq \lambda(b)$, and let $a := \dad_{X_1}^{k}(a_0)$. 
    Then, by construction, for all $n \in \N$ we have $\dad_{X_1}^n(a)<b$ and $(\dad_{X_1}^n(a),b) \in \B$. 
    Therefore, if we let $b_n := \dad_{X_1}^n(b)$, a computation similar to that of the second item using \cref{p:ad-2} shows that $\|[a,b_n]\|_\B = 2^n$, which is again a contradiction.
\end{proof}

This lemma shows that if we intend to exhibit a Hall set $\B$ for which the growth of the structure constants has the magnitude of the Fibonacci sequence, then it is necessary that $X_1$ be maximal in $\B$ (which excludes length-compatible Hall sets), but also that the third property (which is reminiscent of the Archimedean property) of the lemma holds. 
For example, in the case of the Lyndon basis over $\{X_0<X_1\}$, the element $X_1$ is maximal but the third property does not hold (for example, $\dad_{X_1}^n((X_0,(X_0,X_1))) < (X_0,X_1)$ for all $n \in \N$).

An example of a known Hall set satisfying all three conditions is the Spitzer-Foata basis (as named and described in \cite[Chapter~I, Section~5]{MR516004}; see also \cite{zbMATH03645309}).
The order defining this Hall set is compatible with the increasing order on the ratio $s(b) := n_1(b) / n_0(b) \in [0;+\infty]$, where $n_i(b)$ denotes the number of occurrences of $X_i$ in~$b$.
In particular, $X_1$ is maximal and, for every $a, b \in \B \setminus \{X_1\}$, $s(\dad_{X_1}^n(a)) \to + \infty$ as $n \to +\infty$, so it is strictly greater than $s(b)$ for $n$ large enough, and the third property is satisfied.
Unfortunately, numerical simulations we conducted indicate that this basis does not satisfy $\beta_{n+2}(\B) = F_n$, which motivated our construction of a new basis, explained in \cref{s:fibo-def}.
We do not know whether the basis we construct is the unique one satisfying $\beta_{n+2}(\B) = F_n$ or if others exist.

\subsection{Definition of our minimal Hall set} \label{s:fibo-def}

We will base our construction on the following ``$\lambda,\mu$''-lexicographic order, where one indeterminate is considered maximal.

\begin{definition} \label{def:fibo-order}
    Let $X = \{X_0,X_1\}$ and consider the following order on $\Br(X)$:
    \begin{itemize}
        \item $X_0$ is minimal and $X_1$ is maximal, i.e., $X_0 < \Br(X) \setminus X < X_1$,
        \item for $t_1,t_2 \in \Br(X) \setminus X$, $t_1<t_2$ iff either $\lambda(t_1)<\lambda(t_2)$, or $\lambda(t_1) = \lambda(t_2)$ and $\mu(t_1)<\mu(t_2)$.
    \end{itemize}
\end{definition}

\begin{lemma} \label{lem:fibo-set}
    There exists a unique Hall set $\B \subset \Br(X)$ associated with the order of \cref{def:fibo-order}.
\end{lemma}

\begin{proof}
    This order is not a Hall order on $\Br(X)$ because there exists $t \in \Br(X)$ such that $\lambda(t)>t$, for instance $t=(X_1,X_0)$. But we can construct $\B$ by applying \cref{Lem:G} to the set $G:=\{b\in\Br(X);X_1 \notin \Lambda(b) \} \cup \{X_1\}$ (see \cref{def:Lambda}). Clearly, $G$ is a $\lambda$-stable subset of $\Br(X)$, containing $X$, endowed with a total order and for every $b_1<b_2 \in G$, $(b_1,b_2) \in G$. 
    It remains to prove that $<$ is a Hall order on $G$. 
    By contradiction, let $b \in G \setminus X$ of minimal length such that $\lambda(b) > b$.
    Since $\lambda(b) > b$, $\lambda(b) \neq X_0$.
    Since $b \in G$, $\lambda(b) \neq X_1$.
    Hence $\lambda(b) \in G \setminus X$ and satisfies $\lambda(\lambda(b)) < \lambda(b)$ since $b$ was of minimal length.
    By definition of the order, this implies $\lambda(b) < b$.
\end{proof}

In the sequel of this section, $\B$ denotes the Hall set constructed in \cref{lem:fibo-set}.

\begin{remark}
By definition, $X_1 = \max \B$ (thus $r(a,X_1) = +\infty$ for all $a \in \B$), and if $a,b \in \B$ with $a,b \neq X_1$, then a straightforward induction on $|b|$ shows that there exists $n \in \N$ such that $\dad_{X_1}^n (a) >b$. Therefore, all the necessary conditions of \cref{lem:CN-fibonacci} hold for $\B$.
\end{remark}

\subsection{Elementary structural properties}
\label{s:fibo-elementary}

We state in the following lemmas useful properties of this Hall set.

\begin{lemma} \label{sx:order}
    Let $a < b \in \B$ with $b \neq X_1$ and $\lambda(b) < a$.
    For every $n < n' \in \N$ and $p < p' \in \N$,
    \begin{equation}
        \ad_a^n (b) < \ad_a^{n'}(b) < \ad_a^{p'}(X_1) < \ad_a^{p}(X_1).
    \end{equation}
\end{lemma}

\begin{proof}
    These inequalities stem from \cref{def:fibo-order}.
    First, since $\lambda(b) < a$, $b < (a, \ad_{a}^{n'-n-1}(b))$, which yields the first inequality by applying $\ad^n_a$ to both sides.
    Second, since $X_1$ is maximal, $\ad^{p'-p}_a(b)(X_1) < X_1$, which yields the third inequality by applying $\ad^p_a$ to both sides.
    Eventually, when $n' \geq p'$, the middle inequality holds because $\ad^{n'-p'}_a(b) < X_1$ and, when $n' < p'$, it holds because $b < \ad^{p'-n'}_a(X_1)$ since $\lambda(b) < a$ by assumption.
\end{proof}

\begin{lemma} \label{sx:b-mub}
    Let $b \in \B \setminus X$.
    Then $b < \mu(b)$ iff $b=\ad_{\lambda(b)}^m(X_1)$ for some $m \in \N^*$.
\end{lemma}

\begin{proof}
    Let $b \in \B \setminus X$.
    There exists $m \in \N^*$ such that $b = \ad_{\lambda(b)}^m(v)$ with $v = X_1$ or $v \in \B \setminus X$ with $\lambda(v) < \lambda(b)$.
    Since $\mu(b) = \ad^{m-1}_{\lambda(b)}(v)$,
    \cref{sx:order} proves that $b < \mu(b)$ iff $v = X_1$.
\end{proof}

\begin{lemma} \label{sx:a<mub}
    Let $a \in \B$ and $b \in \B \setminus X$ with $\mu(b) \neq X_1$ such that $a < \mu(b)$. Then $a \leq b$.
\end{lemma}

\begin{proof}
    By contradiction, if $b < a < \mu(b)$, then $\lambda(b) \leq \lambda(a) \leq \lambda(\mu(b)) \leq \lambda(b)$ so there exists $m \in \N^*$ such that $a = \ad^m_{\lambda(b)}(v)$ with $v = X_1$ or $v \in \B \setminus X$ with $\lambda(v) < \lambda(b)$.
    Since $b < \mu(b)$, by \cref{sx:b-mub}, $b = \ad^n_{\lambda(b)}(X_1)$ for some $n \in \N^*$.
    Moreover, $n \geq 2$ because $\mu(b) \neq X_1$ by assumption.
    Using \cref{sx:order}, $b < a$ implies $v = X_1$ and $b < a < \mu(b)$ implies $n>m>n-1$.
\end{proof}

\begin{lemma}\label{lem:theta<=2}
    Let $a<b\in\B$ with $b \notin X$.
    Then $\theta_a(b) \leq 2$ if both of the following conditions hold
    \begin{itemize}
        \item either $\lambda(b) = X_0$ or $\lambda^2(b) \leq a$,
        \item either $\mu(b)=X_1$ or $\lambda(\mu(b)) \leq a$.
    \end{itemize}
\end{lemma}

\begin{proof}
    If $\lambda(b) \leq a$ then $\theta_a(b)=1$. 
    Otherwise $a<\lambda(b)<\mu(b)$ so $\theta_a(b)=\theta_a(\lambda(b))+\theta_a(\mu(b))$. The first (resp.\ second) condition gives $\theta_{a}(\lambda(b))=1$ (resp.\ $\theta_a(\mu(b))=1$).
\end{proof}

\subsection{Structural properties of supporting basis elements} 
\label{s:fibo-thm-structure}

We prove the following structural properties on elements of the support of $[a,b]$.
These properties are of independent interest for the understanding of $\B$ and will play a key role in estimating the size of its structure constants, since they allow to stop the decomposition algorithm much earlier than in the general case.

\begin{theorem} \label{th:fibo-structure-new}
    For every $a<b \in \B$, each $c \in \supp [a,b]$ satisfies: $c \notin X$,
    \begin{itemize}
        \item $a \leq \lambda(c) < b$ where the first relation is an equality if and only if $(a,b) \in \B$,
        \item and, when $b \neq X_1$, either $c \leq \lambda(b)$ or $\theta_{\lambda(b)}(c) \leq 2$. 
    \end{itemize}
\end{theorem}
\begin{proof}
    The relation between $a$ and $\lambda(c)$ is proved in \cref{thm:Reutenauer} for any Hall set (see \eqref{eq:a<c'}). 
    Thus, we only prove the other relations.

    We proceed, as in \cref{p:rec-theta}, by induction on $\theta_a(b)$, by applying the classical rewriting scheme.
    We refer to the proof of \cref{p:rec-theta} for a detailed justification of why the induction on $\theta_a(b)$ is legitimate in this setting.

    If $\theta_a(b) = 1$, then $c=(a,b) \in \B$ and satisfies $\lambda(c)=a<b$.
    If $b \neq X_1$, then $\lambda(b) \leq a < c$.
    If $a = X_0$, then $\theta_{\lambda(b)}(c) = 1$.
    Otherwise, by \cref{def:fibo-order}, $\lambda(a) \leq \lambda(b)$, so, by \cref{lem:theta<=2}, $\theta_{\lambda(b)}(c) \leq 2$.

    From now on, we consider $\theta_a(b) > 1$ and we assume the result proved for strictly smaller values. 
    Then $b \neq X_1$.
    Moreover, $\lambda(b) \notin X$, because $X_0 \leq a < \lambda(b)<\mu(b)\leq X_1$, thus $\lambda^2(b)$ and $\mu(\lambda(b))$ are well defined. 
    By Jacobi's identity
    \begin{equation} \label{eq:jacobi-ab-alambdabu-new}
        [a,b]=[[a,\lambda(b)],\mu(b)]+[\lambda(b),[a,\mu(b)]].
    \end{equation}
    
    \step{A preliminary remark} 
    When $\mu(b) \neq X_1$, for each $c \in \supp[a,b]$, $\lambda(c) < \mu(b)$ implies $\lambda(c) < b$. 
    Indeed, by \cref{sx:a<mub}, $\lambda(c) < \mu(b)$ implies $\lambda(c) \leq b$ and, by \cref{lem:theta-dec2/support-absurde}, one cannot have $\lambda(c) = b$.

    \step{Study of the first term}
    By induction,
    \begin{equation}
        [a,\lambda(b)]=\sum \gamma_d \ev(d), 
        \quad \text{ where } 
        \quad a \leq \lambda(d)<\lambda(b)  
        \text{ and } 
        \big(\text{either } d \leq \lambda^2(b) \text{ or } \theta_{\lambda^2(b)}(d) \leq 2 \big).
    \end{equation}
    Let $d \in \supp [a,\lambda(b)]$. Then $a < d < b$    because $a \leq \lambda(d)<d$ and $\lambda(d)<\lambda(b)$.
    
    \begin{itemize}
        \item \emph{Case $\mu(b) = X_1$}. 
        Then $c := (d, \mu(b))=(d,X_1) \in \B$ satisfies $\lambda(c) = d < b$.
        If $\lambda(b) < c$, since $\lambda(d)<\lambda(b)$, then, by \cref{lem:theta<=2}, $\theta_{\lambda(b)}(c) \leq 2$. 
        \item \emph{Case $\mu(b) \neq X_1$}, so $\mu(b) \in \B \setminus X$ and $\lambda(\mu(b))\leq \lambda(b)$.
        \begin{itemize}
            \item \emph{Subcase $\mu(b) < d$}. 
            Then $c:=(\mu(b),d) \in \B$ because $\lambda(d)<\lambda(b)<\mu(b)$. 
            Moreover, $\lambda(c)=\mu(b) < d < b$.
            Then $\lambda(b) < \mu(b) < c$ and, since $\lambda(\mu(b)) \leq \lambda(b)$ and $\lambda(d) < \lambda(b)$, then, by \cref{lem:theta<=2}, $\theta_{\lambda(b)}(c) \leq 2$.
            
            \item \emph{Subcase $d < \mu(b)$}.
            By induction, $[d,\mu(b)]=\sum \alpha_c \ev(c)$ where $\lambda(c)<\mu(b)$ and, either $c \leq \lambda(\mu(b))$, or $\theta_{\lambda(\mu(b))}(c) \leq 2$.
            By the preliminary remark, $\lambda(c)<b$.
            If $\lambda(b) < c$, by \cref{p:theta-dec}, $\theta_{\lambda(b)}(c) \leq \theta_{\lambda(\mu(b))}(c) \leq 2$. 
        \end{itemize}
    \end{itemize}
    
    \step{Study of the second term}
    
    \begin{itemize}
        \item \emph{Case $\mu(b)= X_1$.} 
        Then $d:=(a,\mu(b))=(a,X_1)\in\B$ and $a<d$ by the Hall order property.
        \begin{itemize}
            \item If $d<\lambda(b)$, by induction, $[d,\lambda(b)]=\sum \alpha_c \ev(c)$ where $\lambda(c)<\lambda(b)$ and either $c \leq \lambda^2(b)$ or $\theta_{\lambda^2(b)}(c) \leq 2$.
            Thus, $\lambda(c)<b$ because $\lambda(b)<b$.
            If $\lambda(b) < c$, by \cref{p:theta-dec}, $\theta_{\lambda(b)}(c) \leq \theta_{\lambda^2(b)}(c) \leq 2$.
            
            \item If $\lambda(b)<d$ then $c:=(\lambda(b),d)=(\lambda(b),(a,X_1))\in\B$ because $a < \lambda(b)$. Moreover $\lambda(c)=\lambda(b)<b$, $\lambda(b) < c$ and $\theta_{\lambda(b)}(c) = 1$.
        \end{itemize}
    
    \item \emph{Case $\mu(b)\neq X_1$.}
    By induction, 
    \begin{equation} \label{amub-d-new}
        [a,\mu(b)] = \sum \gamma_{d} \ev(d) 
        \text{ where } a \leq \lambda(d)<\mu(b) \text{ and } 
        \big(\text{either } d \leq \lambda(\mu(b)) \text{ or } \theta_{\lambda(\mu(b))}(d) \leq 2 \big).
    \end{equation}    
    Let $d \in \supp [a,\mu(b)]$. Then $a < d$ because $a \leq \lambda(d) < d$.
    
    \begin{itemize}
        \item \emph{Subcase $d < \lambda(b)$}.
        By induction,
        $[d,\lambda(b)]=\sum \alpha_c \ev(c)$ where $\lambda(c)<\lambda(b)$ and (either $c \leq \lambda^2(b)$ or $\theta_{\lambda^2(b)}(c) \leq 2$), which implies the expected relations, because the order is a Hall order on $\B$ (so $\lambda(b) < b$ and $\lambda^2(b) < \lambda(b)$) and thanks to \cref{p:theta-dec}.
    
        \item \emph{Subcase $\lambda(b) < d$}.
        \begin{itemize}
            \item \emph{Subsubcase $\lambda(d) \leq \lambda(b)$}.
            Then $c := (\lambda(b), d) \in \B$ satisfies $\lambda(c) = \lambda(b) < b$ by the Hall order property and $\lambda(b) < c$ and $\theta_{\lambda(b)}(c) = 1$.
            
            \item \emph{Subsubcase $\lambda(b) < d$}.
            This implies $\lambda(\mu(b)) < d$ and, by \eqref{amub-d-new} and \cref{p:theta-dec}, $\theta_{\lambda(b)}(d) \leq \theta_{\lambda(\mu(b))}(d) \leq 2$.
            Thus, for each $c \in \supp [\lambda(b), d]$, $\lambda(b) < \lambda(c)$ by \eqref{eq:a<c'}, and $\theta_{\lambda(b)}(c) = \theta_{\lambda(b)}(d) \leq 2$ by \cref{lem:theta-dec2/support-absurde}.
        
            To get the conclusion, we only need to check that $\lambda(c)<b$. 
            By \cref{thm:en1}, $c \in \e(\Tr \lbag \lambda(b) < \lambda(d) < \mu(d) \rbag)$.
            By the Hall order property, $c_1 := \lambda(b) < b$.
            By \eqref{amub-d-new}, $c_2 := \lambda(d) < \mu(b)$.
            Hence, $c_3 := (\lambda(b), \lambda(d)) < b$ by \cref{def:fibo-order}.
            Since $\lambda(c) < \mu(c)$ and one of them is of one of the $c_i$'s, we obtain $\lambda(c) < b$ or $\lambda(c) < \mu(b)$, hence $\lambda(c) < b$ by the preliminary remark.
            \qedhere
        \end{itemize}
    \end{itemize}
    \end{itemize}
\end{proof}

\subsection{A first estimate for right-nested trees}
\label{s:fibo-right-quick}

In this paragraph, we prove a bound on $\|[a,b]\|_\B$ depending only on $\theta_a(b)$ in the particular case where $\lambda^2 (b) \leq a$.
We tighten this bound and remove the hypothesis in \cref{s:fibo-theta}.
We start here this lighter version since it is sufficient for the proof of \cref{thm:fibo-easy}.

\begin{lemma} \label{lem:tab-right}
    Let $a < b \in \B$ such that $(a,b) \notin \B$ and $\lambda^2(b) \leq a$. 
    Then $\rf{a}{b}$ is right-nested, in the sense that it has the following form
    \begin{equation} \label{eq:tab-right}
        \rf{a}{b} = \langle \lambda(b), \langle \lambda(\mu(b)), \langle \dotsc \langle \lambda(\mu^{\theta-2}(b)), \mu^{\theta-1}(b) \rangle \rangle \dotsb \rangle
    \end{equation}
    where $\theta = \theta_a(b)$ and, in particular, only the last leaf, $\mu^{\theta-1}(b)$, can be equal to $X_1$.
\end{lemma}

\begin{proof}
    The assumption $(a,b) \notin \B$ implies that $b \neq X_1$, so $\lambda(b)$ makes sense and that $\lambda(b) \neq X_0$, so $\lambda^2(b)$ is also well-defined.
    We proceed by induction on $\theta_a(b)$.
    When $\theta_a(b) = 2$, it is true in any Hall set that $\rf{a}{b} = \langle \lambda(b), \mu(b) \rangle$.
    When $\theta_a(b) \geq 3$, since $\lambda^2(b) \leq a < \lambda(b)$, by definition, $\rf{a}{b} = \langle \lambda(b), \rf{a}{\mu(b)} \rangle$.
    The result follows by induction because $\lambda^2(\mu(b)) \leq \lambda^2(b) \leq a$.
    (The first inequality stems from \cref{def:fibo-order} and from the fact that $\lambda(\mu(b)) \leq \lambda(b)$, since $b \in \B$).
\end{proof}

We start by a definition which allows to distinguish, within $\|[a,b]\|_\B$, the part of the norm relative to the different categories of supporting basis elements.

\begin{definition}
    Let $a < b \in \B$ with $b \neq X_1$.
    Writing $[a,b] = \sum \alpha_c \ev(c)$ where the sum ranges over $\supp [a,b]$ and $\alpha_c \in \K$, we define, for $i \in \intset{1,2}$,
    \begin{equation}
        N_i(a,b) := \sum_{\theta_{\lambda(b)}(c) = i} |\alpha_c|.
    \end{equation}
\end{definition}

\begin{lemma} \label{lem:l2ba-supp}
    Let $a < b \in \B$ such that $(a,b) \notin \B$ and $\lambda^2(b) \leq a$.
    For each $c \in \supp [a,b]$, $\lambda(b) < c$.
    Moreover
    \begin{equation} \label{eq:size-n1n2}
        \|[a,b]\|_\B = N_1(a,b) + N_2(a,b).
    \end{equation}
\end{lemma}

\begin{proof}
    Let $c \in \supp [a,b]$.
    Since $(a,b) \notin \B$, by \cref{eq:a<c'}, $a < \lambda(c)$.
    Hence $\lambda^2(b) < \lambda(c)$, so $\lambda(b) < c$ by \cref{def:fibo-order}.
    Then \eqref{eq:size-n1n2} follows from \cref{th:fibo-structure-new} since the case $c \leq \lambda(b)$ is excluded.
\end{proof}

\begin{lemma} \label{lem:size-n2-init}
    Let $a < b \in \B$ with $\theta_a(b) = 2$.
    Then $\| [a, b] \|_\B \leq 2$ and $N_2(a,b) \leq 1$.
\end{lemma}

\begin{proof}
    Since $\theta_a(b) = 2$, $a < \lambda(b)$, $(a, \lambda(b)) \in \B$ and $(a, \mu(b)) \in \B$. 
    By Jacobi's identity
    \begin{equation}
        [a,b]=[(a,\lambda(b)),\mu(b)]+[\lambda(b),(a,\mu(b))].
    \end{equation}
    
    \step{First term}
    \begin{itemize}
        \item If $(a,\lambda(b))<\mu(b)$ then $c_1 :=((a,\lambda(b)),\mu(b))\in\B$ because $\lambda(\mu(b))\leq a < (a,\lambda(b))$.  
        \item If $\mu(b) < (a,\lambda(b))$ then $c_1:=(\mu(b),(a,\lambda(b)))\in\B$ because $a<\mu(b)$. 
    \end{itemize}
    
    \step{Second term}
    \begin{itemize}
        \item If $\lambda(b) < (a, \mu(b))$, then $c_2 := (\lambda(b), (a, \mu(b))) \in \B$ because $a < \lambda(b)$.
        
        \item If $(a, \mu(b)) < \lambda(b)$, then $c_2 := ((a, \mu(b)), \lambda(b)) \in \B$ because $\lambda^2(b) \leq a < (a,\mu(b))$.
    \end{itemize}
    In particular, $\theta_{\lambda(b)}(c_2) = 1$ because $\lambda(c_2) = \min \{ \lambda(b), (a, \mu(b)) \} \leq \lambda(b)$. 
    So $N_2(a,b) \leq 1$.
\end{proof}

We start with an easy case where $\lambda^2(b) < a$, leading to a straightforward geometric bound.

\begin{proposition} \label{p:lambda2b<a-easy}
	Let $a < b \in \B$ such that $(a,b) \notin \B$ and $\lambda^2(b) < a$. Then
	\begin{equation} \label{eq:l2b<a-n0n1n2}
	    \| [a,b] \|_\B \leq 2^{\theta-1}
	    \quad \text{and} \quad 
	    N_2(a,b) \leq 2^{\theta-1}-1,
	    \quad \text{where} \quad \theta = \theta_a(b).
	\end{equation}
\end{proposition}

\begin{proof}
    By \cref{def:fibo-order}, $\lambda^2(b) < a$ implies that $\lambda(b) < (a, \lambda(b))$.
    We proceed by induction on $\theta_a(b) \geq 2$.
    Initialization is proved by \cref{lem:size-n2-init}. 
    We perform the inductive step.
    Assume $\theta_a(b) \geq 3$. 
    Then $\theta_a(\mu(b))=\theta_a(b)-1 \geq 2$ thus $a<\lambda(\mu(b))$, $\lambda(\mu(b)) \notin X$ and $\lambda^2(\mu(b))$ is well defined.
    Using Jacobi's identity,
	\begin{equation} \label{eq:ab-jac-new}
		[a,b] = [(a, \lambda(b)), \mu(b)] + [\lambda(b), [a, \mu(b)]].
	\end{equation}
	
	\step{First term of \eqref{eq:ab-jac-new}} 
	We have $(a,\lambda(b))<\mu(b)$ by \cref{def:fibo-order} because $a<\lambda(\mu(b))$. Then $c_1:=((a,\lambda(b)),\mu(b))\in\B$ because, since $b\in\B$, $\lambda(\mu(b))\leq \lambda(b) <(a,\lambda(b))$. Moreover $\theta_{\lambda(b)}(c_1)=2$ because $\lambda(b)<\lambda(c_1)=(a,\lambda(b))$.
	
    \step{Second term of \eqref{eq:ab-jac-new}}
    The induction hypothesis applies to $[a, \mu(b)]$ because $a < \mu(b)$, $\lambda^2(\mu(b)) \leq \lambda^2(b) < a$ by \cref{def:fibo-order} and assumption, and $\theta_a(\mu(b))=\theta_a(b)-1 \geq 2$. 
    Hence $[a, \mu(b)] = \sum \alpha_d \ev(d)$ with the claimed size estimates.
    Let $d \in \supp [a, \mu(b)]$.
    \begin{itemize}
        \item If $\theta_{\lambda(b)}(d) = 1$, then $c := (\lambda(b), d) \in \B$ and $\theta_{\lambda(b)}(c) = 1$.
        
        \item If $\theta_{\lambda(b)}(d) = 2$, by \cref{thm:en1}, $\|[\lambda(b), d]\|_\B \leq 2$ and, by \cref{lem:theta-dec2/support-absurde}, for each $c \in \supp[\lambda(b), d]$, $\theta_{\lambda(b)}(c) = 2$.
    \end{itemize}
    Moreover, by \cref{p:theta-dec}, $\theta_{\lambda(b)}(d) \leq \theta_{\lambda(\mu(b))}(d)$.
    
    \step{Conclusion}
	This analysis proves $N_2(a,b) \leq 1 + 2 N_2(a, \mu(b))$ and that
	\begin{equation} \label{eq:ab-n1-n2-rec1}
	    \|[a,b]\|_\B 
        \leq 1 +  N_1(a, \mu(b)) + 2 N_2(a, \mu(b))
        = 1 + \| [a, \mu(b)] \|_\B + N_2(a, \mu(b)),
	\end{equation}
    which concludes the proof of \eqref{eq:l2b<a-n0n1n2} by induction, since $\theta_a(b) = \theta_a(\mu(b)) + 1$.
\end{proof}

\begin{proposition} \label{p:lambda2b<=a-easy}
	Let $a < b \in \B$ such that $(a,b) \notin \B$ and $\lambda^2(b) \leq a$. Then
	\begin{equation} \label{eq:l2b<=a-n0n1n2}
	    \| [a,b] \|_\B \leq (\theta-2) 2^{\theta-2} + \theta
	    \quad \text{and} \quad 
	    N_2(a,b) \leq (\theta-2) 2^{\theta-2}+1,
	    \quad \text{where} \quad \theta = \theta_a(b).
	\end{equation}
\end{proposition}

\begin{proof}
    We proceed by induction on $\theta_a(b) \geq 2$.
    Initialization is proved by \cref{lem:size-n2-init}. 
    We perform the inductive step.
    Assume $\theta_a(b) \geq 3$. 
    Then $\theta_a(\mu(b))=\theta_a(b)-1 \geq 2$ thus $a<\lambda(\mu(b))$, $\lambda(\mu(b)) \notin X$ and $\lambda^2(\mu(b))$ is well defined.
    We study both terms of Jacobi's identity \eqref{eq:ab-jac-new}.
	
	\step{First term}
	We have $(a,\lambda(b))<\mu(b)$ by \cref{def:fibo-order} because $a<\lambda(\mu(b))$.
    \begin{itemize}
    \item If $\lambda(\mu(b))\leq(a,\lambda(b))$ then $c:=((a,\lambda(b)),\mu(b))\in\B$ and $\theta_{\lambda(b)}(c) \in \{ 1, 2 \}$.
    
    \item If $(a,\lambda(b)) < \lambda(\mu(b))$, we apply  \cref{p:lambda2b<a-easy} to the bracket $[(a,\lambda(b)),\mu(b)]$.
    Indeed $\lambda^2(\mu(b))\leq \lambda^2(b)\leq a < (a,\lambda(b))$.
    By \cref{p:theta-dec}, $\theta_{(a,\lambda(b))}(\mu(b))\leq\theta_{a}(\mu(b)) = \theta_a(b) - 1$. 
    Therefore \cref{p:lambda2b<a-easy} gives
    $\|[(a,\lambda(b)),\mu(b)]\|_\B \leq 2^{\theta_a(b)-2}$ and $N_2((a, \lambda(b)), \mu(b)) \leq 2^{\theta_a(b)-2}-1$.
    Any $c\in\supp[(a,\lambda(b)),\mu(b)]$ satisfies $\theta_{\lambda(b)}(c) \leq \theta_{\lambda(\mu(b))}(c)$ by \cref{p:theta-dec}. 
    \end{itemize}

    \step{Second term} 
    By \cref{th:fibo-structure-new}, $[a,\mu(b)]=\sum \alpha_d \ev(d)$ where the sum ranges over $d\in\B$ such that $a<\lambda(d)<\mu(b)$, $\lambda(\mu(b)) < d$ and $\theta_{\lambda(\mu(b))}(d) \leq 2$ using \cref{lem:l2ba-supp}.
    Since $\lambda^2(\mu(b)) \leq \lambda^2(b) \leq a$, the induction hypothesis applies to $[a, \mu(b)]$.
    Let $d \in \supp [a, \mu(b)]$.
    By \cref{p:theta-dec}, $\theta_{\lambda(b)}(d) \leq \theta_{\lambda(\mu(b))}(d)$.
    Moreover, since $\lambda^2(b) \leq a < \lambda(d)$ (by \cref{eq:a<c'}), $\lambda(b) < d$ by \cref{def:fibo-order}.
    \begin{itemize}
    \item If $\theta_{\lambda(b)}(d)=1$, $c:=(\lambda(b),d)\in\B$ and $\theta_{\lambda(b)}(c)=1$.
    \item If $\theta_{\lambda(b)}(d)=2$, then $\|[\lambda(b),d]\|_\B\leq 2$.
    \end{itemize}
    
    \step{Conclusion}
	This analysis proves $N_2(a,b) \leq 2^{\theta_a(b)-2} - 1 + 2 N_2(a, \mu(b))$ and that
	\begin{equation} \label{eq:ab-n1-n2-rec2}
	    \|[a,b]\|_\B 
        \leq 2^{\theta_a(b)-2} +  N_1(a, \mu(b)) + 2 N_2(a, \mu(b))
        = 2^{\theta_a(b)-2} + \| [a, \mu(b)] \|_\B + N_2(a, \mu(b)),
	\end{equation}
    which concludes the proof of \eqref{eq:l2b<=a-n0n1n2} by induction, since $\theta_a(b) = \theta_a(\mu(b)) + 1$.
\end{proof}

\subsection{General estimate}
\label{s:fibo-general}

For $a < b \in \B$, we decompose $\theta_a(b) = n_a(b) + \nu_a(b)$ where $n_a(b)$ (resp.\ $\nu_a(b)$) denotes the number of leaves of $T_a(b)$ different from $X_1$ (resp.\ equal to $X_1$).
We start by stating elementary monotony properties of these quantities.
We then prove a bound on $\|[a,b]\|_\B$ involving these quantities in \cref{th:fibo-size-new}, which yields \cref{cor:fibo-light}, a slightly weaker version of our main goal, \cref{thm:fibo-easy}.

\begin{lemma}\label{Lem:Surcoupe_large}
Let $a \leq h < b \in \B$. Then $n_h(b) \leq n_a(b)$ and $\nu_h(b) \leq \nu_a(b)$.
\end{lemma}

\begin{proof}
    We proceed by induction on $\theta_a(b)$.
    
    \step{Initialization for $\theta_a(b) = 1$} 
    Then either $b=X_1$ or $\lambda(b) \leq a \leq h$. 
    In both cases, $T_a(b)=b=T_h(b)$ so $(n_a(b),\nu_a(b))=(n_h(b),\nu_h(b))$. 
    
    \step{Induction for $\theta_a(b) \geq 2$}
    Then $b \neq X_1$ and $n_a(b) \geq 1$.
    \begin{itemize}
        \item \emph{Case $\lambda(b) \leq h$}. 
        Then $n_h(b) = 1 \leq n_a(b)$ and $\nu_h(b) = 0 \leq \nu_a(b)$.
        \item \emph{Case $h<\lambda(b)$}. 
        By the induction hypothesis, $n_h(\lambda(b)) \leq n_a(\lambda(b))$ and $n_h(\mu(b)) \leq n_h(\mu(b))$. Therefore, $n_h(b) \leq n_a(b)$. The same argument shows that $\nu_h(b) \leq \nu_a(b)$.
        \qedhere
    \end{itemize}
\end{proof}

Under additional assumptions, we now prove that their weighted sum is strictly decreasing. 
This remark plays a key role in the proof of \cref{th:fibo-size-new} and even more in its refined version for particular brackets stated in \cref{th:X0b-fibo}.

\begin{lemma} \label{Cor:surcoupe_strict}
Let $a < h < b \in \B$.	Assume that either $h = (a, X_1)$ or $a < \lambda(h)$. Then
\begin{equation} \label{ineg_surcoup}
2n_h(b)+\nu_h(b) \leq n_a(b) + \nu_a(b).
\end{equation}
In particular, if $b \neq X_1$, $2n_h(b) + \nu_h(b) < 2n_a(b) + \nu_a(b)$.
\end{lemma}

\begin{proof}
    We proceed by induction on $\theta_a(b)$.
    
    \step{Initialization for $\theta_a(b) = 1$} 
    Then either $b=X_1$ or $\lambda(b)\leq a$.
    If $b = X_1$, $n_a(b)=n_h(b)=0$ and $\nu_a(b)=\nu_h(b)=1$, so \eqref{ineg_surcoup} holds.
    Let us prove that the case $b \neq X_1$ with $\lambda(b) \leq a$ cannot happen.
    We deduce from $h<b$ that $\lambda(h) \leq \lambda(b)$. 
    In the case $h=(a,X_1)$, this relation leads to $\lambda(b)=a$ and $b=(a,\mu(b)) \leq (a,X_1)=h$, which is a contradiction. 
    In the case $a<\lambda(h)$, this relation leads to $a<\lambda(h)\leq\lambda(b)\leq a$, also a contradiction.

    \step{Induction for $\theta_a(b) \geq 2$}
    Then $n_a(b) \geq 1$.
    \begin{itemize}
        \item  \emph{Case $\lambda(b) \leq h$.} Then $n_h(b)=1$, $\nu_h(b)=0$ thus \eqref{ineg_surcoup} holds because the right-hand side is equal to $\theta_a(b)\geq 2$.
        
        \item \emph{Case $h < \lambda(b)$.} 
        By the induction hypothesis
        $2n_h(\lambda(b))+\nu_h(\lambda(b)) \leq n_a(\lambda(b)) + \nu_a(\lambda(b))$ and
        $2n_h(\mu(b))+\nu_h(\mu(b)) \leq n_a(\mu(b)) + \nu_a(\mu(b))$
        and \eqref{ineg_surcoup} follows by summing these two inequalities.
        \qedhere
    \end{itemize}
\end{proof}

\begin{theorem} \label{th:fibo-size-new}
    For every $a<b \in \B$, 
    \begin{equation} \label{eq:ab-fibo}
        \| [a, b] \|_\B \leq F_{2n_a(b)+\nu_a(b)}.
    \end{equation}
\end{theorem}

\begin{proof}
We proceed, as in \cref{p:rec-theta}, by induction on $\theta_a(b)$, by applying the classical rewriting scheme.
    We refer to the proof of \cref{p:rec-theta} for a detailed justification of why the induction on $\theta_a(b)$ is legitimate in this setting.
    
    \step{Case 1: $\theta_a(b) = 1$ i.e.\ either $b=X_1$ or $\lambda(b)\leq a$}
    Then $\| [a, b] \|_\B = 1$, which is the value in the right-hand side of~\eqref{eq:ab-fibo} both when $b = X_1$ and when $n_a(b) = 1$, $\nu_a(b) = 0$ since $F_1 = F_2 = 1$.
    
    \medskip
    
    From now on, we assume $\theta_a(b) > 1$, so $b \notin X$ and $a<\lambda(b)$. Thus $\lambda(b) \notin X$.
    
    \step{Case 2: $\theta_a(\lambda(b))=1$ i.e.\ $\lambda^2(b)\leq a < \lambda(b)$} 
    We apply \cref{p:lambda2b<=a-easy}.
    Combining \eqref{eq:l2b<=a-n0n1n2} and \eqref{eq:n2n-light} yields $\| [a,b] \|_\B \leq F_{2\theta_a(b)-1}$.
    Moreover $2\theta_a(b)-1 = 2n_a(b) + \nu_a(b) + (\nu_a(b)-1)$.
    By \cref{lem:tab-right}, $\nu_a(b) \in \{ 0, 1 \}$, so $F_{2\theta_a(b)-1} \leq F_{2n_a(b)+\nu_a(b)}$, and  \eqref{eq:ab-fibo} holds.
    
    \medskip
    
    From now on, we also assume $a < \lambda^2(b)$. 
    Hence $\theta_a(\lambda(b)) > 1$.
    The proof relies on the Jacobi decomposition \eqref{eq:jacobi-ab-alambdabu-new}.
    We introduce the notations $n_1 := n_a(\lambda(b))$, $\nu_1 := \nu_a(\lambda(b))$, $n_2 := n_a(\mu(b))$, $\nu_2 := \nu_a(\mu(b))$.
    In particular $n_1 + \nu_1 > 1$. By the induction hypothesis,
    \begin{align}
        \label{eq:alambda-f}
        \| [a, \lambda(b)] \|_\B 
        & \leq F_{2n_1+\nu_1}, \\
        \label{eq:amu-f}
        \| [a, \mu(b)] \|_\B 
        & \leq F_{2n_2+\nu_2}.
    \end{align}

    \step{Case 3: $\theta_a(\lambda(b))\geq 2$ i.e.\ $a<\lambda^2(b)$ and $\mu(b)=X_1$}
    Concerning the first term of \eqref{eq:jacobi-ab-alambdabu-new}, for each $d \in \supp [a, \lambda(b)]$, $(d, \mu(b)) = (d, X_1) \in \B$, so~\eqref{eq:alambda-f} yields $\|[[a,\lambda(b)], \mu(b)]\|_\B \leq F_{2n_1+\nu_1}$.
    Concerning the second term of \eqref{eq:jacobi-ab-alambdabu-new}, $d := (a, \mu(b)) = (a, X_1) \in \B$.
    If $\lambda(b) < d$, then $(\lambda(b), d) \in \B$ because $a < \lambda(b)$.
    If $d < \lambda(b)$, the induction hypothesis yields $\| [[a, \mu(b)], \lambda(b)] \|_\B \leq F_{2 n_1'+\nu_1'}$ where $n_1' := n_d(\lambda(b))$ and $\nu_1' := \nu_d(\lambda(b))$.
    By \cref{Cor:surcoupe_strict} (applied to $a \gets a$, $h \gets d= (a, X_1)$, $b \gets \lambda(b)$),  $2n_1'+\nu_1' < 2n_1+\nu_1$.
    Hence $\| [[a, \mu(b)], \lambda(b)] \|_\B \leq F_{2 n_1 + \nu_1 - 1}$.
    Summing our estimates for both terms of \eqref{eq:jacobi-ab-alambdabu-new} provides
    \begin{equation}
         \| [a, b] \|_\B \leq 
        F_{2n_1+\nu_1} + F_{2n_1+\nu_1-1}
        = F_{2n_1 + (\nu_1+1)},
    \end{equation}
    which proves \eqref{eq:ab-fibo} when $\mu(b) = X_1$ since $n_a(b) = n_1$ and $\nu_a(b) = \nu_1 + 1$.

    \step{Case 4: $\theta_a(\lambda(b))\geq 2$ i.e.\ $a<\lambda^2(b)$, $\mu(b) \notin X$ and $\theta_a(\mu(b))=1$}
    Since $\mu(b) \neq X_1$, $n_2 = 1$ and $\nu_2 = 0$.
    We consider both terms of \eqref{eq:jacobi-ab-alambdabu-new}.
    \begin{itemize}
        \item For the first term, let $d \in \supp [a, \lambda(b)]$.
        Since $(a, \lambda(b)) \notin \B$, hence, by \eqref{eq:a<c'}, $a < \lambda(d)$.
        Since $\theta_a(\mu(b)) = 1$, $\lambda(\mu(b)) \leq a < \lambda(d)$ so $\mu(b) < d$ automatically.
        Then $(\mu(b), d) \in \B$ since, by \cref{th:fibo-structure-new}, $\lambda(d) < \lambda(b) < \mu(b)$.
        So $\| [ [a, \lambda(b)], \mu(b)]\|_\B \leq F_{2n_1+\nu_1}$.
        
        \item For the second term, when $\theta_a(\mu(b)) = 1$, $d := (a, \mu(b)) \in \B$.
            
        If $\lambda(b) < d$, $(\lambda(b), d) \in \B$ since $a < \lambda(b)$, so $\| [\lambda(b), [a, \mu(b)]]\|_\B = 1$. 
        
        If $d < \lambda(b)$, by the induction hypothesis, $\| [d, \lambda(b)] \|_\B \leq F_{2 n_1' + \nu_1'}$ where $n_1' := n_d(\lambda(b))$ and $\nu_1' := \nu_d(\lambda(b))$. 
        By \cref{Lem:Surcoupe_large} (applied to $a \gets a$, $h \gets d$, $b \gets \lambda(b)$), $n_1' \leq n_1$ and $\nu_1' \leq \nu_1$. 
        So $\| [\lambda(b), [a, \mu(b)]]\|_\B \leq F_{2n_1+\nu_1}$.
    \end{itemize}
    As a conclusion, using \eqref{eq:fibo-2+2},
    \begin{equation}
        \| [a,b] \|_\B \leq 2 F_{2n_1+\nu_1}
        \leq F_{2 n_1 + \nu_1 + 2} = F_{2(n_1+1) + \nu_1},
    \end{equation}
    which proves \eqref{eq:ab-fibo} in this case.

    \step{Case 5: $\theta_a(\lambda(b))\geq 2$ i.e.\ $a<\lambda^2(b)$, $\mu(b) \notin X$ and $\theta_a(\mu(b))\geq 2$}

    \begin{itemize}
        \item We consider the first term of \eqref{eq:jacobi-ab-alambdabu-new}.
        Let $d \in \supp [a, \lambda(b)]$.
        Since $(a, \lambda(b)) \notin \B$, $a < \lambda(d)$, by~\eqref{eq:a<c'}.
        
        If $\mu(b) < d$, then $\| [\mu(b), d] \|_\B = 1$ since, by \cref{th:fibo-structure-new}, $\lambda(d) < \lambda(b) < \mu(b)$.
        
        If $d < \mu(b)$, by the induction hypothesis, $\| [d, \mu(b)] \|_\B \leq F_{2n_2'+\nu_2'}$ where $n_2' := n_d(\mu(b))$ and $\nu_2' := \nu_d(\mu(b))$.
        Since $a < \lambda(d)$, we can apply \cref{Cor:surcoupe_strict} to ($a \gets a$, $h \gets d$, $b \gets \mu(b) \neq X_1$), which proves that $2n_2'+\nu_2' \leq 2n_2+\nu_2-n_2$.
        
        Hence, summing over $d$ and using \eqref{eq:alambda-f}, this provides the estimate
        \begin{equation} \label{eq:almu-bis}
            \| [[a, \lambda(b)], \mu(b)] \|_\B \leq 
            F_{2 n_1+\nu_1} F_{2 n_2 + \nu_2 - n_2}.
        \end{equation}
        
        \item We consider the second term of \eqref{eq:jacobi-ab-alambdabu-new}.
        Let $d \in \supp [a, \mu(b)]$.
        Since we already covered the case $\theta_a(\mu(b)) = 1$, $(a,\mu(b))\notin \B$, hence, by \eqref{eq:a<c'}, $a < \lambda(d)$.
        
        If $\lambda(b) < d$, we know from the proof of \cref{th:fibo-structure-new} that $\| [\lambda(b), d] \|_\B \leq 2$.
        
        If $d < \lambda(b)$, by the induction hypothesis, $\| [d, \lambda(b)] \|_\B \leq F_{2n_1'+\nu_1'}$ where $n_1' = n_d(\lambda(b))$ and $\nu_1' = \nu_d(\lambda(b))$.
        Since $a < \lambda(d)$, we can apply \cref{Cor:surcoupe_strict} to ($a \gets a$, $h \gets d$, $b \gets \mu(b)$), which proves that $2n_1'+\nu_1' \leq 2n_1+\nu_1-1$.
        \end{itemize}
        
        Hence, summing over $d$ and using \eqref{eq:amu-f}, this provides the estimate
        \begin{equation} \label{eq:lamu-max3}
            \| [\lambda(b), [a, \mu(b)]] \|_\B \leq F_{2n_2+\nu_2} \max \{ 2, F_{2 n_1+\nu_1-n_1} \}.
        \end{equation}
        
        \step{Subcase 5.1: $n_1 = 1$ and $\nu_1 = 1$} 
        Then the maximum in the right-hand side of \eqref{eq:lamu-max3} is $2$ because $F_2 = 1$.
        Summing \eqref{eq:almu-bis} and \eqref{eq:lamu-max3}, and using \eqref{eq:fibo-2+2} yields
        \begin{equation} \label{eq:ab-ff-11}
            \|[a,b]\|_\B \leq F_3 F_{2n_2+\nu_2-1} + 2 F_{2n_2+\nu_2}
            = 2 F_{2n_2+\nu_2+1}
            \leq F_{2n_2+\nu_2+3},
        \end{equation}
        which proves \eqref{eq:ab-fibo} in this case since $2n_1+\nu_1 = 3$.
        
        \step{Subcase 5.2: $n_1 = 2$ and $\nu_1 = 0$}
        Then the maximum in the right-hand side of \eqref{eq:lamu-max3} is $2$ because $F_2 = 1$.
        Summing \eqref{eq:almu-bis} and \eqref{eq:lamu-max3}, and using \eqref{eq:fibo-cross} yields
        \begin{equation} \label{eq:ab-ff-20}
            \|[a,b]\|_\B \leq F_4 F_{2n_2+\nu_2-1} + 2 F_{2n_2+\nu_2}
            \leq F_{2n_2+\nu_2+3}
            = F_{2(n_1+n_2)+(\nu_1+\nu_2)-1},
        \end{equation}
        which proves \eqref{eq:ab-fibo} in this case.
        
        \step{Subcase 5.3: $n_1+\nu_1 > 2$}
        Then the maximum in the right-hand side of \eqref{eq:lamu-max3} is $F_{n_1+\nu_1}$ (with equality when $n_1 + \nu_1 = 3$).
        Summing \eqref{eq:almu-bis} and \eqref{eq:lamu-max3}, and using \eqref{eq:fibo-cross} yields,
        \begin{equation} \label{eq:ab-ff-n1n2}
            \|[a,b]\|_\B \leq F_{2n_1+\nu_1} F_{2n_2+\nu_2-1}
            + F_{2n_1+\nu_1-1} F_{2n_2+\nu_2}
            \leq F_{2(n_1+n_2) + (\nu_1+\nu_2) - 1},
        \end{equation}
        which proves \eqref{eq:ab-fibo} in this case.
\end{proof}

An immediate consequence is that \cref{th:fibo-size-new} allows to prove a slightly weaker version of \cref{thm:fibo-easy} (with an estimate by $F_{n-1}$ instead of $F_{n-2}$).

\begin{corollary} \label{cor:fibo-light}
    Let $a < b \in \B$ with $n := |a|+|b| \geq 3$.
    Then $\| [a,b] \|_\B \leq F_{n-1}$.
\end{corollary}

\begin{proof}
    Let $c \in L_\B(\rf{a}{b})$.
    So $c \neq X_0$.
    If $c \neq X_1$, $|c| \geq 2$.
    Thus, $|b| \geq 2 n_a(b) + \nu_a(b)$.
    Since $|a| \geq 1$, $n-1=|a|+|b|-1\geq 2n_a(b) + \nu_a(b)$ so \eqref{eq:ab-fibo} proves that $\|[a,b]\|_\B \leq F_{n-1}$.
\end{proof}

\begin{remark} \label{rk:cor-fibo}
    As soon as $|a| > 1$ or $|b| > 2 n_a(b) + \nu_a(b)$, estimate \eqref{eq:ab-fibo} even proves that $\|[a,b]\|_\B \leq F_{n-2}$.
    Hence, to conclude the proof of \cref{thm:fibo-easy}, it only remains to prove that, when $|a| = 1$ (so $a = X_0$) and $|b| = 2n_a(b) + \nu_a(b)$, one has $\|[a,b]\|_\B \leq F_{|b|-1}$.
    This is the goal of \cref{s:fibo-refined} (see \cref{th:X0b-fibo}).
\end{remark}

\subsection{Enhanced estimate for brackets with simple leaves}
\label{s:fibo-refined}

In this paragraph, we focus on the special case when $a = X_0$ and $b \in \B \setminus X$ has simple leaves in the sense that $L_\B(\rf{a}{b}) \subset A$ where $A := \{ (X_0,X_1), X_1 \}$.
Such brackets are those satisfying $|b| = 2 n_a(b) + \nu_a(b)$ and $|b| \geq 2$.
In this setting, we prove in \cref{th:X0b-fibo} that the index of the Fibonacci bound \eqref{eq:ab-fibo} can be decreased.

We start with a short lemma on structural properties of such brackets, then prove the desired estimate in a particular case before proceeding to the main result.

\begin{lemma} \label{lem:x0x1min}
    Let $b \in \B \cap \e(\Br(A))$.
    Then $b \geq (X_0, X_1)$ with equality iff ($b \neq X_1$ and $\lambda(b) = X_0$).
\end{lemma}

\begin{proof}
    This is an immediate consequence of \cref{Lem:BRA} and the minimality of $X_0$ in $\B$.
\end{proof}

\begin{lemma} \label{lem:ab-forme-speciale}
    Let $a := X_0$ and $b \in \B \setminus X$ such that $|b| = 2 n_a(b) + \nu_a(b) \geq 2$.
    Assume that $\lambda^2(b) = X_0 < \lambda(b)$.
    Then $b = \ad_{(X_0,X_1)}^{n_a(b)}(X_1)$ and $\|[a,b]\|_\B \leq F_{2n_a(b)+\nu_a(b)-1}$.
\end{lemma}

\begin{proof} 
    Let us prove both claims successively.
    
    \step{First item}
    Since $\theta_a(b) > 1$, $\lambda(b) = \e(\lambda(\rf{a}{b}))$ so $\lambda(b) \in \e(\Br(A))$. By \cref{lem:x0x1min}, $\lambda(b) = (X_0, X_1)$. 
    Let $n \geq 1$ such that $b = \ad_{(X_0,X_1)}^n (c)$ with either $c = X_1$ or $\lambda(c) < (X_0,X_1)$. 
    Let us check that $c \neq X_1$ cannot occur.
    If $c\neq X_1$ then $c \in \e(\Br(A))$.
    If $c = (X_0, X_1)$, $b \notin \B$ since $\ev(b) = 0$. 
    Otherwise $\lambda(c) \in \e(\Br(A))$ so $\lambda(c) \geq (X_0, X_1)$ by \cref{lem:x0x1min} which contradicts $\lambda(c) < (X_0, X_1)$.
    So $c = X_1$ and $n = n_a(b)$.
    
    \step{Size estimate}
    Since $\lambda^2(b) = X_0 = a$, estimate \eqref{eq:l2b<=a-n0n1n2} of \cref{p:lambda2b<=a-easy} yields $\|[a,b]\|_\B \leq (n-1)2^{n-1} + n + 1$ where $n = n_a(b)$.
    Moreover, $|b|=2n_a(b)+\nu_a(b) = 2n+1$ so $F_{|b|-1} = F_{2n}$.
    This proves the claimed estimate for $n \geq 8$ thanks to \cref{eq:n2n-9}.
    For $n \in \intset{1,7}$, the estimate $\| [a,b] \|_\B \leq F_{2n}$ can be checked using computer algebra for these short explicit brackets. 
    It could alternatively be proved from the enhanced bound of \cref{cor:atheta-1} and \cref{eq:n2n-f2n2}.
\end{proof}

\begin{proposition} \label{th:X0b-fibo}
    Let $a := X_0$ and $b \in \B \setminus X$ such that $|b| = 2 n_a(b) + \nu_a(b) \geq 2$. Then
    \begin{equation} \label{eq:ab-fibo-moins1}
        \| [a, b] \|_\B \leq F_{2n_a(b)+\nu_a(b)-1}.
    \end{equation}
\end{proposition}

\begin{proof}
    The proof follows along the same lines as the proof of \cref{th:fibo-size-new}, taking advantage of the additional assumptions.
    
     \step{Case 1: $\theta_a(b) = 1$ i.e.\ either $b=X_1$ or $\lambda(b)\leq a$}
    Then $\|[a,b]\|_\B = 1$ and $F_{2n_a(b)+\nu_a(b)-1} \geq F_1 = 1$ since we assumed that $|b| = 2n_a(b)+\nu_a(b) \geq 2$.
    So \eqref{eq:ab-fibo-moins1} holds in this case.
    
    \step{Case 2: $\theta_a(\lambda(b))=1$ i.e.\ $\lambda^2(b)\leq a < \lambda(b)$}
    Since $\lambda^2(b) \leq a = X_0$ and $X_0$ is minimal in $\B$, $\lambda^2(b) = X_0$.
    Then \eqref{eq:ab-fibo-moins1} follows from  \cref{lem:ab-forme-speciale}.
    
    \medskip
    
    From now on, we assume $a < \lambda^2(b)$ (so $\theta_a(\lambda(b)) > 1$).
    We reuse the notations $n_1 := n_a(\lambda(b))$, $\nu_1 := \nu_a(\lambda(b))$, $n_2 := n_a(\mu(b))$, $\nu_2 := \nu_a(\mu(b))$.
    
    \step{Case 3: $\theta_a(\lambda(b))\geq 2$ i.e.\ $a<\lambda^2(b)$ and $\mu(b)=X_1$}
    Then $n_a(b) = n_1$ and $\nu_a(b) = \nu_1 + 1$.
    \begin{itemize}
        \item If $n_a(b) = 1$, then $b = \dad_{X_1}^{\nu_a(b)}(X_0,X_1)
        = \dad_{X_1}^{\nu_a(b)+1}(X_0)$.
        Since $r(X_0,X_1) = + \infty$, \eqref{eq:X0bn-Fn} yields $\|[a,b]\|_\B = F_{\nu_a(b)+1}$.
        Since $n_a(b) = 1$, $\nu_a(b)+1= 2n_a(b) + \nu_a(b)-1$ so \eqref{eq:ab-fibo-moins1} is proved.
        
        \item If $n_a(b) = n_1 > 1$, we write \eqref{eq:jacobi-ab-alambdabu-new}.
        Concerning the first term of \eqref{eq:jacobi-ab-alambdabu-new}, we apply \eqref{eq:ab-fibo-moins1} to $a \gets a$ and $b \gets \lambda(b)$ (the latter indeed satisfies the assumptions of the proposition).
        Hence $\|[a,\lambda(b)]\|_\B \leq F_{2n_1+\nu_1-1}$.
        For each $d \in \supp [a, \lambda(b)]$, $(d, \mu(b)) = (d, X_1) \in \B$, so $\|[[a,\lambda(b)], \mu(b)]\|_\B \leq F_{2n_1+\nu_1-1}$.
        Concerning the second term of \eqref{eq:jacobi-ab-alambdabu-new}, $d := (a, \mu(b)) = (a, X_1) \in \B$.
        If $\lambda(b) < d$, then $(\lambda(b), d) \in \B$ because $a < \lambda(b)$.
        If $d < \lambda(b)$, \eqref{eq:ab-fibo} yields $\| [[a, \mu(b)], \lambda(b)] \|_\B \leq F_{2 n_1'+\nu_1'}$ where $n_1' := n_d(\lambda(b))$ and $\nu_1' := \nu_d(\lambda(b))$.
        By \cref{Cor:surcoupe_strict} (applied to $a \gets a$, $h \gets d= (a, X_1)$, $b \gets \lambda(b)$), $2n_1'+\nu_1' \leq 2n_1 + \nu_1 - n_1 \leq 2n_1+\nu_1-2$.
        Hence $\| [[a, \mu(b)], \lambda(b)] \|_\B \leq F_{2 n_1 + \nu_1 - 2}$.
        Summing our estimates for both terms of \eqref{eq:jacobi-ab-alambdabu-new} provides $
         \| [a, b] \|_\B \leq 
        F_{2n_1+\nu_1-1} + F_{2n_1+\nu_1-2}
        = F_{2n_1 + \nu_1}$,
        which proves \eqref{eq:ab-fibo-moins1} when $\mu(b) = X_1$ since $n_a(b) = n_1$ and $\nu_a(b) = \nu_1 + 1$.
    \end{itemize}
    
    \step{Case 4: $\theta_a(\lambda(b))\geq 2$ i.e.\ $a<\lambda^2(b)$, $\mu(b) \notin X$ and $\theta_a(\mu(b))=1$}
    This case does not happen.
    Indeed, since $\mu(b) \neq X_1$, $\theta_a(\mu(b)) = 1$ implies that $\lambda(\mu(b)) \leq a = X_0$ so, by \cref{lem:x0x1min}, $\mu(b) = (X_0, X_1)$.
    But, by \cref{lem:x0x1min}, $\lambda(b) \geq (X_0,X_1)$, which contradicts $\lambda(b) < \mu(b)$.
    
    \step{Case 5: $\theta_a(\lambda(b))\geq 2$ i.e.\ $a<\lambda^2(b)$, $\mu(b) \notin X$ and $\theta_a(\mu(b))\geq 2$} 
    We proceed exactly as in the \emph{general case} of the proof of    \cref{th:fibo-size-new}.
    When $(n_1,\nu_1) = (2,0)$ or $n_1+\nu_1 > 2$, estimates \eqref{eq:ab-ff-20} and \eqref{eq:ab-ff-n1n2} directly prove the conclusion \eqref{eq:ab-fibo-moins1}.
    When $n_1 = 1$ and $\nu_1 = 1$, $\lambda(b) = ((X_0,X_1),X_1)$ so $[a,\lambda(b)] = \ev(d)$ where $d := (\ad_{X_0}^2(X_1),X_1) \in \B$ and \eqref{eq:almu-bis} becomes
    \begin{equation}
        \| [[a, \lambda(b)], \mu(b)] \|_\B \leq F_{2n_2+\nu_2-n_2}.
    \end{equation}
    Together with \eqref{eq:lamu-max3}, this yields
    \begin{equation}
        \|[a,b]\|_\B \leq F_{2n_2+\nu_2-1} + 2 F_{2n_2+\nu_2}
        = F_{2n_2+\nu_2+2}
        = F_{2(n_1+n_2)+(\nu_1+\nu_2)-1},
    \end{equation}
    which concludes the proof of \eqref{eq:ab-fibo-moins1}.
\end{proof}

As mentioned in \cref{rk:cor-fibo}, the bound \eqref{eq:ab-fibo-moins1} of \cref{th:X0b-fibo} concludes the proof of \cref{thm:fibo-easy} with the optimal index of the Fibonacci sequence.

\color{black}
\subsection{Optimal \texorpdfstring{$\theta$}{theta}-based estimate}
\label{s:fibo-theta}

In this paragraph, we derive an optimal bound for $\|[a,b]\|_\B$ with respect to $\theta_a(b)$.
We define the following integer-valued function which will correspond to the optimal $\theta$-based estimate of the structure constants of $\B$:
\begin{equation} \label{eq:a-theta}
    A(\theta) := (\theta-3) 2^{\theta-2} + \theta + 1.
\end{equation}

First, we give in \cref{s:sharp-example} examples of brackets $a < b \in \B$ such that $\|[a,b]\|_\B = A(\theta_a(b))$, illustrating that the optimal $\theta$-based estimate within $\B$ is in particular larger than $2^{\theta-1}$.
We discuss in \cref{s:sharp-minimality} on the apparent paradox that $\B$ is minimal for length-based estimates but not for $\theta$-based ones due to the previous examples.

Since this question interested us, was involved in an earlier proof of \cref{th:fibo-size-new} and is of independent interest for a deeper understanding of $\B$, we then prove that $\|[a,b]\|_\B \leq A(\theta_a(b))$ for every $a < b \in \B$.
We start in \cref{s:sharp-right} by the particular case where $\rf{a}{b}$ is right-nested (in the sense of \cref{lem:tab-right}).
Eventually, we extend this bound to all brackets in \cref{s:sharp-all}.

\subsubsection{An example of brackets attaining the upper bound}
\label{s:sharp-example}

\begin{proposition} \label{prop:fibo-sature}
    Let $p \geq 1$, $m \geq 2$, $a := (X_0, X_1)$, $h := \ad_a^m(X_1)$ and $b := \ad_h^p(X_1)$.
    Then $\theta_a(b) = p+1$ and $\|[a,b]\|_{\B}=A(\theta_a(b))$.
\end{proposition}

\begin{proof} 
    We have 
    \begin{equation} \label{ah<h}
        \lambda(h) = a < h  \qquad \text{and} \qquad (a,h)<h.
    \end{equation}
    In particular, $(a,b) \notin \B$, $(a,h)\in\B$, $(a,X_1)\in\B$ thus $\theta_a(b)=p+1$ and our goal is to prove $\|[a,b]\|_{\B} = (p-2)2^{p-1}+p+2$. 
    Using \cref{eq:X0X1nX2-JacR} with $r=2$ and \cref{alphas:valeur} (although these formulas are stated for a bracket of the form $[X_0, \ad_{X_1}^n(X_2)]$, their derivation solely relies on Jacobi's identity, so remains just as valid for our bracket $[a, \ad_h^p(X_1)]$), we get
    \begin{equation} 
        [a,b]=\ad_{h}^{p}([a,X_1])+p \ad_{h}^{p-1}([(a,h),X_1])+\sum_{s=2}^{p} (s-1)\binom{p}{s} [\ad_{h}^{s-2}((a,h),h),\ad_{h}^{p-s}(X_1)].
    \end{equation}
    Let us prove that the $(p+1)$ brackets in the right-hand side are (evaluations of) different elements of $\B$. 
    The first two ones are by (\ref{ah<h}). Let $s \in \intset{2 , p}$. We have $\ad_{h}^{s-2}((a,h),h) < \ad_{h}^{p-s}(X_1)$, either because $X_1$ is maximal when $(p-s)\leq(s-2)$, or because $(a,h)<h$ when $(s-2)<(p-s)$. Thus $(\ad_{h}^{s-2}((a,h),h),\ad_{h}^{p-s}(X_1))\in\B$,
    because, when $s\leq p-1$, we have $h<\ad_h^{s-2}((a,h),h)$ 
    (this relation holds because $\lambda(h)=a<(a,h)$ when $s=2$). In conclusion,
    \begin{equation} 
        \|[a,b]\|_\B = 1 + p + \sum_{s=2}^{p} (s-1)\binom{p}{s} = 1 + p + \sum_{s=2}^p p \binom{p-1}{s-1} - \left( 2^p-1-p \right),
    \end{equation}
    which concludes the proof.
    (One could wish to rely on \cref{eq:X0X1nX2-JacR} and on \cref{Prop:free-lie-alphabetic} with the triple $\{ a, h, X_1 \}$ to avoid checking that the brackets are indeed part of the basis, but, unfortunately, although this set is alphabetic, it is not free since $h = \ad_a^m(X_1)$).
\end{proof}

\subsubsection{Discussion on the minimality of our basis}
\label{s:sharp-minimality}

\cref{prop:fibo-sature} proves that, for every $\theta \geq 3$, there exist $a < b \in \B$ with $\theta_a(b) = \theta$ and $\|[a,b]\|_\B = A(\theta_a(b)) > 2^{\theta_a(b)-1}$ (and, asymptotically, $A(\theta) \gg 2^{\theta-1}$).
Hence, although the structure constants of the basis constructed in \cref{s:fibo-def} have a minimal growth with respect to the length of the involved brackets, they do not have minimal growth with respect to $\theta_a(b)$.
Indeed, for length-compatible Hall sets (see \cref{Thm:lgth-comp-main}) and for the Lyndon basis (see \cref{thm:lyndon-bound}), one has $\|[a,b]\|_\B \leq 2^{\theta_a(b)-1}$, which is the minimal $\theta$-based bound, due to the lower-bound examples of \cref{prop:lower-theta}.

In \cref{Prop:theta/length}, we proved that for any Hall set with $|X| = 2$, $\theta_a(b) \leq |b|-1$ and that this bound was attained within each Hall set.
Hence, the apparent paradox between the fact that $\B$ is somehow length-minimal but not $\theta$-minimal does not come from a better estimation of $\theta_a(b)$ from~$|b|$.
However, the brackets such that $\theta_a(b) = |b|-1$ are $a = X_0$ and $b = \dad^{n}_{X_1}(X_0)$ for some $n \in \N^*$, which does not match the form of the brackets attaining equality in \cref{prop:fibo-sature}, which explains why there is no contradiction.

\subsubsection{Sharp bound for right-nested brackets}
\label{s:sharp-right}

The goal of this paragraph is to prove \cref{thm:a-theta-droite}, which improves the bound given in \cref{p:lambda2b<=a-easy} from $(\theta-2) 2^{\theta-2} + \theta$ down to $A(\theta)$ (already for $\theta=3$, the latter is strictly smaller than the former). 
The proof builds upon the method developed in \cref{s:fibo-right-quick}.
The tighter bound comes notably from the identification of particular elements of the support that ``change type'' (from being part of $N_2$ to being part of $N_1$) at each induction step, thereby avoiding a two-fold increase of the associated part of the norm. 

We start by a slight precision in the conclusions of \cref{lem:size-n2-init}.

\begin{lemma} 
    \label{lem:size-n2-init-better}
    Let $a < b \in \B$ with $\theta_a(b) = 2$.
    Then $\|[a,b]\|_\B \leq 2$ and $N_2(a,b) \leq 1$.
    Moreover, when $N_2(a,b) = 1$, the unique $c \in \supp [a,b]$ such that $\theta_{\lambda(b)}(c) = 2$ satisfies $\lambda(c) \leq (a, \lambda(b))$.
\end{lemma}

\begin{proof}
    The first assertions are proved in \cref{lem:size-n2-init}.
    When $N_2(a,b) = 1$, with the notations of the proof of \cref{lem:size-n2-init}, the unique $c \in \supp [a,b]$ such that $\theta_{\lambda(b)}(c) = 2$ is named $c_1$ and satisfies $\lambda(c_1) = (a, \lambda(b))$ or $\lambda(c_1) = \mu(b)$ (when $\mu(b) < (a, \lambda(b))$.
    In both cases, $\lambda(c) \leq (a, \lambda(b))$.
\end{proof}

We now wish to weaken the assumption $\lambda^2(b) < a$ of \cref{p:lambda2b<a-easy} down to $\lambda(b) < (a, \lambda(b))$.
We start by checking that this assumption is stable in the induction process going from $b$ to $\mu(b)$.

\begin{lemma} \label{Lem:bmubgentil}
Let $a<b \in \B$ such that $\theta_a(b)\geq3$, $(a,\lambda(b))\in\B$ and $\lambda(b)<(a,\lambda(b))$. Then $(a,\lambda(\mu(b)))\in\B$ and $\lambda(\mu(b))<(a,\lambda(\mu(b)))$. 
\end{lemma}

\begin{proof}
The assumption $\theta_a(b)\geq 3$ implies $(a,b)\notin\B$ thus $b\neq X_1$, so $\lambda(b)$ makes sense and $a<\lambda(b)$. Since $(a,\lambda(b))\in\B$ then $\lambda(b)\neq X_0$, so $\lambda^2(b)$ is well defined and $\lambda^2(b)\leq a$.
Moreover $\theta_a(\mu(b))=\theta_a(b)-1\geq 2$, so $(a,\mu(b))\notin\B$, thus $\mu(b)\neq X_1$, $\lambda(\mu(b))$ makes sense and $a<\lambda(\mu(b))$. Then $\lambda(\mu(b))\neq X_0$ so $\lambda^2(\mu(b))$ makes sense. Since $b\in\B$, $\lambda(\mu(b))\leq\lambda(b)$ thus, using $\lambda(b) < (a, \lambda(b))$ and \cref{def:fibo-order} and
$\lambda^2(\mu(b))\leq \lambda^2(b) \leq a$. 
We deduce that $(a,\lambda(\mu(b)))\in\B$.

Since $\lambda(b) < (a, \lambda(b))$, by \cref{sx:order}, $\lambda(b) = \ad^m_a(v)$ for some $m \in \N$ with $v \neq X_1$ and $\lambda(v) < a$.
By contradiction, if $(a, \lambda(\mu(b))) < \lambda(\mu(b))$, then, by \cref{sx:order}, $\lambda(\mu(b)) = \ad^n_a(X_1)$ with $n \in \N$.
Then, by \cref{sx:order}, $\lambda(\mu(b)) > \lambda(b)$, which contradicts the assumption $b \in \B$.
\end{proof}

We now improve \cref{p:lambda2b<a-easy} by weakening the assumption and adding the slight precision reminiscent from the initialization of \cref{lem:size-n2-init-better}.

\begin{proposition} \label{Prop:lambda2b<a}
    Let $a < b \in \B$ such that $(a, b) \notin \B$, $(a,\lambda(b))\in\B$ and $\lambda(b)<(a,\lambda(b))$.
    Then the estimates \eqref{eq:l2b<a-n0n1n2} still hold and, moreover, there exists $c \in \supp [a,b]$ such that $\theta_{\lambda(b)}(c) = 2$ and $\lambda(c) \leq (a, \lambda(b))$ (except when $\theta_a(b) = 2$ and $N_2(a,b) = 0$).
\end{proposition}

\begin{proof}
    The proof is the same as the one of \cref{p:lambda2b<a-easy} (by induction on $\theta_a(b) \geq 2$).
    The initialization for $\theta_a(b) = 2$ is performed in \cref{lem:size-n2-init-better}.
    The key point is that \cref{Lem:bmubgentil} allows to apply the induction hypothesis to the bracket $[a, \mu(b)]$.
    When $\theta_a(b) > 2$, the additional assertion on $\supp [a,b]$ comes from the fact that the first term of \eqref{eq:ab-jac-new}, named $c_1$ in the proof of \cref{p:lambda2b<a-easy}, satisfies $\lambda(c_1) = (a, \lambda(b))$.
\end{proof}

When $(a, \lambda(b)) < \lambda(b)$, we give induction relations for $\|[a, b]\|_\B$ and $N_2(a,b)$ based on the quantities $\|[a,\mu(b)]\|_\B$ and $N_2(a,\mu(b))$.

\begin{proposition} \label{Prop:lambdab_mechant}
    Let $a<b\in\B$ with  $\theta_a(b)\geq 3$, $\lambda^2(b)\leq a$ and $(a,\lambda(b))<\lambda(b)$. 
    \begin{itemize}
        \item If $\lambda(\mu(b))\leq(a,\lambda(b))$, then 
        $N_2(a,b) \leq 2(N_2(a,\mu(b))-P)$ and
        $\|[a,b|\|_\B \leq 1+\|[a,\mu(b)]\|_\B + N_2(a,\mu(b))-P$
        
        \item Otherwise
        $N_2(a,b) \leq 2^{\theta_a(\mu(b))-1} - 1 + 2(N_2(a,\mu(b)) - P)$ and $\|[a,b]\|_\B\leq 2^{\theta_a(\mu(b))-1}+\|[a,\mu(b)]\|_\B +N_2(a,\mu(b)) - P$,
    \end{itemize}
    where $P$ is the fraction of $\|[a,\mu(b)]\|_\B$ indexed by $d\in\supp[a,\mu(b)]$ such that $\lambda(\mu(b))<\lambda(d)\leq\lambda(b)$.
\end{proposition}

\begin{proof}
    The assumptions imply $a<\lambda(b)$, $a<\lambda(\mu(b))$, $(a,\lambda(b))\in\B$ and $\theta_a(\mu(b))=\theta_a(b)-1 \geq 2$.
    We study both terms of Jacobi's identity \eqref{eq:ab-jac-new}.

    \step{First term $[(a,\lambda(b)),\mu(b)]$} 
    We have $(a,\lambda(b))<\mu(b)$ by \cref{def:fibo-order} because $a<\lambda(\mu(b))$.
    \begin{itemize}
        \item If $\lambda(\mu(b))\leq(a,\lambda(b))$ then $c:=((a,\lambda(b)),\mu(b))\in\B$ and $\theta_{\lambda(b)}(c)=1$ because $\lambda(c)<\lambda(b)$.
        
        \item If $(a,\lambda(b)) < \lambda(\mu(b))$, we apply \cref{Prop:lambda2b<a} to the bracket $[(a,\lambda(b)),\mu(b)]$. 
        Indeed $\lambda^2(\mu(b))\leq \lambda^2(b)\leq a < (a,\lambda(b))$, thus $\lambda(\mu(b)) < ((a, \lambda(b)), \lambda(\mu(b)))$.
        By \cref{p:theta-dec}, $\theta_{(a,\lambda(b))}(\mu(b))\leq\theta_{a}(\mu(b))$. 
        Therefore \cref{Prop:lambda2b<a} gives $\|[(a,\lambda(b)),\mu(b)]\|_\B \leq 2^{\theta_a(\mu(b))-1}$, $N_2((a, \lambda(b)), \mu(b)) \leq 2^{\theta_a(\mu(b))-1}-1$.
        Each $c\in\supp[(a,\lambda(b)),\mu(b)]$ satisfies $\theta_{\lambda(b)}(c) \leq \theta_{\lambda(\mu(b))}(c)$ by \cref{p:theta-dec}.
    \end{itemize}
    Hence, if $\lambda(\mu(b)) \leq (a,\lambda(b))$, the first term does not contribute to $N_2(a,b)$ and contributes at most $1$ to $\|[a,b]\|_\B$.
    Otherwise, the first term contributes at most $2^{\theta_a(\mu(b))-1}-2$ to $N_2(a,b)$ and at most $2^{\theta_a(\mu(b))-1}$ to $\|[a,b]\|_\B$.
    
    \step{Second term $[\lambda(b), [a, \mu(b)]]$} 
    By \cref{th:fibo-structure-new}, $[a,\mu(b)]=\sum \alpha_d \ev(d)$ where the sum ranges over $d\in\B$ such that $a<\lambda(d)<\mu(b)$, $\lambda(\mu(b)) < d$ and $\theta_{\lambda(\mu(b))}(d)\leq 2$ using \cref{lem:l2ba-supp}.
    Let $d \in \supp[a,\mu(b)]$. 
    Then $\lambda(b)<d$, by \cref{def:fibo-order}, because $\lambda^2(b)\leq a < \lambda(d)$. 
    \begin{itemize}
        \item If $\theta_{\lambda(\mu(b))}(d)=1$ then $\theta_{\lambda(b)}(d)=1$ by \cref{p:theta-dec}, thus $c:=(\lambda(b),d)\in\B$ and $\theta_{\lambda(b)}(c)=1$.
        
        \item If $\theta_{\lambda(\mu(b))}(d)=2$ then $\theta_{\lambda(b)}(d) \in \{1, 2\}$ by \cref{p:theta-dec}, 
        \begin{itemize}
            \item if $\theta_{\lambda(b)}(d)=1$ i.e.\ $d$ is an index of the sum defining $P$, then $c:=(\lambda(b),d)\in\B$ and $\theta_{\lambda(b)}(c)=1$,
            
            \item if $\theta_{\lambda(b)}(d)=2$ then $\|[\lambda(b),d]\|_\B\leq 2$.
        \end{itemize}
    \end{itemize}
    Hence, the second term contributes at most $2 (N_2(a, \mu(b)) - P)$ to $N_2(a,b)$ and at most $(N_1(a,\mu(b)) + P) + 2(N_2(a, \mu(b)) - P) = \|[a,\mu(b)]\|_\B + N_2(a,\mu(b)) - P$ (thanks to \cref{lem:l2ba-supp}) to $\|[a,b]\|_\B$, which concludes the proof.
\end{proof}

\begin{lemma} \label{lem:mechant-gentil}
    Let $a < b \in \B$ with $\theta_a(b) \geq 3$, $\lambda^2(b) \leq a$, $(a, \lambda(b)) < \lambda(b)$ and $\lambda(\mu(b)) < (a, \lambda(\mu(b)))$.
    Then $\| [a,b] \|_\B \leq 2^{\theta-1}-1$ and $N_2(a,b) \leq 2^{\theta-1}-4$ where $\theta := \theta_a(b)$.
\end{lemma}

\begin{proof}
    Since $\lambda(\mu(b)) < (a, \lambda(\mu(b)))$, the bracket $[a, \mu(b)]$ satisfies the hypotheses of \cref{Prop:lambda2b<a} and $\theta_a(\mu(b)) = \theta - 1$.
    Thus, $\| [a, \mu(b)] \|_\B \leq 2^{\theta-2}$ and $N_2(a, \mu(b)) \leq 2^{\theta-2}-1$.
    Moreover, there exists $c \in \supp [a, \mu(b)]$ such that $\lambda(\mu(b)) < \lambda(c) \leq (a, \lambda(\mu(b))) \leq (a, \lambda(b)) < \lambda(b)$, except when $\theta = 3$ and $N_2(a, \mu(b)) = 0$.
    The hypothesis $\lambda(\mu(b)) < (a, \lambda(\mu(b)))$ implies that $\lambda(\mu(b)) < (a, \lambda(b))$.
    Hence, we can apply the first item of \cref{Prop:lambdab_mechant}.
    
    \step{Case $\theta > 3$ or $N_2(a,\mu(b)) > 0$}
    Then $P \geq 1$.
    Thus $N_2(a,b) \leq 2 (2^{\theta-2}-1-1) = 2^{\theta-1}-4$ and $\|[a,b]\|_\B \leq 1 - 1 + 2^{\theta-2} + 2^{\theta-2} - 1 = 2^{\theta-1} - 1$. 
    
    \step{Case $\theta = 3$ and $N_2(a, \mu(b)) = 0$}
    Then $N_2(a,b) \leq 2 (0 - 0) = 0 = 2^{\theta-1} - 4$ and $\|[a,b]\|_\B \leq 1 - 0 + 2^{\theta-2} + 0 = 3 = 2^{\theta-1}-1$.
\end{proof}

\begin{lemma} \label{lem:A-mechant}
    Let $a < b \in \B$ with $\theta_a(b) \geq 2$, $\lambda^2(b) \leq a$ and $(a, \lambda(b)) < \lambda(b)$.
    Then
    \begin{equation} \label{eq:A-mechant}
        \|[a,b]\|_\B \leq A(\theta)
        \quad \text{and} \quad 
        N_2(a,b) \leq A(\theta)-\theta,
        \quad \text{where} \quad
        \theta := \theta_a(b).
    \end{equation}
\end{lemma}

\begin{proof}
    We proceed by induction on $\theta := \theta_a(b) \geq 2$.
    
    \step{Initialization for $\theta = 2$}
    According to \cref{lem:size-n2-init-better}, $\|[a,b]\|_\B \leq 2 = A(2)$, $N_2(a,b) \leq 1$, and moreover, if $N_2(a,b) = 1$, the unique $c \in \supp [a,b]$ such that $\theta_{\lambda(b)}(c) = 2$ satisfies $\lambda(c) \leq (a, \lambda(b))$.
    Thus, by assumption, $\lambda(c) < \lambda(b)$ so $\theta_{\lambda(b)}(c) = 1$.
    Hence $N_2(a,b) = 0 = A(2)-2$.
    So \eqref{eq:A-mechant} holds.
    
    \step{Induction for $\theta \geq 3$}
    If $\lambda(\mu(b)) < (a, \lambda(\mu(b)))$, estimate \eqref{eq:A-mechant} follows from \cref{lem:mechant-gentil} since $2^{\theta-1}-1 \leq A(\theta)-\theta$ and $2^{\theta-1}-1 \leq A(\theta)$ for $\theta \geq 3$.
    Otherwise, when $(a, \lambda(\mu(b))) < \lambda(\mu(b))$, the bracket $[a, \mu(b)]$ satisfies the hypotheses of \cref{lem:A-mechant}.
    Thus $\|[a, \mu(b)]\|_\B \leq A(\theta-1)$ and $N_2(a, \mu(b)) \leq A(\theta-1) - (\theta-1)$.
    Moreover, by \cref{Prop:lambdab_mechant}, even if $P = 0$,
    \begin{equation}
        \begin{split}
            N_2(a,b) 
            & \leq 2^{\theta-2}-1 + 2 N_2(a, \mu(b)) \\ 
            & \leq 2^{\theta-2}-1 + 2 (\theta-4) 2^{\theta-3} + 2 = A(\theta) - \theta
        \end{split}
    \end{equation}
    and
    \begin{equation} \label{eq:ind-norm-ab}
        \begin{split}
        \|[a,b]\|_\B 
        & \leq 2^{\theta-2} +\|[a,\mu(b)]\|_\B + N_2(a, \mu(b)) \\
        & \leq 2^{\theta-2} + A(\theta-1)
        + A(\theta-1) - (\theta-1) = A(\theta),
        \end{split}
    \end{equation}
    so \eqref{eq:A-mechant} is proved.
\end{proof}

\begin{theorem} \label{thm:a-theta-droite}
    Let $a < b \in \B$ with $(a,b) \notin \B$ such that $\lambda^2(b) \leq a$.
    Then $\|[a,b]\|_\B \leq A(\theta_a(b))$.
\end{theorem}

\begin{proof}
    This follows from \cref{Prop:lambdab_mechant} when $\lambda(b) < (a, \lambda(b))$ and from \cref{lem:A-mechant} when $(a, \lambda(b)) < \lambda(b)$ (the situation $\lambda(b) = (a, \lambda(b))$ being impossible, since $|(a, \lambda(b))| > |\lambda(b)|$).
\end{proof}

We eventually study the following particular subcase, which provides a theoretical proof to a numerical claim made in the proof of \cref{lem:ab-forme-speciale}.  

\begin{corollary} \label{cor:atheta-1}
    Let $\theta \geq 2$. 
    Then $\|[X_0,\ad_{(X_0,X_1)}^{\theta-1}(X_1)]\|_\B \leq A(\theta) - 1$.
\end{corollary}

\begin{proof}
    Let $a := X_0$ and $b_\theta := \ad^{\theta-1}_{(X_0,X_1)}(X_1)$.
    When $\theta = 2$, by Jacobi's identity,
    \begin{equation} \label{eq:jacobi-cancel}
        [a, b_2]
        = [[X_0, [X_0, X_1]], X_1] + [[X_0, X_1], [X_0, X_1]]
        = [[X_0, [X_0, X_1]], X_1],
    \end{equation}
    where $((X_0, (X_0, X_1)), X_1) \in \B$.
    Hence $\|[a,b_2]\|_\B = 1 = A(2) - 1$ and $N_2(a,b_2) = 0$, which, thanks to the cancellation in \eqref{eq:jacobi-cancel}, are strictly better estimates than in \cref{lem:size-n2-init-better}.
    
    Moreover, for every $\theta \geq 2$, $\lambda^2(b_\theta) = X_0 = a$ and $(a, \lambda(b)) = (X_0, (X_0, X_1)) < (X_0, X_1) = \lambda(b)$.
    Since $\mu(b_\theta) = b_{\theta-1}$, the estimate $\|[a,b_\theta]\|_\B \leq A(\theta)-1$ is propagated by induction by \eqref{eq:ind-norm-ab}.
\end{proof}

\subsubsection{Sharp bound in the general case}
\label{s:sharp-all}

We prove \cref{thm:a-theta} which extends \cref{thm:a-theta-droite} to all brackets.
We start with a definition and preliminary results concerning situations where the structure constants are much smaller than expected (see \cref{lem:surcoupe_2theta}) and is related to \cref{Cor:surcoupe_strict}.

\begin{definition}
   We define a relation $\ll$ on $\B$ by the following conditions: for $a_0,a \in \mathcal B$, $a_0 \ll a$ when $a_0 \neq X_1$, $a \neq X_0$, and either $a = X_1$, or $a = (a_0,X_1)$, or $a_0<\lambda(a)$. 
\end{definition}

\begin{lemma}
    \label{lem:ll-comp}
    Let $a_0 \ll a \in \B$. 
    Then $a_0 < a$.
    If $b \in \B \setminus X$ is such that $a<b$, then $a_0<\lambda(b)$.
\end{lemma}

\begin{proof}
    The first item follows from the relations $a_0<(a_0,X_1)$ and $\lambda(a)<a$. 
    For the second item, let $b \in \B \setminus X$ such that $a<b$. 
    Then $a \neq X_1$. If $a=(a_0,X_1)$ then the relation $a<b$ implies $a_0<\lambda(b)$ because $\mu(b)<X_1$. 
    If $a_0<\lambda(a)$ then $a_0<\lambda(b)$ because $\lambda(a)\leq\lambda(b)$.
\end{proof}

\begin{remark}
    \cref{lem:ll-comp} implies that the relation $\ll$ is transitive. In fact, it is an order on~$\B$. 
    Indeed, the conditions $a\ll b$ and $b \ll a$ are incompatible since each of them implies the same comparison for the relation $<$. 
    However, since $(X_0,X_1)$ and $(X_0,(X_0,X_1))$ cannot be compared by~$\ll$, this order is not total.
\end{remark}


\begin{lemma}\label{lem:surcoupe_2theta}
    Let $a_0 \ll a < b < X_1 \in \B$.
    Then $\theta_{a_0}(b) \geq 2$ and $\|[a,b]\|_\B \leq 2^{\theta_{a_0}(b)-2}$.
\end{lemma}

\begin{proof}
    Since $a_0 \ll a <b$, \cref{lem:ll-comp} shows that $a_0 <  \lambda(b)$, so $\theta_{a_0}(b) \geq 2$. 
    We prove the estimate on $\|[a,b]\|_\B$ by induction on $\theta_{a_0}(b) \geq 2$.
    
    \step{Initialization for $\theta_{a_0}(b) = 2$}
    Then $\lambda^2(b) \leq a_0$ so $\lambda(b) \leq a$ (indeed, either $a=(a_0,X_1)$ and then $\lambda(b)=(\lambda^2(b),\mu\lambda(b)) \leq (a_0,X_1)=a$ or $a_0<\lambda(a)$ and then $\lambda^2(b)<\lambda(a)$ thus $\lambda(b)<a$).
    Therefore, $(a,b) \in \B$ and $\|[a,b]\|_\B = 1 = 2^{\theta_{a_0}(b)-2}$.
    
    \step{Inductive step for $\theta_{a_0}(b) \geq 3$}
    Since $a < \lambda(b)$, $a_0 < \lambda^2(b)$ by \cref{lem:ll-comp} and $\theta_{a_0}(\lambda(b)) \geq 2$.
    Moreover, either $\mu(b) = X_1$, or the same argument 
    implies that $\theta_{a_0}(\mu(b)) \geq 2$.
    
    \begin{itemize}
        \item \emph{Case $\mu(b) = X_1$.} 
        In this case, $[a,b] = [[a,\lambda(b)],X_1] + [\lambda(b),(a,X_1)]$. 
        For the first term, by the induction hypothesis, $\|[[a,\lambda(b)],X_1]\|_\B \leq 2^{\theta_{a_0}(\lambda(b)) - 2} = 2^{\theta_{a_0}(b) - 3}$. 
        For the second term, if $\lambda(b) < (a,X_1)$, then $(\lambda(b), (a,X_1)) \in \B$ so $\|[\lambda(b),(a,X_1)]\|_\B = 1 \leq 2^{\theta_{a_0}(b) - 3}$.
        If $(a,X_1) < \lambda(b)$, then $a_0 \ll a \ll (a,X_1)$, so by the induction hypothesis, $\|[(a,X_1),\lambda(b)]\|_\B \leq 2^{\theta_{a_0}(\lambda(b))-2} = 2^{\theta_{a_0}(b) - 3}$. 
        Therefore, summing the estimates proves that $\|[a,b]\|_\B \leq 2^{\theta_{a_0}(b)-2}$.
        
        \item \emph{Case $\mu(b) < X_1$.} 
        We decompose $[a,b]$ using Jacobi's identity.
        
        By the induction hypothesis, $[a,\lambda(b)] = \sum_{d} \alpha_d \ev(d)$ with $\sum_d |\alpha_d| \leq 2^{\theta_{a_0}(\lambda(b))-2}$. 
        Let $d \in \supp[a,\lambda(b)]$. 
        If $\mu(b) < d$, the inequality $\lambda(d) < \lambda(b) < \mu(b)$ from \cref{th:fibo-structure-new} shows that $(\mu(b),d) \in \B$.
        If $d < \mu(b)$, the inequality $\lambda(d) \geq a > a_0$ shows that, by the induction hypothesis, $\|[d,\mu(b)]\|_\B \leq 2^{\theta_{a_0}(\mu(b))-2}$. 
        Therefore $\|[[a,\lambda(b)],\mu(b)]\|_\B \leq 2^{\theta_{a_0}(b)-4}$.
        
        Since $\mu(b) \neq X_1$, $\lambda(\mu(b)) \geq \lambda(\lambda(b))> a_0$. 
        By the induction hypothesis, $[a,\mu(b)] = \sum_d \alpha_d \ev(d)$ with $\sum_d |\alpha_d| \leq 2^{\theta_{a_0}(\mu(b))-2}$.
        Let $d \in \supp[a,\mu(b)]$.
        If $d < \lambda(b)$, then using the induction hypothesis again shows that $\|[d,\lambda(b)]\|_\B \leq 2^{\theta_{a_0}(\lambda(b))-2}$. 
        If $\lambda(b) < d$, then $\lambda(\mu(b)) \leq \lambda(b) < d$ thus, by \cref{th:fibo-structure-new} and \cref{p:theta-dec}, $\theta_{\lambda(b)}(d) \leq \theta_{\lambda(\mu(b))}(d) \leq 2$ thus $\|[\lambda(b),d]\|_\B \leq 2$. 
        Taking into account that $\theta_{a_0}(\lambda(b)) \geq 2$ we conclude that $\|[\lambda(b),[a,\mu(b)]]\|_\B \leq 2^{\theta_{a_0}(b)-3}$.
        
        Combining both parts and using Jacobi's identity, one gets $\|[a,b]\|_\B \leq 2^{\theta_{a_0}(b)-2}$. \qedhere
    \end{itemize}
\end{proof}

\begin{theorem} \label{thm:a-theta}
    Let $a < b \in \B$. 
    Then $\|[a,b]\|_\B \leq A(\theta_a(b))$. Moreover, if $b \neq X_1$ and $(a, \lambda(b)) \notin \B$, then the inequality is strict.
\end{theorem}

\begin{proof}
    We proceed by induction on $\theta_a(b)$. 
    We may assume that $(a,b) \notin \B$.
    Thanks to \cref{thm:a-theta-droite}, which corresponds to the case $\theta_a(\lambda(b)) = 1$, we may further assume that $\theta_a(\lambda(b)) \geq 2$, i.e.\ that $\lambda(b) \neq X_0$ and $a<\lambda^2(b)$.
    
    \step{Case $\mu(b) = X_1$} 
    In this case, $[a,b] = [[a,\lambda(b)],X_1] - [(a,X_1),\lambda(b)]$. 
    By the induction hypothesis, we have $\|[a,\lambda(b)]\|_\B \leq A(\theta_a(\lambda(b)))$, and by \cref{lem:surcoupe_2theta} (since $(a,X_1)<\lambda(b)<X_1$), we have $\|[(a,X_1),\lambda(b)]\|_\B \leq 2^{\theta_a(\lambda(b))-2}$. Combining both estimates gives:
    \begin{equation}
        \|[a,b]\|_\B \leq A(\theta_a(b)-1) + 2^{\theta_a(b)-3} = A(\theta_a(b)) - ((\theta-2)2^{\theta-2}+1) < A(\theta_a(b)).
    \end{equation}
    
    \step{Case $\mu(b)\neq X_1$} 
    Using Jacobi's identity, we write $[a,b] = [[a, \lambda(b)], \mu(b)] + [\lambda(b), [a, \mu(b)]]$.
    We may assume that $\theta_a(\mu(b)) \geq 2$ (otherwise, $\lambda(\mu(b)) \leq a$ implies that $\lambda^2(b) \leq a$ so $\theta_a(\lambda(b)) = 1$).
    
    \begin{itemize}
        \item For the first term, by the induction hypothesis and \cref{th:fibo-structure-new}, $[a,\lambda(b)] = \sum_{d} \alpha_d \ev(d)$, with $a < \lambda(d) < \lambda(b)$, and $\sum_{d} |\alpha_d| \leq A(\theta_a(\lambda(b)))$. 
        Let $d \in \supp[a,\lambda(b)]$. 
        If $\mu(b) < d$, then $(\mu(b),d) \in \B$ and $\|[\mu(b),d]\|_\B = 1$. 
        If $d<\mu(b)$, then by \cref{lem:surcoupe_2theta} (since $\mu(b) \neq X_1$ and $a \ll d$), we have $\|[d,\mu(b)]\|_\B \leq 2^{\theta_a(\mu(b)) - 2}$. 
        Therefore, $\|[[a,\lambda(b)],\mu(b)]\|_\B \leq A(\theta_a(\lambda(b)))2^{\theta_a(\mu(b))-2}$.
    
        \item For the second term, by the induction hypothesis and \cref{th:fibo-structure-new}, $[a,\mu(b)] = \sum_{d} \alpha_d \ev(d)$, with $a < \lambda(d) < \mu(b)$, and $\sum_{d} |\alpha_d| \leq A(\theta_a(\mu(b)))$. Moreover, we know that either $d \leq \lambda(\mu(b)) \leq \lambda(b)$, or $\theta_{\lambda(b)}(d) \leq \theta_{\lambda(\mu(b))}(d) \leq 2$.
        If $\lambda(b) < d$, then $\|[\lambda(b),d]\|_\B \leq 2$.
        If $d < \lambda(b)$, by \cref{lem:surcoupe_2theta}, $\|[d,\lambda(b)]\|_\B \leq 2^{\theta_{a}(\lambda(b)) - 2}$.
        In both cases $\|[d,\lambda(b)]\|_\B \leq 2^{\theta_{a}(\lambda(b)) - 1}$, and thus $\|[\lambda(b),[a,\mu(b)]]\|_\B \leq A(\theta_a(\mu(b)))2^{\theta_a(\lambda(b)) - 1}$.
    \end{itemize}
    Finally, we get the estimate:
    \begin{equation}
        \|[a,b]\|_\B \leq A(\theta_a(\lambda(b)))2^{\theta_a(\mu(b))-2} + A(\theta_a(\mu(b)))2^{\theta_a(\lambda(b)) - 1}.
    \end{equation}
    Since $\theta_a(\lambda(b)) \geq 1$ and $\theta_a(\mu(b))\geq 1$, the result then comes from following claim:  
    \begin{equation}\label{eq:Apq}
        \forall p,q \geq 1, 
        \quad A(p)2^{q-2} + A(q)2^{p-1} < A(p)2^{q-1} + A(q)2^{p-1} \leq A(p+q).
    \end{equation}
    The first inequality is clear since $A(p) \neq 0$.
    For the second inequality, dividing by $2^{p+q-3}$, and using formula that defines $A$ shows that \cref{eq:Apq} is equivalent to:
    \begin{equation}
        p+q-6 + \frac{p+1}{2^{p-2}} + \frac{q+1}{2^{q-2}} \leq 2(p+q - 3) + \frac{p+q+1}{2^{p+q-3}}.
    \end{equation}
    This inequality can be checked directly when $p\in \{1,2\}$ (and thus holds when $q \in \{1,2\}$ by symmetry), and follows from $\frac{p+1}{2^{p-2}} + \frac{q+1}{2^{q-2}} \leq \frac{p+q+2}{2} \leq p+q$ when $p,q \geq 3$. Therefore \cref{eq:Apq} holds, which concludes the proof.
\end{proof}

\color{black}

\section{A super-geometric Hall set}
\label{sec:super-geom}

When $X$ is infinite, we know from \cref{cor:x-infini-sature} that the structure constants can grow super-geometrically with respect to the length of the involved brackets and to $\theta_a(b)$.
The main goal of this section is to prove \cref{thm:Mn-easy}, which illustrates that, even in the most constrained case when $|X| = 2$, there exists a Hall set on $X$ whose structure constants grow super-geometrically.
We conclude that the structure constants can grow super-geometrically both with respect to the length of the brackets and to $\theta_a(b)$.
Hence, the results obtained in \cref{sec:length} (for the length-compatible Hall sets) and \cref{sec:lyndon} (for the Lyndon sets), where the structure constants grow geometrically with the length of the involved brackets, cannot be extended to all Hall sets, even when only two indeterminates are considered.

We start by describing a natural strategy in \cref{rk:fail-embed} which however fails to prove the desired conclusion.
We then construct a specific order and the associated Hall set in \cref{ss:super-geom-basis}.
We compute the exact decomposition in the basis of a nasty Lie bracket in \cref{ss:Mn-nasty}.
Eventually, we prove a refined version of \cref{thm:Mn-easy} in \cref{ss:Mn-optimization} by optimizing the choice of the nasty bracket.

\subsection{A natural but unfitted strategy}
 \label{rk:fail-embed}
 
Assume that $X$ is finite.
A possible strategy to tackle \cref{thm:Mn-easy} would be to attempt to realize in $\Br(X)$ the critical construction of \cref{cor:x-infini-sature}.

One could attempt to construct a Hall set $\B \subset \Br(X)$, and, for every $p$ large enough, a free alphabetic subset $X_p := \{ x_0, x_1, \dotsc, x_p \} \subset \B$.
Then, using the isometry described in \cref{Prop:free-lie-alphabetic}, one could rely on the construction of \cref{p:en1-optimal} to obtain $\| [x_0, b] \|_\B = \lfloor e (p-1)! \rfloor$, where $b := (\dotsb ((x_1,x_p), x_{p-1}), \dotsc, x_2)$ (proving that $b \in \B$ and achieving the equality  would require that the order on $\B \cap \Br_{X_p}$ has been chosen as described in \cref{p:order-sharp-en1}).

Let us explain why this strategy fails to obtain a super-geometric lower growth with respect to the length of the considered brackets.

\bigskip

Let $\B \subset \Br(X)$ be a Hall set.
The construction would require in particular that $x_1, \dotsc, x_p$ are $p \geq 1$ distinct elements of $\B$.
Let us derive a lower bound for $|b| = |x_1| + \dotsb + |x_p|$ (not even considering the fact that executing the strategy would also require additional constraints, such as~$X_p$ being a free alphabetic subset of $\B$).
In particular $|b| \geq |y_1| + \dotsb + |y_p| =: C_p$ where $y_1, \dotsc, y_p$ are the first $p$ elements of $\B$, sorted by non-decreasing length (breaking ties using for example the order on $\B$).

For any $\ell \in \N^*$, 
\begin{equation} \label{eq:card-bell}
    |\{ t \in \B ; \enskip |t| = \ell \}| \leq |X|^\ell.
\end{equation}
From Witt's formula \cite{zbMATH03116668}, the inequality is strict for every $\ell > 1$.
However, this increases the lower-bound $C_p$ since there are fewer brackets of small length.
Hence, we continue the discussion as if \eqref{eq:card-bell} was an equality for all $\ell \in \N^*$.

Let $L := \max |y_i| = |y_p|$.
Then $|X|^{L+1} \geq p+1$.
And, since the $y_i$ were chosen by non-decreasing length,
\begin{equation} \label{eq:louche}
    C_p = 
    \sum_{i=1}^p |y_i|
    \geq \sum_{\ell = 1}^{L-1}
    \ell |X|^\ell
    \geq (L-1) |X|^{L-1}
    \geq \frac{1}{|X|^2} (p+1) \left(
    \frac{\ln (p+1)}{\ln |X|} - 2
    \right).
\end{equation}
Thus, for $p$ large enough, one has $p \ln p \leq (2 |X|^2 \ln |X|) |b|$.

Even if the construction is performed such that $\| [x_0, b ] \|_\B = \lfloor e(p-1)! \rfloor$, the bounds associated with Stirling's approximation (see e.g.\ \cite{zbMATH03115226}) yield, for $p \geq 3$,
\begin{equation}
    \| [x_0, b ] \|_\B
    \leq p^p
    = e^{p \ln p}
    \leq |X|^{2 |X|^2 |b|},
\end{equation}
which is a geometric bound with respect to $|b|$.

\bigskip

Hence, we develop in the sequel a method which avoids this path.
In particular, the brackets and the order we consider are not the same as in \cref{cor:x-infini-sature}.

\subsection{Construction of an appropriate Hall set}
\label{ss:super-geom-basis}

Let $X_0, X_1 \in X$. 
We start by introducing the elementary building blocks of our construction.
For $i \in \N$, define 
\begin{equation}
    A_i := \ad^i_{X_0}(X_1).
\end{equation}

\begin{lemma}
    The following subset of $\Br(X)$ is free in the sense of \cref{def:free-subset}:
    \begin{equation}
        A := \{ A_i ; \enskip i \in \N \}.
    \end{equation}
\end{lemma}

\begin{proof}
    All elements of $A$ contain $X_1$ exactly once, while elements of $(\Br_A, \Br_A)$ contain $X_1$ at least twice, so $A \cap (\Br_A, \Br_A) = \emptyset$.
\end{proof}

We now introduce a subset of $\Br(X)$ within which we will be working. Let
\begin{equation}
    G := \{ X_0, X_1 \} \cup G^*
    \quad \textrm{where} \quad 
    G^* := \{ b \in \Br_A ; \enskip X_1 \notin \Lambda(b)
	\},
\end{equation}
where $\Lambda(b)$ is the set of iterated left factors of $b$ (see \eqref{eq:LAMBDA}).

\begin{lemma}
    $G$ is $\lambda$-stable and $X_1 \notin G^*$.
\end{lemma}

\begin{proof}
    $X_1 \notin G^*$ because $X_1 \in \Lambda(X_1) = \{ \lambda^0(X_1) \} = \{ X_1 \}$.
    Moreover, if $b \in G^*$, either $b = A_i$ for some $i \geq 1$ so that $\lambda(b) = X_0 \in G$, or $b = (b_1, b_2)$ with $b_1, b_2 \in \Br_A$ and $X_1 \notin \Lambda(b) = \Lambda(b_1) \cup \{ b \}$. Hence $\lambda(b) = b_1 \in G^*$. Hence $\lambda(G\setminus X) = \lambda(G^*) \subset G^* \cup \{ X_0 \} \subset G$.
\end{proof}

\begin{definition}[Score on $\Br_A$]
    We define a score map $s : \Br(A) \to \N$.
    For $i \in \N$, set
    \begin{equation}
    	s(A_i) := 
    	\begin{cases}
    		i & \textrm{if } i \in \{ 0, 1, 2 \}, \\
    		3 \cdot 2^{i-3} & \textrm{if } i \geq 3.
    	\end{cases}
    \end{equation}
    By induction, if $b \in \Br(A)$ is of the form 
    $\langle b_1, b_2 \rangle$ with $b_1, b_2 \in \Br(A)$, we set
    \begin{equation} \label{eq:sb1b2}
        s(\langle b_1, b_2 \rangle) := s(b_1) + s(b_2).
    \end{equation}
    In particular, for all $b \in \Br_A \setminus \{ X_1 \}$, $s(b) > 0$.
    Then, since $A$ is free, by \cref{def:free-subset}, $\Br_A$ is isomorphic to $\Br(A)$.
    Thus, setting $s(\e(b)) := s(b)$ for all $b \in \Br(A)$ extends $s$ to $\Br_A$.
\end{definition}

\begin{definition}[Iteration by $X_1$ and germs]
    For $b \in \Br(X)$ and $\nu \in \N$, we will use in the sequel of this section the notation $b 1^\nu$ instead of $\dad_{X_1}^\nu(b)$ for concision.
    For every $b \in G^*$, there exists a unique couple $(b^*, \nu(b)) \in G^* \times \N$ such that $b = b^* 1^{\nu(b)}$ and $\mu(b^*) \neq X_1$.
    One checks that $b^* \in G^*$ because $X_1 \notin \Lambda(b) \supset \Lambda(b^*)$.
\end{definition}

\begin{definition}[Order on $G$] \label{def:order-G}
    We define an order on $G$ by setting $X_0 < G^* < X_1$ and, inside~$G^*$, the lexicographic order on the quadruple   $s(b), \lambda(b^*), \mu(b^*), \nu(b)$.
\end{definition}

\begin{proposition}
    There exists a Hall order on $\Br(X)$ extending the above order on $G$.
\end{proposition}

\begin{proof}
    Since $G$ is $\lambda$-stable, by \cref{Prop:Prolonger_ordre}, it is sufficient to check that the order defined in \cref{def:order-G} is a Hall order on $G$. 
    First, this order is total on $G$. 
    Indeed, the triple $\lambda(b^*), \mu(b^*), \nu(b)$ defines $b$ uniquely.
    Second, this order is compatible with $\lambda$. 
    Indeed, let $b \in G^*$.
    If $b = A_i$ for some $i \geq 1$, $\lambda(b) = X_0 < b$.
    Otherwise $b = (b_1, b_2)$ with $b_1 \in G^*$ and $b_2 \in \Br_A$.
    If $b_2 = X_1$, then $s(b) = s(b_1)$ by \eqref{eq:sb1b2} and $b^* = b_1^*$ so $\lambda(b^*) = \lambda(b_1^*)$ and $\mu(b^*) = \mu(b_1^*)$, and $\nu(b) = \nu(b_1) + 1$. 
    So $b_1 < b$.
    If $b_2 \neq X_1$, $s(b_2)  > 0$ and, by \eqref{eq:sb1b2}, $s(b_1) < s(b)$ which implies that $b_1 < b$.
\end{proof}

In the sequel, we consider $\B \subset \Br(X)$ the Hall set associated with any such order on $\Br(X)$.

\subsection{Decomposition of a nasty Lie bracket}
\label{ss:Mn-nasty}

Let $\nu \geq 0$. We then consider the sequence of brackets defined by induction as
\begin{equation}
    B^\nu_2 := A_2 1^\nu
    \quad \textrm{and} \quad 
    B^\nu_{k+1} := (B^\nu_k, A_{k+1})
    \textrm{ for } k \geq 2.
\end{equation}
We prove the following equality.

\begin{proposition}\label{Prop:estimation_bstar}
    Let $p \geq 2$ and $\nu \geq 0$. Then $A_1$ and $B^\nu_p$ belong to $\B$ and
    \begin{equation}
        \| [ A_1, B^\nu_p ] \|_\B = p^\nu + p - 2.
    \end{equation}
\end{proposition}

\begin{proof}
    \emph{Step 1: We check that the considered brackets are indeed in $\B$.}
    By definition of the order, $X_0 < X_1$, so, for all $i \in \N$, $A_i \in \B$. 
    Since $G^* < X_1$, for each $g \in G^*$ and $\nu \in \N$, $g1^\nu \in \B$.
    In particular, $B^\nu_2 = A_2 1^\nu \in \B$.
    Moreover, for all $k \geq 2$, $B^\nu_k \in G^*$ and, by induction, one checks that $s(B^\nu_k) = 2 + \sum_{i=3}^k 3 \cdot 2^{i-3} = 3 \cdot 2^{k-2} - 1$ and $s(A_{k+1}) = 3 \cdot 2^{k-2}$.
    Thus, $B^\nu_k < A_{k+1}$ and $\lambda(A_{k+1}) = X_0 < B^\nu_k$.
    This guarantees that $B^\nu_k \in \B$ for all $\nu \geq 0$ and $k \in \N$.
    
    \step{Step 2: We decompose the bracket on the basis}
    Applying the iterated Jacobi identity to $[A_1,B_p^\nu]$ yields
    \begin{equation} \label{eq:jacobi-a1bpnu}
        [A_1,B_p^\nu] = [[\dotsb[[A_1,B_2^\nu],A_3],\dotsc],A_p] + \sum_{k=2}^{p-1} [[B_k^\nu,[A_1,A_{k+1}]],\cdots],A_p].
    \end{equation}
    First, we remark that for each $k \in \intset{2, p-1}$, the element in the sum belongs to $\pm \ev(\B)$.
    Indeed, $(A_1, A_{k+1}) \in \B$ because $s(A_1) < s(A_{k+1})$ and $\lambda(A_{k+1}) = X_0 < A_1$.
    Then, $(B^\nu_k, (A_1, A_{k+1})) \in \B$ because $s(B^\nu_k) = 3\cdot 2^{k-2}-1 <  3\cdot2^{k-2} + 1 = s((A_1, A_{k+1}))$ and $\lambda((A_1,A_{k+1})) = A_1 < B^\nu_k$.
    Let $C_{k+1} := (B^\nu_k, (A_1, A_{k+1}))$.
    We prove by induction on $\ell \in \intset{k+1,p}$ that $C_\ell \in \B$ where $C_{\ell+1} := (A_\ell, C_\ell)$.
    Indeed, one checks that $s(C_\ell) = 3 \cdot 2^{\ell-2} = s(A_\ell)$ so that $A_\ell < C_\ell$ because $\lambda(A_\ell^*) = \lambda(A_\ell) = X_0 < \lambda(C_\ell)$.
    Moreover, $\lambda(C_\ell) < A_\ell$ because $\lambda(C_\ell) = A_{\ell-1}$ if $\ell \geq k+2$ and $\lambda(C_{k+1}) = B^\nu_k$.
    Hence the sum of $p-2$ terms is already expressed on $\B$.
    
    Second, we take care of the first term of the right-hand side of \eqref{eq:jacobi-a1bpnu}.
    By \cref{p:ad-k} (applied to the derivation $\dad_{X_1}$ on $\mathcal{L}(X)$, $k \leftarrow p$, $b_k \leftarrow A_2$, $b_{k-1} \leftarrow A_1$ and $b_i \leftarrow A_{p+1-i}$ for $1 \leq i \leq k-2$),
    \begin{equation}
        \begin{split}
        (-1)^{p-2} & [[\dotsb[[A_1,B_2^\nu],A_3],\dotsc],A_p] \\
        & =
        \sum_{j_1 + \cdots + j_p = \nu} (-1)^{\nu - j_p} \binom{\nu}{j_1, \dots, j_p}[A_p 1^{j_1},[A_{p-1} 1^{j_2}, [\dotsc , [A_1 1^{j_{p-1}},A_2]\dotsb]]]1^{j_p}.
        \end{split}
    \end{equation}
    One checks that all these brackets lie in $\B$.
    First, $s(A_1) < s(A_2)$ so $A_1 1^{j_{p-1}} < A_2$ and $\lambda(A_2) = X_0 < A_1 1^{j_{p-1}}$, hence $(A_1 1^{j_{p-1}}, A_2) \in \B$. 
    Then, by construction, for $k \geq 3$, $s(A_k) = 3 \cdot 2^{k-3} = \sum_{i=1}^{k-1} s(A_i)$ but $\lambda((A_k 1^{j_{1+p-k}})^*) = \lambda(A_k) = X_0 < \lambda((A_{k-1} 1^{j_{2+p-k}}, \cdots)^*) = A_{k-1}$ so the left members are lower than the right members. 
    Eventually since $s(A_i)$ is increasing, these brackets also satisfy the Hall condition that the left part of the right member is lower than the left member.

    Hence, the summation formula \eqref{eq:multinomial-sum} yields
    \begin{equation}
        \|[\cdots[A_1,B_2^\nu],A_3],\cdots],A_p]\|_\B = p^\nu.
    \end{equation}
    This concludes the proof, because all considered brackets of $\B$ are distinct.
\end{proof}

\subsection{Conclusion of the proof by optimization}
\label{ss:Mn-optimization}

We now prove the following result, which of course implies its unquantified counterpart \cref{thm:Mn-easy} stated in the introduction.
In \cref{rk:super-geom-theta}, we then discuss the consequences of our construction for $\theta$-based lower bounds.

\begin{theorem}
    For the Hall set $\B \subset \Br(X)$ constructed in \cref{ss:super-geom-basis}, there exists $c_0 > 0$ such that, for every $n \in \N^*$ large enough, there exist $a < b \in \B$ with $|a| = 2$ and $|b| = n$ such that
    \begin{equation} \label{eq:sqrt-n-n}
        \| [a, b] \|_\B \geq \left(\sqrt{\frac{n}{c_0 \ln n}}\right)^n.
    \end{equation}
\end{theorem}

\begin{proof}
    Let $n \geq 4$.
    As in the previous paragraph, we consider $a := A_1$ and $b := B^\nu_p$.
    We now optimize the choice of $\nu \in \N$ and $p \geq 2$ to obtain the claimed estimate.
    For $i \in \N$, we have $|A_i| = i + 1$.
    Moreover, for $\nu \in \N$ and $p \geq 2$,
    \begin{equation}
        \begin{split}
            |B^\nu_p| 
            & = |B^\nu_2| + \sum_{i=3}^{p} |A_i| 
            = \nu +3 + \sum_{i=3}^{p} (1+i) 
            = \nu + \frac{(p+1)(p+2)}{2} - 3.
        \end{split}
    \end{equation}
    For $p \geq 2$ such that $p^2 \leq n$, we define $\nu_p := n + 3 - \frac{(p+1)(p+2)}{2} \geq 0$.
    From now on $\nu := \nu_p$, which ensures that $|b| = n$.
    Moreover, $\nu_p \geq n - p^2$.
    By \cref{Prop:estimation_bstar},
    \begin{equation}
        \| [a, b] \|_\B \geq p^{\nu_p} \geq p^{n - p^2}.
    \end{equation}
    We must now choose $p \in \intset{2, \lfloor\sqrt{n}\rfloor}$ to take advantage of this lower bound. 
    We set
    \begin{equation}
        p := \left \lfloor \sqrt{\frac{n}{\ln n}} \right \rfloor. 
    \end{equation}
    For $n$ large enough, $p \geq 2$ and the construction is valid.
    Moreover
    \begin{equation}
        n-p^2 \geq n - \frac{n}{\ln n} = n \left(1 - \frac{1}{\ln n}\right).
    \end{equation}
    Thus
    \begin{equation}
        \| [a, b] \|_\B \geq M(n)^n
    \end{equation}
    where
    \begin{equation}
        M(n) := \left( \sqrt{\frac{n}{\ln n}} - 1 \right)^{1 - \frac{1}{\ln n}}
        \geq \sqrt{\frac{n}{c_0 \ln n}}
    \end{equation}
    for some appropriate choice of $c_0 > 0$ and $n$ large enough using elementary asymptotic analysis, which concludes the proof.
\end{proof}

\begin{remark} \label{rk:super-geom-theta}
    Let $a := A_1$ and $b := B^\nu_p$ as above.
    Then $\theta_a(b) = p - 1 + \nu$. 
    Indeed, for every $i \geq 2$, $a < A_i$ and $(a, A_i) \in \B$, and $a < X_1$ and $(a, X_1) \in \B$.
    Hence, a similar procedure as the one above proves that, 
    for every $\theta \in \N$ large enough, there exists $a, b \in \B$ with $\theta_a(b) = \theta$ and
    \begin{equation}
        \|[a,b]\|_\B \geq \frac{1}{\theta^2} \left( \frac{\theta}{e \ln \theta} \right)^\theta,
    \end{equation}
    which is therefore super-geometric (relative to $\theta$) and qualitatively not very far from the general $\theta$-based upper bound \eqref{eq:e-theta}.
    This lower bound can for example be obtained with the choice $p := \lfloor \theta / \ln \theta \rfloor$ and $\nu_p = \theta+1-p$ and then follows from elementary asymptotic analysis.
    
    In particular, this proves that there exists a Hall set exceeding the general geometric $\theta$-based lower bound of \cref{prop:lower-theta} and the geometric $\theta$-based upper bounds which were valid for length-compatible and Lyndon Hall sets.
\end{remark}

\section{Asymmetric estimates}
\label{sec:asym}

The main goal of this section is to prove \cref{p:thm-asym} which provides asymmetric estimates for the structure constants relative to any Hall set.
In particular, this gives a positive answer in the case of Hall bases to the open problem raised in \cite[Section 2.4.3]{2020arXiv201215653B} and allows to apply the conditional result in \cite[Section 4.4.3]{2020arXiv201215653B}) to such bases, which was our main motivation.

Throughout this section, $\B \subset \Br(X)$ is a Hall set where we single out the role of a particular indeterminate $X_0 \in X$.
For $b \in \B$, we denote by $n_0(b)$ the number of occurrences of $X_0$ in $b$, and $n(b) := |b|-n_0(b)$ the number of leaves of $b$ that are different from $X_0$.

We start with preliminary definitions in \cref{sec:asym-def} where we introduce a way to represent trees which reflects the asymmetry between $X_0$ and $X \setminus \{ X_0 \}$.
Then, in \cref{sec:asym-main}, we prove estimates for the norm of such brackets, culminating with \cref{p:asym-main}.
Eventually, we explain in \cref{sec:asym-reduction} how these notions allow to prove \cref{p:thm-asym}.

\subsection{An asymmetric representation of trees} \label{sec:asym-def}

We start by defining a new family of brackets, which stores additional $X_0$ factors throughout the whole tree and will allow us to manipulate these $X_0$ factors differently from the non-$X_0$ leaves.
 
\begin{definition}[Weighted brackets] \label{def:weighted-brackets}
    Let $A \subset \B$ with $X_0 \notin A$.
    We define $\Br^\star(A)$ by induction as follows, together with maps $\rho$ and $\omega$ from $\Br^\star(A)$ to $\N$.
    \begin{itemize}
        \item For each $a \in A$ and $\nu \in \N$, the pair $\ell := a,\nu$ is called a \emph{leaf} and belongs to $\Br^\star(A)$. 
        We will use the notations $\ell = a 0^\nu$ (instead of the pair notation), $\omega(\ell) = \rho(\ell) := \nu$, $\alpha(\ell) := a$, $|\ell| := 1$ and $L(\ell) := \{ a \}$.
        
        \item For each $t_1, t_2 \in \Br^\star(A)$ and $\nu \in \N$, the triple $t := t_1,t_2,\nu$ belongs to $\Br^\star(A)$. 
        We will use the notations $t = \langle t_1, t_2 \rangle 0^\nu$, $\omega(t) := \nu$, $\rho(t) := \rho(t_1) + \rho(t_2) + \omega(t)$, $|t| := |t_1|+|t_2|$ and $L(t) := L(t_1) \cup L(t_2)$.
    \end{itemize}
\end{definition}

\begin{example}
    If $A = \{ a_1, a_2, a_3 \}$, $t := \langle a_1 0^2, \langle a_2 0^7, a_1 \rangle 0^4 \rangle 0^3 \in \Br^\star(A)$ and can be visualised as
    \begin{equation}
        \Tree [.{$0^3$} ${a_1 0^2}$ [.$0^4$ ${a_2 0^7}$ $a_1$ ] ]
    \end{equation}
    In particular, $|t| = 3$ and $L(t) = \{ a_1, a_2 \}$.
\end{example}

\begin{remark}
    In this section, we use the notation $\langle \cdot, \cdot \rangle$ to denote the weighted bracket in $\Br^\star(A)$, while, in all previous sections of this paper, this notation is used for the bracket in $\Br(\Br(X))$.
    No confusion is possible here since we only manipulate $\Br^\star(A)$ in this section, and not $\Br(\Br(X))$.
    Also, we used the notations $|t|$ and $L(t)$ to denote the \emph{length} and \emph{leaves} of $t \in \Br^\star(A)$, without indexing them by $A$, to avoid overloading the formulas in the sequel.
    Again, no confusion is possible in this section.
\end{remark}

\begin{definition}[Canonical evaluation]
    Let $A \subset \B$ with $X_0 \notin A$.
    There is a natural evaluation mapping $\e^\star$ from $\Br^\star(A)$ to $\Br(X)$, defined by induction as
    \begin{itemize}
        \item for a leaf $\ell = a0^\nu$ with $a \in A$ and $\nu \in \N$, $\e^\star(\ell) := \ad_{X_0}^\nu(a)$,
        \item for a tree $t = \langle t_1, t_2 \rangle 0^\nu$ with $t_1, t_2 \in \Br^\star(A)$ and $\nu \in \N$, $\e^\star(t) := \ad_{X_0}^\nu((\e^\star(t_1), \e^\star(t_2)))$.
    \end{itemize}
    So elements of $\Br^\star(A)$ can be evaluated in $\Br(X)$ then in $\mathcal{L}(X)$.
\end{definition}

\begin{definition}
    We define
    \begin{itemize}
        \item $\B_\sharp := \B \setminus \{ X_0 \}$,
        \item $\BrStarAlpha(\B_\sharp)$ the subset of trees $t \in \Br^\star(\B_\sharp)$ such that $\{ X_0 \} \cup L(t)$ is alphabetic.
    \end{itemize}
\end{definition}

Along the decomposition algorithm, the additional $X_0$ factors of the considered trees will ``move upwards''. 
To track this phenomenon, we introduce the following definition.

\begin{definition}[Total depth of the $X_0$ factors]
    For $t \in \Br^\star(\B_\sharp)$ and $d \in \N$, let
	\begin{equation}
	    P_d(t) := 
	    \begin{cases}
	        d \nu & \textrm{if $t = a0^\nu$ is a leaf}, \\
	        d \nu + P_{d+1}(t_1) + P_{d+1}(t_2) & \textrm{if $t = \langle t_1, t_2 \rangle 0^\nu$}.
	    \end{cases}
	\end{equation}
	Hence, $P(t) := P_0(t)$ represents the sum of the $0$-based depths of the additional $X_0$ factors in $t$.
	As an example, if $t = \langle a_1 0^2, \langle a_2 0^7, a_3 \rangle 0^4 \rangle 0^3$, $P(t) = 0 \times 3 + (1 \times 2) + (1 \times 4 + 2 \times 7 + 2 \times 0) = 20$.
\end{definition}

\subsection{Estimates for the norm of asymmetric trees} \label{sec:asym-main}

The goal of this paragraph is to prove \cref{p:asym-main}, which yields an estimate of the norm of (evaluations of) brackets of $\BrStarAlpha(\B_\sharp)$.
We start with a particular case in \cref{p:asym-X0-min} and a decomposition result in \cref{p:asym-jac-main}.

\begin{lemma} \label{p:asym-X0-min}
    Let $t \in \BrStarAlpha(\B_\sharp)$. 
    Assume that $X_0 < \min L(t)$.
    Then
    \begin{equation} \label{eq:asym-X0-min}
        \| \ev(\e^\star(t)) \|_\B \leq |t|^{\rho(t)} (|t|-1)!.
    \end{equation}
\end{lemma}

\begin{proof}
    Since $X_0$ is minimal, the proof consists in ``pushing'' all the additional $X_0$ factors all the way down on the leaves of $t$ by iterating the Jacobi identity.
    We introduce
    \begin{equation}
        G := \{ \ad^k_{X_0}(g) ; \enskip g \in L(t), k \in \N \} \subset \B_\sharp.
    \end{equation}
    Since $\{ X_0 \} \cup L(t)$ is alphabetic and $X_0$ is minimal within this set, iterating \cref{lem:alphabetic} proves that $G$ is alphabetic.
    Hence, by \cref{Prop:n!}, for any $h \in \Br(G)$, $\| \e(h) \|_\B \leq (|h|-1)!$.
    Thus, to prove \eqref{eq:asym-X0-min}, it is sufficient to prove that $\ev (\e^\star(t))$ is the sum of at most $|t|^{\rho(t)}$ evaluations of brackets of $\Br(G)$ of length $|t|$.
    
    We proceed by induction on $|t| \geq 1$.
    When $|t| = 1$, $t = g 0^\nu$, so, by definition, $ \e^\star(t) = \e (h)$ where $h \in \Br(G)$ is the leaf $\ad_{X_0}^\nu(g) \in G$.
    Assume that $|t| \geq 2$, then $t = \langle t_1, t_2 \rangle 0^\nu$ with $t_1, t_2 \in \Br^\star(\B_\sharp)$, $|t_1| + |t_2| = |t|$, $\rho(t) = \rho(t_1) + \rho(t_2) + \nu$.
    The Leibniz formula proves that
    \begin{equation}
        \ev (\e^\star(t))
         = \sum_{j=0}^\nu \binom{\nu}{j}
         \left[
         \e^\star(t_1 0^j),
         \e^\star(t_2 0^{\nu-j})
         \right].
    \end{equation}
    Since $|t_1 0^j| < |t|$ and $|t_2 0^{\nu-j}| < |t|$, we can apply the induction hypothesis.
    Hence, we know that $t_1 0^j$ is the sum of at most $|t_1|^{\rho(t_1)+j}$ elements of $\Br(G)$ of length $|t_1|$, and $t_2 0^{\nu-j}$ of at most $|t_2|^{\rho(t_2)+\nu-j}$ elements of $\Br(G)$ of length $|t_2|$.
    Moreover, since
    \begin{equation}
        \sum_{j=0}^\nu \binom{\nu}{j}
        |t_1|^{\rho(t_1)+j}
        |t_2|^{\rho(t_2)+\nu-j}
        = |t_1|^{\rho(t_1)} |t_2|^{\rho(t_2)}
        (|t_1|+|t_2|)^{\nu}
        \leq |t|^{\rho(t)},
    \end{equation}
    we conclude that $t$ is indeed the sum of at most $|t|^{\rho(t)}$ elements of $\Br(G)$ of length $|t|$.
\end{proof}

\begin{lemma} \label{p:asym-jac-main}
    Let $t \in \BrStarAlpha(\B_\sharp)$ with $|t| \geq 2$.
    Assume that there exists a leaf $\ell$ of $t$ such that $\alpha(\ell) = \min L(t) < X_0$ and $\omega(\ell) = 0$.
    Then
    \begin{equation} \label{eq:asym-jac-main}
        \pm \ev(\e^\star(t)) = \sum_{k=1}^{q}  \ev(\e^\star(h_k)) 
        + \sum_{k=1}^{r} \ev(\e^\star(t_k)) 
    \end{equation}
    where $q \in \intset{1,|t|-1}$, $r \in \intset{0,\rho(t)}$ and
    \begin{itemize}
        \item $h_k \in \Br^\star(L(t) \cup (\alpha(\ell), L(t)))$ with $|h_k| = |t|-1$, $\rho(h_k) = \rho(t)$ and $P(h_k) \leq P(t)$,
        \item $t_k \in \Br^\star(L(t))$ with $|t_k| = |t|$, $\rho(t_k) = \rho(t)-1$ and $P(t_k) \leq P(t) - k$. 
    \end{itemize}
    In particular, $h_k \in \BrStarAlpha(\B_\sharp)$ by \cref{lem:alphabetic}.
\end{lemma}

\begin{proof}
    Let $g := \alpha(\ell)$ be the minimal element of $L(t)$.
    Heuristically, the decomposition \eqref{eq:asym-jac-main} comes from distributing $g$ over all the non-$X_0$ leaves and $X_0$ factors of its neighbor thanks to the Jacobi identity.
    The first sum corresponds to terms where $g$ hits a non-$X_0$ leaf of $t$, yielding a shorter tree on a new alphabetic subset.
    The second sum corresponds to terms where $g$ hits an additional $X_0$ factor.
    An appropriate ordering of the created brackets $t_k$ yields the delicate point of the estimate $P(t_k) \leq P(t)-k$, which proves that the $X_0$ factors move upwards in this process.

    \bigskip \noindent
    \emph{Step 1: Case when $t = \langle g, w \rangle$.}
    We proceed by induction on $|w| \geq 1$.
    We actually propagate the following stronger equality (without the sign alternative in \eqref{eq:asym-jac-main})
    \begin{equation}
        \ev(\e^\star(t)) = \sum_{k=1}^{q}  \ev(\e^\star(h_k)) 
        + \sum_{k=1}^{r} \ev(\e^\star(t_k))
    \end{equation}
    and the stronger conditions that for all $d \in \N$, $P_d(h_k) \leq P_d(t)$ and $P_d(t_k) \leq P_d(t) - k$.
    We will use that, since we assumed  $\omega(\ell) = 0$, $P_d(g) = 0$.
    
    \medskip
    
    First, if $|w| = 1$, then $w = h 0^\nu$ for some $h \in \B_\sharp$ and $\nu := \omega(w) \geq 0$.
    Distributing $g$ using the Jacobi identity yields
    \begin{equation}
        \ev(\e^\star(\langle g, h0^\nu \rangle))
        = \ev(\e^\star((g, h) 0^\nu))
        + \sum_{k=1}^{\nu} \ev(\e^\star( 
            \langle (g, X_0), h 0^{\nu-k} \rangle 0^{k-1}
        )),
    \end{equation}
    which is \eqref{eq:asym-jac-main} with $h_1 := (g, h) 0^\nu$ and $t_k := \langle (g, X_0), h 0^{\nu-k} \rangle 0^{k-1}$.
    Moreover, for every $d \in \N$, $P_d(\langle g, h0^\nu \rangle) = (d+1) \nu$ and $P_d(t_k) = d(k-1) + (d+1) (\nu-k) = P_d(t) - k - d$. 
    So $P_d(t_k) \leq P_d(t) - k$.
    
    Now assume that $|w| \geq 2$. 
    Then $w = \langle w_1, w_2 \rangle 0^\nu$ with $w_1, w_2 \in \Br^\star(L(w))$ and $\nu := \omega(w) \geq 0$.
    Distributing $g$ using the Jacobi identity yields
    \begin{equation}
        \begin{split}
            \ev(\e^\star(\langle g, w \rangle))
        & = \ev(\e^\star( \langle \langle g, w_1 \rangle, w_2 \rangle 0^\nu ))
        + \ev(\e^\star( \langle w_1, \langle g, w_2 \rangle \rangle 0^\nu )) \\
       & \quad + \sum_{k=1}^{\nu} \ev(\e^\star( 
            \langle (g, X_0), \langle w_1, w_2 \rangle  0^{\nu-k} \rangle 0^{k-1} ))
        \end{split}
    \end{equation}
    We apply the induction hypothesis to $\langle g, w_1 \rangle$ and $\langle g, w_2 \rangle$, indexing all variables with $1$ or $2$ respectively.
    Thus
    \begin{equation}
        \begin{split}
            \ev(\e^\star(\langle g, w \rangle))
        & = 
        \sum_{k=1}^{q_1}  \ev(\e^\star(\langle h^1_k, w_2 \rangle 0^\nu)) 
        + \sum_{k=1}^{q_2}  \ev(\e^\star(\langle w_1, h^2_k \rangle 0^\nu)) 
          \\
        & + \sum_{k=1}^{r_1} \ev(\e^\star(\langle t^1_k, w_2 \rangle 0^\nu))
        + \sum_{k=1}^{r_2} \ev(\e^\star(\langle w_1, t^2_k \rangle 0^\nu))
        \\
       & \quad + \sum_{k=1}^{\nu} \ev(\e^\star( 
            \langle (g, X_0), \langle w_1, w_2 \rangle  0^{\nu-k} \rangle 0^{k-1} )).
        \end{split}
    \end{equation}
    Let us check that this is the expected decomposition.
    \begin{itemize}
        \item The first two sums yield the first sum of \eqref{eq:asym-jac-main}. 
        One checks that
        \begin{itemize}
            \item $|\langle h^1_k, w_2 \rangle 0^\nu| = |h^1_k| + |w_2| = |\langle g, w_1 \rangle|-1+|w_2| = |t| -1$,
            \item $\rho(\langle h^1_k, w_2 \rangle 0^\nu) = \nu + \rho(h^1_k) + \rho(w_2) = \nu + \rho(\langle g, w_1 \rangle) + \rho(w_2) = \rho(t)$,
            \item for every $d \in \N$, $P_d(\langle h^1_k, w_2 \rangle 0^\nu) = d \nu + P_{d+1}(h^1_k) + P_{d+1}(w_2) \leq d \nu + P_{d+1}(\langle g, w_1 \rangle) + P_{d+1}(w_2) = d \nu + P_{d+2}(w_1) + P_{d+1}(w_2) 
            \leq (d+1) \nu + P_{d+1}(\langle w_1, w_2 \rangle) = P_d(t)$,
            \item and similarly for the trees $\langle w_1, h^2_k \rangle 0^\nu$,
            \item and eventually, $q_1 + q_2 = |w_1| + |w_2| = |t|-1$.
        \end{itemize}
        
        \item The last three sums yield the second sum of \eqref{eq:asym-jac-main}.
        The properties on the lengths, the leaves and the values of $\rho$ of the trees are straightforward.
        We focus on the delicate estimate of the values of $P$.
        Let $d \in \N$.
        \begin{itemize}
            \item First, $P_d(t) = P_d(\langle g, \langle w_1, w_2 \rangle 0^\nu \rangle) = (d+1) \nu + P_{d+2}(w_1) + P_{d+2}(w_2)$.
            \item For $k \in \intset{1,\nu}$, $P_d(\langle (g, X_0), \langle w_1, w_2 \rangle  0^{\nu-k} \rangle 0^{k-1}) = d (k-1) + (d+1) (\nu-k) + P_{d+2}(w_1) + P_{d+2}(w_2) = P_d(t) - k - d \leq P_d(t) - k$.
            \item For $k \in \intset{1,r_1}$, since, by the induction hypothesis, $P_{d+1}(t^1_k) \leq P_{d+1}(\langle g, w_1 \rangle) - k = P_{d+2}(w_1) - k$, we have 
            $P_d(\langle t^1_k, w_2 \rangle 0^\nu)
            = d \nu + P_{d+1}(w_2) + P_{d+1}(t^1_k)
            \leq d \nu + P_{d+2}(w_2) - \rho(w_2) + P_{d+2}(w_1) - k
            = P_{d}(t) - \nu - \rho(w_2) - k$.
            \item Similarly, for $k \in \intset{1,r_2}$, $P_d(\langle w_1, t^2_k \rangle 0^\nu) \leq P_d(t) - \nu - \rho(w_1) - k$.
        \end{itemize}
        Recall that $r_1 \leq \rho(w_1)$ and $r_2 \leq \rho(w_2)$.
        Let $r := \nu + r_1 + r_2$ and define, for $k \in \intset{1,r}$
        \begin{equation}
            t_k := \begin{cases}
                \langle (g, X_0), \langle w_1, w_2 \rangle  0^{\nu-k} \rangle 0^{k-1}
                & \textrm{for } k \in \intset{1,\nu}, \\
                \langle t^1_{k-\nu}, w_2 \rangle 0^\nu
                & \textrm{for } k \in \intset{\nu+1, \nu+r_1}, \\
                \langle w_1, t^2_{k-\nu-r_1} \rangle 0^\nu
                & \textrm{for } k \in \intset{\nu+r_1+1, \nu+r_1+r_2}.
            \end{cases}
        \end{equation}
        The previous estimates prove that, for this labeling, $P_d(t_k) \leq P_d(t) - k$.
    \end{itemize}

    \bigskip
    \noindent \emph{Step 2: General case.}
    Up to the left-right anti-symmetry (so up to a sign in $\mathcal{L}(X)$), we can assume that there is a $w \in \Br^\star(L(t))$ and $\mu \geq 0$ such that $t' 0^\mu$ where $t' := \langle g, w \rangle$ is a subtree of $t$.
    By the previous step,  
    \begin{equation}
        \ev(\e^\star(t')) = \sum_{k=1}^{q'} \ev(\e^\star(h_k')) 
        + \sum_{k=1}^{r'} \ev(\e^\star(t_k')) 
    \end{equation}
    where $q' \in \intset{1,|w|}$, $r' \in \intset{0,\rho(w)}$ and
    \begin{itemize}
        \item $h_k' \in \Br^\star(L(w) \cup (g, L(w)))$ with $|h_k'| = |w|$, $\rho(h_k') = \rho(w)$ and $P(h_k') \leq P(w)$,
        \item $t_k' \in \Br^\star(L(w))$ with $|t_k'| = |w|+1$, $\rho(t_k') = \rho(w)-1$ and $P(t_k') \leq P(w) - k$. 
    \end{itemize}
    This proves \eqref{eq:asym-jac-main} where $h_k$ is obtained by replacing $t'$ by $h_k'$ in $t$, and $t_k$ is obtained by replacing $t'$ by $t_k'$ in $t$.
    The properties on the leaves, the lengths and the values of $\rho$ for the $h_k$ and $t_k$ can be checked easily.
    The only delicate part is to verify the values of $P$.
    Let $d$ be the $0$-based depth of $t'$ in $t$.
    Then $P(h_k) = P(t) - P_d(t') + P_d(h_k') \leq P(t)$ because we proved $P_d(h_k') \leq P_d(t')$.
    Similarly, $P(t_k) = P(t) - P_d(t') + P_d(t_k') \leq P(t) - k$.
\end{proof}

\begin{proposition} \label{p:asym-main}
    Let $t \in \BrStarAlpha(\B_\sharp)$. 
    Then
    \begin{equation} \label{eq:asym-t-gen}
        \| \ev(\e^\star(t)) \|_\B \leq C(|t|)^{\rho(t)} (|t|-1)!,
    \end{equation}
    where, for $n \in \N^*$,
    \begin{equation} \label{eq:asym-Cn}
        C(n) := \begin{cases}
            1 & \textrm{for } n = 1, \\
            2^{\frac{n(n-1)}{2}+1} & \textrm{when } n\geq 2.
        \end{cases}
    \end{equation}
\end{proposition}

\begin{proof}
    For $n \geq 1$, $\rho \geq 0$ and $P \geq 0$, we introduce
    \begin{equation} \label{eq:def-fnrhop}
        F_n(\rho,P) := \max \left\{ \frac{\| \ev(\e^\star(t)) \|_\B}{(|t|-1)!} ; \enskip t \in \BrStarAlpha(\B_\sharp), |t| = n, \rho(t) \leq \rho, P(t) \leq P \right\}.
    \end{equation}
    In particular, for each $n \geq 1$, $F_n$ is a non-decreasing function of $\rho$ and $P$.
    Moreover, since, for every $t \in \Br^\star(\B_\sharp)$, $P(t) \leq (|t|-1) \rho(t)$, we have, for every $P\geq 0$, 
    \begin{equation} \label{eq:fn-barfn}
        F_n(\rho, P) \leq \overline{F}_n(\rho) := F_n(\rho, (n-1)\rho).
    \end{equation}
    
    \bigskip
    \noindent \emph{Step 1: We compute the values of $\overline{F}_n(\rho)$ for $n = 1$ or $\rho = 0$.}
    First, for every $\rho \geq 0$,
    \begin{equation} \label{eq:f-n1}
        \overline{F}_1(\rho) = 1.
    \end{equation}
    The proof is by induction on $\rho \geq 0$.
    Let $t \in \BrStarAlpha(\B_\sharp)$ such that $|t| = 1$ and $\rho(t) = \rho$, so $t = g0^\rho$ for some $g \in \B_\sharp$.
    If $\rho = 0$, $t$ is already a leaf labeled by $g \in \B$, without additional $X_0$ factors, so $\e^\star(t) = g \in \B$ and thus $\|\ev(\e^\star(t))\|_\B = 1$.
    Otherwise, $\rho > 0$. 
    If $X_0$ is minimal among $\{ X_0, g \}$, $\e^\star(t) = \ad_{X_0}^\rho(g) \in \B$ so  $\|\ev(\e^\star(t))\|_\B = 1$.
    If $g$ is minimal among $\{ X_0, g \}$,  $\ev(\e^\star(t)) = - \ev( \e^\star( (g, X_0)0^{\rho-1} ))$, so the proof follows by induction.
    
    Second, for every $n \geq 1$,
    \begin{equation} \label{eq:f-rho0}
        \overline{F}_n(0) \leq 1.
    \end{equation}
    Indeed, for $t \in \BrStarAlpha(\B_\sharp)$ with $\rho(t) = 0$ and $|t| = n$, since there is no additional $X_0$ factors, $\e^\star(t)$ is equal, in $\Br(X)$, to the evaluation of a bracket of length $n$ over an alphabetic subset of $\B_\sharp$.
    So \cref{Prop:n!} yields $\| \ev(\e^\star(t)) \|_\B \leq (n-1)!$.
    
    \bigskip
    \noindent \emph{Step 2: Proof of a functional inequality using the Jacobi identity.} We prove that, for every $n \geq 2$, $\rho \geq 1$ and $P \geq 0$,
    \begin{equation} \label{eq:asym-fnrhop-ineq}
        F_n(\rho,P) \leq 
        \max \left\{ 
        n^\rho, \enskip  
        F_{n-1}(\rho,P) + \sum_{k=1}^{+\infty} F_n( \rho-1,P-k) \right\},
    \end{equation}
    with the convention that $F_n(\rho', P') = 0$ when $P' < 0$.
    
    \medskip
    
    Let $t \in \BrStarAlpha(\B_\sharp)$ with $|t| = n$, $\rho(t) \leq \rho$ and $P(t) \leq P$. 
    We separate three cases.
    Let $g := \min (\{ X_0 \} \cup L(t))$.
    \begin{itemize}
        \item \emph{Case $g = X_0$.}
        Then \cref{p:asym-X0-min} proves that $\frac{\| \ev( \e^\star(t)) \|_\B}{(|t|-1)!} \leq n^\rho$.
        This case yields the first term of the right-hand side maximum in \eqref{eq:asym-fnrhop-ineq}.
        
        \item \emph{Case $g \in L(t)$ and there is a leaf $\ell$ of $t$ such that $g = \alpha(\ell)$ and $\omega(\ell) = 0$.}
        Using the decomposition \eqref{eq:asym-jac-main} with the notations of \cref{p:asym-jac-main} and the monotony of $F_{n-1}$ and $F_n$, we obtain
        \begin{equation}
            \begin{split}
            \frac{\| \ev(\e^\star(t)) \|_\B}{(n-1)!}
            & \leq \frac{1}{(n-1)!} \sum_{k=1}^q \| \e^\star(h_k) \|_\B
            + \frac{1}{(n-1)!} \sum_{k=1}^r \| \e^\star(t_k) \|_\B \\
            & \leq \frac{(n-1)}{(n-1)!} (n-2)! F_{n-1}(\rho,P) + \sum_{k=1}^r F_n(\rho-1,P-k), 
            \end{split}
        \end{equation}
        which yields the second term of the right-hand side maximum in \eqref{eq:asym-fnrhop-ineq}.
        
        \item \emph{Case $g \in L(t)$ but for each leaf $\ell$ of $t$ such that $g = \alpha(\ell)$, $\omega(\ell) > 0$.}
        Take such a leaf and let $\nu := \omega(\ell) > 0$, we have $\ev (\e^\star(t)) = - \ev(\e^\star(\bar t))$ where $\ell = g0^\nu$ in $t$ has been replaced by $(g,X_0) 0^{\nu-1}$ in $\bar t$.
        By \cref{lem:alphabetic}, $\{ X_0 \} \cup L(\bar t)$ is alphabetic.
        Moreover, $|\bar t| = |t|$, $\rho(\bar t) = \rho(t)-1$ and $P(\bar t) \leq P(t)$.
        Iterating this procedure if necessary brings us eventually back to the first two cases.
    \end{itemize}
    
    \bigskip
    \noindent \emph{Step 3: Resolution of the functional inequality and conclusion.}
    Let $n \geq 2$, $\rho \geq 1$ and $P \geq 0$.
    Let $f_{n,\rho,P} := \max \{ n^\rho, F_{n-1}(\rho, P) \} \in \R_+$.
    Iterating \eqref{eq:asym-fnrhop-ineq} and using the monotony of $f_{n,\rho,P}$ with respect to $\rho$ and $P$ yields
    \begin{equation}
        \begin{split}
            F_n(\rho, P) & \leq f_{n,\rho,P} + \sum_{k_1=1}^{+\infty} F_n(\rho-1,P-k_1) \\
            & \leq f_{n,\rho,P} + \sum_{k_1=1}^{+\infty} \left(
            f_{n,\rho-1,P-k_1} 
            + \sum_{k_2=1}^{+\infty} F_{n}(\rho-2,P-k_1-k_2)
            \right) \\
            & \leq \cdots \\
            & \leq f_{n,\rho,P}
            \left( 1 + \sum_{r=1}^{\rho-1} \left| 
            \left\{  (k_1, \dotsc, k_r) \in (\N^*)^r ; \enskip k_1 + \dotsb + k_r \leq P \right\} \right| \right) \\
           & \quad + \overline{F}_n(0) \left| 
            \left\{  (k_1, \dotsc, k_\rho) \in (\N^*)^\rho ; \enskip k_1 + \dotsb + k_\rho \leq P \right\} \right|.
        \end{split}
    \end{equation}
    Using the classical bound for the number of compositions of an integer therefore yields
    \begin{equation}
        F_n(\rho, P) \leq 2^P \max \{ n^\rho, F_{n-1}(\rho,P), \overline{F}_n(0) \}.
    \end{equation}
    By \eqref{eq:f-rho0}, $\overline{F}_n(0) \leq 1 \leq n^\rho$.
    Thus, using \eqref{eq:fn-barfn},
    \begin{equation}
        \overline{F}_n(\rho) \leq 2^{(n-1)\rho} \max \{ n^\rho, \overline{F}_{n-1}(\rho) \}.
    \end{equation}
    Recalling \eqref{eq:f-n1}, we obtain $\overline{F}_2(\rho) \leq 2^{\rho} \max \{ 2^\rho, 1 \} = 2^{2\rho} = C(2)^\rho$.
    Then for $n \geq 3$, with $C$ defined as in \eqref{eq:asym-Cn}, one can ignore the term $n^\rho$ in the maximum since $n \leq C(n-1)$ for $n \geq 3$. We thus obtain
    \begin{equation} \label{eq:fnbar-est}
        \overline{F}_n(\rho) \leq C(n)^\rho,
    \end{equation}
    Hence, if $t \in \BrStarAlpha(\B_\sharp)$, recalling \eqref{eq:def-fnrhop}, 
    \begin{equation}
        \| \ev(\e^\star(t)) \|_\B \leq (|t|-1)! \times  \overline{F}_{|t|}(\rho(t)),
    \end{equation}
    which, together with \eqref{eq:fnbar-est}, concludes the proof of \eqref{eq:asym-t-gen} for $n \geq 2$ and $\rho \geq 1$ (the cases $n = 1$ or $\rho = 0$ were already covered in the initializations step).
\end{proof}

\subsection{Proof of the main generic asymmetric estimate} \label{sec:asym-reduction}

We prove the following refined version of \cref{p:thm-asym}.
For $a < b \in \B$, let $\rho_a(b)$ denote the number of occurrences of $X_0$ as a leaf of $\rf{a}{b}$ and $n_a(b) := \theta_a(b) - \rho_a(b)$.
In particular, by construction, $\rho_a(b) \leq n_0(b)$ and $n_a(b) \leq n(b)$.

\begin{theorem} \label{p:thm-asym-refined}
    Let $a < b \in \B$. 
    Then
    \begin{equation} \label{eq:asym-main-ab}
        \| [a, b] \|_\B \leq C(n_a(b) + 1)^{\rho_a(b)} n_a(b)!,
    \end{equation}
    where $C$ is the non-decreasing sequence defined in \eqref{eq:asym-Cn}.
\end{theorem}

\begin{proof}
    Let $a < b \in \B$.
    When $b = X_0$, by the axioms of a Hall set, $(a,b) \in \B$ so $\| [a,b] \|_\B = 1$, which corresponds to the estimate with $n_a(b) = 0$, $\rho_a(b) = 1$ and $C(1) = 1$.
    When $b \neq X_0$, we separate two cases.
    \begin{itemize}
    
	\item \emph{Case $X_0 \notin L_\B(\rf{a}{b})$} (which happens in particular when $a = X_0$ or if $X_0$ is minimal in $\B$).
    Then $\rho_a(b) = 0$ and $\theta_a(b) = n_a(b)$.
	By \cref{Prop:Ta(b)_alphabetic}, $\{ a \} \cup L_\B(\rf{a}{b})$ is an alphabetic subset of $\B$.
	Seeing $(a, b)$ as $\e (\langle a , \rf{a}{b} \rangle)$, a bracket of $\theta_a(b) + 1 = n_a(b) + 1$ elements of an alphabetic subset, yields $\| [a, b] \|_\B \leq n_a(b)!$ by \cref{Prop:n!}.
	
	\item \emph{Case $X_0 \in L_\B(\rf{a}{b})$} (which implies $a \neq X_0$).
    Let $A := (\{ a \} \cup L_\B(\rf{a}{b})) \setminus \{ X_0 \} \subset \B_\sharp$.
	Then, collecting all the $X_0$ factors in the map $\omega$, there exists a tree $t \in \Br^\star(A)$ with $|t| = n_a(b) + 1$ and $\rho(t) = \rho_a(b)$ such that $\ev(\e(\langle a , \rf{a}{b} \rangle)) = \pm \ev (\e^\star( t ))$.
	Moreover, $L(t) \cup \{ X_0 \} = A \cup \{ X_0 \} = \{ a \} \cup L_\B(\rf{a}{b})$ is alphabetic by \cref{Prop:Ta(b)_alphabetic}.
	So $t \in \BrStarAlpha(\B_\sharp)$ and the estimate follows from \cref{p:asym-main}.
    \end{itemize}
    Hence, \eqref{eq:asym-main-ab} is valid in both cases.
\end{proof}

\begin{remark}
    Estimate \eqref{eq:asym-main-ab} implies that, for each Hall set and with respect to each indeterminate, the asymmetric asymptotic growth is at most geometric.
    
    \cref{thm:LB2} proves that, for each Hall set, the asymmetric asymptotic growth with respect to $\max X$ is at least geometric.
    
    Since, in each Hall set, $\min X = \min \B$, the first case of the proof of \cref{p:thm-asym-refined} implies that the asymmetric asymptotic growth with respect to $\min X$ is in fact bounded.
\end{remark}

\section{Perspectives and open problems}

We present some possible extensions related with the growth of structure constants of free Lie algebras, which we find interesting.

\medskip

In \cref{sec:rewriting}, we mentioned that structure constants are related with the (space and time) complexity of the rewriting algorithm which allows to decompose $[a,b]$ on the basis.
In this direction, one could start by considering that the comparison operation $a < b$ and the bracket creation operation are elementary operations of time-complexity $\mathcal{O}(1)$ and then investigate the time-complexity of the $\texttt{Rewrite}$ function described in \cref{p:rec-theta}.

\medskip

In \cref{Sec:refined_general_bound}, we proved that $\| [a, b] \|_\B \leq \lfloor e(n-1)! \rfloor$ when $|b| = n$ and that this estimate is sharp is the sense of \cref{cor:x-infini-sature}.
The asymptotic optimality case for large $n$ uses an infinite set $X$.
When $|X| = 2$, our construction of \cref{sec:super-geom} only provides an asymptotic growth behaving roughly like $\sqrt{n}^n$ (see \eqref{eq:sqrt-n-n}), which is quite far from $n^n$.
Therefore, an interesting direction would be to investigate the dependency of the optimal estimates with respect to the cardinal of $X$.

\medskip

The examples in \cref{sec:geom-lower} prove that the structure constants of free Lie algebras relative to Hall bases grow at least geometrically with the length of the involved brackets.
Hence, in this sense, the product operation has an exponential size within these bases.
An interesting open problem would be to determine if there exist other bases of $\mathcal{L}(X)$ (see \cref{s:other-bases} for a short discussion) having a bounded size, a polynomial size, or at least a sub-exponential size, in the sense that
\begin{equation}
    \limsup_{n \to + \infty} \| [a, b] \|_\B < + \infty, \quad
    \limsup_{n \to + \infty} \frac{\ln \| [a, b] \|_\B}{\ln n} < + \infty
    \quad \mathrm{or} \quad 
    \limsup_{n \to + \infty} \frac{\ln \| [a, b] \|_\B}{n} = 0,
\end{equation}
where the suprema are taken over $a$ and $b$ elements of the basis such that $[a,b] \in \mathcal{L}_n(X)$.

\medskip

In \cref{sec:asym}, we proved asymmetric estimates of the form $C(n_a(b)+1)^{\rho_a(b)} n_a(b)!$, where $C(n)$ behaves roughly like $2^{n^2}$ (see \eqref{eq:asym-Cn}).
It would be interesting to investigate whether this behavior is optimal or if one can improve our proof to obtain a better dependency on $n$ of the geometric rate $C(n)$, along with a ``critical asymmetric basis'' which achieves this rate.
One could also investigate how the maximal or minimal growth of $C(n)$ depends on $|X|$.

\medskip

Eventually, this paper focuses mainly on the ``worst case'' size of the Lie bracket operation within Hall bases.
From the point of view of computational applications, it could also be interesting to investigate its average size for fixed length, e.g.\ through the quantities
\begin{equation}
    \gamma_n(\B) 
    := \operatorname{mean} \left\{
    \| [a, b] \|_\B ; \enskip
    a < b \in \B, |a|+|b| = n
    \right\}.
\end{equation}
What is the asymptotic growth of $\gamma_n(\B)$?
As in our case, one could start by considering this question for the classical length-compatible or Lyndon bases.

\appendix

\section{Detailed example of the decomposition algorithm}
\label{sec:ex-bpi}

In this paragraph, to facilitate the comprehension of the proof of \cref{thm:en1}, we propose an illustration of the execution of the construction of $B_\pi$ along the path $\pi = (4, 5, 7)$ for $\rf{a}{b}$ having the structure of \eqref{ex:tab-b1min}.
The construction starts with $B_\emptyset = \langle a, \rf{a}{b} \rangle$, so
\begin{equation}
   B_\emptyset = \Tree [. $a$ [. [. $b_3$ [. $b_1$ $b_2$ ] ] [. $b_4$ [. [. $b_5$ [. $b_6$ $b_7$ ] ] $b_8$ ] ] ] ]
\end{equation}
Then, $a$ is pushed on $b_4$. So
\begin{equation}
   B_{(4)} = \Tree [. [. $b_3$ [. $b_1$ $b_2$ ] ] [. $(a,b_4)$ [. [. $b_5$ [. $b_6$ $b_7$ ] ] $b_8$ ] ] ]
\end{equation}
Then, assuming that it is minimal, $(a,b_4)$ is pushed on $b_5$. So
\begin{equation}
   B_{(4,5)} = \Tree [. [. $b_3$ [. $b_1$ $b_2$ ] ] [. [. $((a,b_4),b_5)$ [. $b_6$ $b_7$ ] ] $b_8$ ] ]
\end{equation}
Eventually, assuming that $((a,b_4),b_5)$ is minimal, it is pushed on $b_7$. So
\begin{equation}
   B_{(4,5,7)} = \Tree [. [. $b_3$ [. $b_1$ $b_2$ ] ] [. [. $b_6$ $a_{4,5,7}$ ] $b_8$ ] ]
\end{equation}
where $a_{(4,5,7)} = (((a,b_4),b_5),b_7)$.
The process stops here since the sibling of $a_{4,5,7}$ in $B_{(4,5,7)}$ is a leaf, so there is no longer path starting with $(4,5,7)$.

\section{A Leibniz-type inversion rule}

We prove a formula linked with the general Leibniz rule, which is used throughout the paper to balance iterated Lie brackets away from a particular term.

\begin{definition}[Multinomial coefficient]
    For $k \in \N^*$, $j_1, \dots, j_k \in \N$ and $\nu = j_1 + \dotsb + j_k$, we define the multinomial coefficient
    \begin{equation}
        \binom{\nu}{j_1, \dotsc, j_k} := \frac{\nu!}{j_1!\dotsb j_k!}.
    \end{equation}
    Moreover, for all $\nu \geq 0$,
    \begin{equation} \label{eq:multinomial-sum}
        \sum_{j_1 + \dotsb + j_k = \nu} \binom{\nu}{j_1, \dots, j_k} = k^\nu
    \end{equation}
\end{definition}

\begin{lemma} \label{p:ad-2}
    Let $\mathcal{A}$ be an algebra and $D$ a derivation on $\mathcal{A}$. 
    For every $\nu \in \N$ and $b_1, b_2 \in \mathcal{A}$,
    \begin{equation}
        b_1 (D^\nu b_2) = \sum_{j=0}^\nu (-1)^j \binom{\nu}{j} D^{\nu-j} ((D^j b_1) b_2).
    \end{equation}
\end{lemma}

\begin{proof}
    The proof is by induction on $\nu \geq 0$.
    The case $\nu = 0$ is trivial.
    We assume the property up to some $\nu \in \N$ and prove it for $\nu + 1$.
    Since $D$ is a derivation,
    \begin{equation}
        b_1 (D^{\nu+1} b_2) = 
        D( b_1 (D^\nu b_2)) - (D b_1) (D^\nu b_2).
    \end{equation}
    Applying the induction hypothesis on both terms and Pascal's formula concludes the proof.
\end{proof}

\begin{lemma} \label{p:ad-k}
    Let $\mathcal{A}$ be an algebra and $D$ a derivation on $\mathcal{A}$. 
    For $k \geq 2$, $\nu \in \N$ and $b_1, \dotsc, b_k \in \mathcal{A}$,
    \begin{equation}\label{equilibre_0_k}
            b_1 \dotsb b_{k-1} (D^\nu b_k)
            = \sum_{j_1 + \cdots + j_k = \nu} (-1)^{\nu - j_k} \binom{\nu}{j_1, \dots, j_k}
            D^{j_k} ((D^{j_1} b_1) \dotsb (D^{j_{k-1}} b_{k-1}) b_k),
    \end{equation}
    where all products of $k$ terms are understood as being right-nested if $\mathcal{A}$ is non-associative (i.e.\ by convention in this result $a b c d$ denotes $(a(b(cd)))$).
\end{lemma}

\begin{proof}
    The case $k = 2$ is covered in \cref{p:ad-2}.
    We now proceed by induction on $k$, assuming that the formula holds for some $k \geq 2$, and proving it for $k+1$ elements.
    By the induction hypothesis and linearity, one has:
    \begin{equation}
            b_1 b_2 \dotsb b_k (D^\nu b_{k+1})
            = \sum_{j_2 + \cdots + j_{k+1} = \nu} (-1)^{\nu - j_{k+1}} \binom{\nu}{j_2, \dots, j_{k+1}}
            b_1 D^{j_{k+1}} ((D^{j_2} b_2) \dotsb (D^{j_k} b_k) b_{k+1})
    \end{equation}
    Applying the two-elements case to each term yields
    \begin{equation}
        \begin{split}
            b_1 D^{j_{k+1}} ((D^{j_2} b_2) & \dotsb (D^{j_k} b_k) b_{k+1}) \\
            & = \sum_{j_1+j_{k+1}'=j_{k+1}} (-1)^{j_{1}} \binom{j_{k+1}}{j_1}  D^{j_{k+1}'} ( 
                (D^{j_1} b_1) (D^{j_2} b_2) \dotsc (D^{j_k} b_k) b_{k+1}
            ).
        \end{split}
    \end{equation}
    If $j_2+\cdots+j_{k+1} = \nu$ and $j_1 + j_{k+1}' = j_{k+1}$, then $\nu = j_{1} + \cdots + j_k + j_{k+1}'$, and:
    \begin{equation}
        (-1)^{\nu - j_{k+1}} \binom{\nu}{j_2, \dots, j_{k+1}} (-1)^{j_{1}} \binom{j_{k+1}}{j_1} = (-1)^{\nu-j_{k+1}'}\binom{\nu}{j_1,j_2, \dots, j_{k+1}'},
    \end{equation}
    which concludes the proof.
\end{proof}

\section{Computations and estimates on some binomial sums}

In this purely numerical section, we state and prove numerical formulas and estimates on the quantities $A^r_s(n)$ defined in \eqref{eq:def-arsn} which are involved in the derivation of our lower bounds for the growth of the structure constants.
They follow from elementary manipulations of sums using famous binomial identities.

\begin{lemma} \label{p:sum-r-binom}
	Let $1 \leq r \leq s$ and $n \geq 2s + 1$. Then
	\begin{equation}
		\sum_{p=0}^{r-1} (-1)^p \binom{n-p-1}{n-s-1} \binom{n-s}{p} 
		= (-1)^r \frac{r(r-n)}{s(n-s)} \binom{n-r-1}{n-s-1} \binom{n-s}{r}.
	\end{equation}
\end{lemma}

\begin{proof}
	We proceed by induction on $r$.
	For $r = 1$, using the absorption and symmetry identities, one checks that both sides are equal to $\binom{n-1}{s}$.
	We assume that the result holds for some $r \geq 1$ and we prove it for $r+1$.
	Let $s \geq r+1$ and $n \geq 2s+1$ (so in particular $s \geq r$ so that we can apply the induction equality).
	Hence
	\begin{equation}
		\begin{split}
			S_{r+1} := \sum_{p=0}^{(r+1)-1} (-1)^p \binom{n-p-1}{n-s-1} \binom{n-s}{p} 
		&	=
			(-1)^r \frac{r(r-n)}{s(n-s)} \binom{n-r-1}{n-s-1} \binom{n-s}{r} \\
		& \quad + (-1)^r \binom{n-r-1}{n-s-1} \binom{n-s}{r}.
		\end{split}
	\end{equation}
	Using the absorption identities twice yields
	\begin{equation}
		S_{r+1} = (-1)^r \left( \frac{r(r-n)}{s(n-s)} + 1 \right) \frac{(r+1)}{n-s-r} \frac{n-(r+1)}{s-r} \binom{n-(r+1)-1}{n-s-1} \binom{n-s}{r+1}.
	\end{equation}
	Simplifying the fraction yields indeed
	\begin{equation}
		S_{r+1} = (-1)^{r+1} \frac{(r+1)((r+1)-n)}{s(n-s)} \binom{n-(r+1)-1}{n-s-1} \binom{n-s}{r+1},
	\end{equation}
	which concludes the proof.
\end{proof}

\begin{lemma} \label{p:ars-final}
	Let $1 \leq r \leq s$ and $n \geq 2s+1$. Then
	\begin{equation} \label{eq:ars-final}
		A^r_s(n) = \frac{n-2s}{s} \binom{n-r}{n-s} \binom{n-s-1}{r-1}
	\end{equation}
\end{lemma}

\begin{proof}
	We start with the following elementary formula, valid for $0 \leq a \leq b$ (including the cases $0 = a < b$ and $a = b = 0$ with the convention that $\binom{\cdot}{-1} = 0$),
	\begin{equation} \label{eq:binom-diff}
		\binom{b}{a} - \binom{b}{a-1} = \frac{b+1-2a}{b+1} \binom{b+1}{a}
	\end{equation}
	Starting from the definition \eqref{eq:def-arsn} of $A^r_s(n)$ and using \eqref{eq:binom-diff} to simplify the difference and the symmetry identity for the first binomial, we obtain
	\begin{equation}
		A^r_s(n) = (n-2s) \left| 
		\sum_{p=r}^s \frac{(-1)^p}{n-2p} \binom{n-1-p}{n-1-2p} \binom{n-2p}{s-p}
		\right|.
	\end{equation}
	Then, using the absorption identity for the upper index of the second binomial,
	\begin{equation}
		A^r_s(n) = (n-2s) \left| 
		\sum_{p=r}^s \frac{(-1)^p}{n-p-s} \binom{n-1-p}{n-1-2p} \binom{n-1-2p}{s-p}
		\right|.
	\end{equation}
	Using the ``trinomial revision'' formula (see e.g.\ \cite[(5.21)]{zbMATH00718142}) yields
	\begin{equation}
		A^r_s(n) = (n-2s) \left| 
		\sum_{p=r}^s \frac{(-1)^p}{n-p-s} \binom{n-1-p}{s-p} \binom{n-1-s}{n-1-p-s}
		\right|.
	\end{equation}
	Using the absorption identity for the upper index of the second binomial and the symmetry identity twice yields
	\begin{equation}
		A^r_s(n) = \frac{n-2s}{n-s} \left| 
		\sum_{p=r}^s (-1)^p \binom{n-1-p}{n-1-s} \binom{n-s}{p}
		\right|.
	\end{equation}	
	Moreover, thanks to \cite[(5.25)]{zbMATH00718142},
	\begin{equation}
		\sum_{p=0}^{n-1} (-1)^p \binom{n-1-p}{n-1-s} \binom{n-s}{p}
		= (-1)^s \binom{2-n}{s} = 0.
	\end{equation}
	when $n \geq 2$.
	Since the summand vanishes for $p > s$, we have
	\begin{equation}
		A^r_s(n) = \frac{n-2s}{n-s} \left| 
		\sum_{p=0}^{r-1} (-1)^p \binom{n-1-p}{n-1-s} \binom{n-s}{p}
		\right|.
	\end{equation}
	Thus, using \cref{p:sum-r-binom} to compute the sum and the absorption identity twice to absorb the leading factors leads to the formula \eqref{eq:ars-final}, which concludes the proof.
\end{proof}

\begin{lemma} \label{p:arsn-fibo}
	Let $1 \leq r \leq s$ and $n \geq 2s+1$.
	Then
	\begin{equation} \label{leq:ars-fn}
		A^r_s(n) \geq \binom{n-s-1}{s}.
	\end{equation}
\end{lemma}

\begin{proof}
	Expanding the binomials in \eqref{eq:ars-final}, inequality \eqref{leq:ars-fn} is equivalent to
	\begin{equation}
		\begin{split}
			& \frac{n-2s}{s} \frac{(n-r)!}{(n-s)!(s-r)!} \frac{(n-s-1)!}{(r-1)!(n-s-r)!} \geq \frac{(n-s-1)!}{s!(n-2s-1)!} \\
			\Longleftrightarrow \quad & \frac{(s-1)! (n-2s)! (n-r)!}{(n-s)!(s-r)!(r-1)!(n-s-r)!} \geq 1 \\
			\Longleftrightarrow \quad & \binom{s-1}{r-1} \frac{(n-2s)!}{(n-s-r)!} \frac{(n-r)!}{(n-s)!} \geq 1 \\
			\Longleftrightarrow \quad & \binom{s-1}{r-1}
			\prod_{j=r}^{s-1} \frac{n-j}{n-s-j}
			\geq 1,
		\end{split}
	\end{equation}
	which is always true since the binomial is an integer and $n-j \geq n-s-j$ for each $j \in \intset{r, s-1}$.
\end{proof}

\begin{lemma} \label{p:arsn-2n}
	For every $r \geq 1$, there exists $C_r > 0$ such that, for every $n \geq 2r + 1$, 
	\begin{equation}
		A^r_{\lfloor \frac{n-1}{2} \rfloor}(n) \geq C_r n^{r-\frac 5 2} 2^n.
	\end{equation}
\end{lemma}

\begin{proof}
	Let $r \geq 1$. 
	For $n \geq 2r+1$, we use the notation $m := \lfloor \frac{n-1}{2} \rfloor$.
	Thus, $n = 2m+q$ with $q = 1$ or $q = 2$ and $m \geq r$.
	Substituting these values in \eqref{eq:ars-final} yields
	\begin{equation}
		A^r_m(2m+q) = \frac{q}{m} \binom{2m+q-r}{m+q} \binom{m+q-1}{r-1}.
	\end{equation}
	Since $r$ and $q$ are fixed, Stirling's formula provides the following asymptotic for large $m$ (or, equivalently, large $n$)
	\begin{equation}
		A^r_m(2m+q) \sim \frac{q}{m} \frac{(2m)^{2m+q-r} \sqrt{2\pi m}}{m^{m+q}m^{m-r} 2 \pi m} \frac{m^{m+q-1}\sqrt{2 \pi m}}{(r-1)! m^{m+q-r} \sqrt{2 \pi m}}
		= \frac{q m^{r-\frac{5}{2}} 2^{2m+q-r}}{(r-1)!\sqrt{2 \pi}}.
	\end{equation}
	Hence,
	\begin{equation}
		A^r_{\lfloor \frac{n-1}{2} \rfloor}(n)
		\sim \frac{4^{1-r} q}{(r-1)! \sqrt{\pi}} n^{r - \frac 5 2} 2^n,
	\end{equation}
	where $q$ alternates between $1$ and $2$ so is bounded below.
	Since $A^r_{\lfloor \frac{n-1}{2} \rfloor}(n) > 0$ for all values of $n$, choosing $C_r$ sufficiently small concludes the proof.
\end{proof}

\begin{lemma} \label{p:sum-a1sn}
    For every $n \geq 3$,
    \begin{equation}
        1 + \sum_{s=1}^{\lfloor \frac{n-1}{2} \rfloor} A^1_s(n) 
        = \binom{n-1}{\lfloor \frac{n-1}{2} \rfloor}
        \leq 2^{n-2}.
    \end{equation}
\end{lemma}

\begin{proof}
    Let $n \geq 3$ and $m := \lfloor \frac{n-1}{2} \rfloor \geq 1$.
    Using \eqref{eq:ars-final} and the absorption identity,
    \begin{equation}
        \begin{split}
            \alpha_n := 1 + \sum_{s=1}^{\lfloor \frac{n-1}{2} \rfloor} A^1_s(n) 
            & = 1 + \sum_{s=1}^m \frac{n-2s}{s} \binom{n-1}{n-s} \\
            & = 1 + \sum_{s=1}^m \frac{n}{s} \binom{n-1}{s-1} - 2 \sum_{s=1}^m \binom{n-1}{s-1} \\
            & = \sum_{s=0}^m \binom{n}{s}
            - 2 \sum_{s=0}^{m-1} \binom{n-1}{s}.
        \end{split}
    \end{equation}
    Using elementary half summation formulas of binomial coefficients yields, when $n = 2m+1$,
    \begin{equation}
        \alpha_{2m+1}
        = \sum_{s=0}^m \binom{2m+1}{s}
            - 2 \sum_{s=0}^{m-1} \binom{2m}{s}
        = \frac{1}{2} 2^{2m+1} - 2 \frac{1}{2} \left(2^{2m} - \binom{2m}{m}\right)
        = \binom{2m}{m},
    \end{equation}
    and, when $n = 2m+2$, thanks to the absorption identity,
    \begin{equation}
        \begin{split}
            \alpha_{2m+2} & = \sum_{s=0}^m \binom{2m+2}{s} - 2 \sum_{s=0}^{m-1} \binom{2m+1}{s} \\
            & = \frac{1}{2} \left( 2^{2m+2} - \binom{2m+2}{m+1} \right)  
            - 2 \left( \frac{1}{2} 2^{2m+1} - \binom{2m+1}{m} \right) \\
            & = 2 \binom{2m+1}{m} - \frac{1}{2} \binom{2m+2}{m+1} = \binom{2m+1}{m}.
        \end{split}
    \end{equation}
    Hence, in both cases, $\alpha_n = \binom{n-1}{m}$, which is the claimed equality.
    We now prove the upper bound.
    Using the bounds associated with Stirling's approximation (see e.g.\ \cite{zbMATH03115226}),
    \begin{equation}
        \alpha_n = \binom{n-1}{m} \leq
        \frac{e^{\frac{1}{12}} \sqrt{2\pi(n-1)}(n-1)^{n-1}} 
        {\sqrt{2\pi m} m^m \sqrt{2\pi (n-1-m)} (n-1-m)^{(n-1-m)}}.
    \end{equation}
    When $n = 2m+1$, this yields
    \begin{equation}
        \alpha_{2m+1} \leq \frac{e^{\frac{1}{12}}}{\sqrt{\pi m}} 2^{2m}.
    \end{equation}
    When $n = 2m+2$, this yields
    \begin{equation}
        \alpha_{2m+2} \leq \frac{e^{\frac{1}{12}}}{\sqrt{\pi m}} 2^{2m+1} \left( 1 + \frac 1 m \right)^m
        \leq \frac{e^{\frac{13}{12}}}{\sqrt{\pi m}} 2^{2m+1}.
    \end{equation}
    Hence, in both cases, $\alpha_n \leq \frac{1}{2} 2^{n-1} = 2^{n-2}$ for $m \geq 12$ (so $n \geq 25$).
    One concludes the proof by checking by hand or with computer algebra that $\alpha_n \leq 2^{n-2}$ for $n \in \intset{3,24}$.
\end{proof}

\section{Some estimates with Fibonacci numbers}

As in \cref{s:fibo}, we denote by $(F_\nu)_{\nu \in \N}$ the 0-based Fibonacci numbers.
We gather here elementary estimates involving these numbers which are used in \cref{s:fibo}.

\begin{lemma}
    For all $p, q \in \N^*$,
    \begin{align}
        \label{eq:fibo-2+2}
        2 F_p & \leq F_{p+2}, \\
        \label{eq:fibo-cross}
        F_{p} F_{q-1} + F_{p-1} F_{q} & \leq F_{p+q-1}.
    \end{align}
\end{lemma}

\begin{proof}
    Inequality \eqref{eq:fibo-2+2} is consequence of the monotonicity of the sequence: $F_{p+2}=F_{p+1}+F_p \geq 2 F_p$.
    Inequality \eqref{eq:fibo-cross} clearly holds if either $p$ or $q$ equals 1, and otherwise follows from the fact that:
    \begin{equation}
        0 \leq F_{p-2}F_{q-2} = (F_{p} - F_{p-1})(F_{q} - F_{q-1}) = F_{p}F_{q} + F_{p-1}F_{q-1} - (F_{p} F_{q-1} + F_{p-1} F_{q})
    \end{equation}
    and from the classical identity $F_{p}F_{q} + F_{p-1}F_{q-1} = F_{p+q-1}$  (see \cite[eq.\ (1)]{tagiuri1901}).
\end{proof}

\begin{lemma}
    The following estimates holds:
    \begin{align}
        \label{eq:n2n-light}
        & \forall n \geq 2, &
        (n-2) 2^{n-2} + n
        & \leq F_{2n-1},
        \\
        \label{eq:n2n-f2n2}
        & \forall n \geq 2, &
        (n-3) 2^{n-2} + n + 0
        & \leq F_{2n-2},
        \\
        \label{eq:n2n-9}
        & \forall n \geq 9, &
        (n-2) 2^{n-2} + n & \leq F_{2n-2}
    \end{align}
\end{lemma}

\begin{proof}
    These estimates hold for large values of $n$ because the geometric growth rate of the right-hand side is $\varphi^2 \approx 2.618 > 2$ and can be checked numerically for small values of $n$.
\end{proof}

\section{Other bases of free Lie algebras}
\label{s:other-bases}

Hall bases are of course not the only bases of $\mathcal{L}(X)$.
For example, one can construct bases whose elements are linear combinations of Lie monomials (see e.g.\ \cite{zbMATH00238486}).
For such ``polynomial bases'' one could study the growth of the structure constants with respect to $|a|+|b|$ where $|\cdot|$ would denote the degree of the Lie polynomial, i.e.\ the length of the longest Lie monomial involved.

Staying within the scope of ``monomial bases'' (i.e.\ bases of the form $\ev(\mathcal{A})$ for some $\mathcal{A} \subset \Br(X)$, not necessarily a Hall set), one can use other construction processes than the one yielding Hall sets, as illustrated by \cite{zbMATH05078139} or \cite{zbMATH05246413}. 
However, even if one uses an alternative construction to obtain a monomial basis of $\mathcal{L}(X)$, one could wonder if there exists a Hall set yielding, up to sign, the same basis.
We give below a short argument showing that this is not the case in general.
Hence, there indeed exist monomial bases which cannot be seen as  Hall bases.

\begin{proposition}
    Let $X$ be a set with $|X| \geq 2$.
    There exists $\mathcal{A} \subset \Br(X)$ such that $\ev(\mathcal{A})$ is a basis of $\mathcal{L}(X)$ but such that, for every Hall set $\B \subset \Br(X)$, $\ev(\mathcal{A}) \not\subset \pm \ev(\B)$.
\end{proposition}
    
\begin{proof}
    The proof consists in constructing an example of a finite subset $\mathcal{A'} \subset \Br(X)$ such that $\dim \vect \ev(\mathcal{A}') = |\mathcal{A}'|$ (so $\mathcal{A}'$ can be completed into an $\mathcal{A}$ such that $\ev(\mathcal{A})$ is a basis of $\mathcal{L}(X)$) and assume by contradiction that there exists a Hall set $\B \subset \Br(X)$ such that, for every $a \in \mathcal{A}'$, there exists $b \in \B$ such that $\ev(a) = \pm \ev(b)$.
    We start by an elementary remark requiring three letters, before proving the result when $|X| \geq 4$, and eventually extending the result to $|X| \in \{2,3\}$.
    
    \step{Step 1: 
    We prove that, if $[X_i, [X_j, X_k]] \in \pm \ev(\B)$, where $X_i, X_j, X_k$ are distinct elements of $X$ and~$\B$ is a Hall set, then $X_i \neq \min \{ X_i, X_j, X_k \}$, where the minimum is relative to the order in $\B$}
    Indeed, this implies that one of the following four cases occur.
    \begin{itemize}
        \item If $(X_i, (X_j, X_k)) \in \B$, then $\lambda((X_j, X_k)) \leq X_i$, so $X_j \leq X_i$, so $X_j < X_i$ (since $X_j \neq X_i$).
        \item If $(X_i, (X_k, X_j)) \in \B$, then $\lambda((X_k, X_j)) \leq X_i$, so $X_k \leq X_i$, so $X_k < X_i$ (since $X_k \neq X_i$).
        \item If $((X_j, X_k), X_i) \in \B$, then $X_j = \lambda((X_j, X_k)) < (X_j, X_k) < X_i$, so $X_j < X_i$.
        \item If $((X_k, X_j), X_i) \in \B$, then $X_k = \lambda((X_k, X_j)) < (X_k, X_j) < X_i$, so $X_k < X_i$.
    \end{itemize}
    
    \step{Step 2: We prove that the result holds when $|X| \geq 4$}
    Let $X_1, X_2, X_3, X_4 \in X$.
    Define
    \begin{equation}
        \mathcal{A}' := \{ 
        (X_1, (X_2, X_3)),
        (X_2, (X_1, X_3)),
        (X_3, (X_4, X_1)),
        (X_4, (X_3, X_2))
        \}.
    \end{equation}
    Assume that there exists a Hall set $\B \subset \Br(X)$ such that $\ev(\mathcal{A}') \subset \pm \ev(\B)$.
    Using the previous step and the first two brackets, we obtain that $X_1 \neq \min \{ X_1, X_2, X_3 \}$ and $X_2 \neq \min \{ X_1, X_2, X_3 \}$, so $X_3 < X_1$.
    By symmetry, using the last two brackets, $X_1 < X_3$, which is a contradiction.
    
    \step{Step 3: Proof when $|X| \in \{ 2, 3 \}$}
    Let $X_0, X_1 \in X$.
    For $i \in \N$ let $N_i := \ad_{X_0}^i(X_1)$ and $M_i := \ad_{X_1}^i(X_0)$.
    Define
    \begin{align}
        \mathcal{A}'_N & := \{ 
        (N_1, (N_2, N_3)),
        (N_2, (N_1, N_3)),
        (N_3, (N_4, N_1)),
        (N_4, (N_3, N_2))
        \}, \\
        \mathcal{A}'_M & := \{ 
        (M_1, (M_2, M_3)),
        (M_2, (M_1, M_3)),
        (M_3, (M_4, M_1)),
        (M_4, (M_3, M_2))
        \}
    \end{align}
    and $\mathcal{A}' := \mathcal{A}'_N \cup \mathcal{A}'_M$.
    Assume that there exists a Hall set $\B \subset \Br(X)$ such that $\ev(\mathcal{A}') \subset \pm \ev(\B)$.
    Assume that, in $\B$, $X_0 < X_1$ (the other case being symmetric using $M$ instead of $N$).
    Since $X_0 < X_1$, for each $i \in \N$, $N_i \in \B$ and the set $N := \{ N_i ; i \in \N \}$ is a free alphabetic subset of $\B$.
    By \cref{Prop:free-lie-alphabetic}, $\B_N := \e^{-1}(\B \cap \Br_N)$ is a Hall set of $\Br(N)$, and the canonical morphism of Lie algebras $\mathcal{L}(N) \to \mathcal{L}(X)$ induces an isomorphism onto the Lie subalgebra generated by $N$. Thus, $\B_N \subset \Br(N)$ is a Hall set such that $\ev(\mathcal{A}'_N) \subset \ev(\B_N)$ with $|N| \geq 4$. This contradicts Step 2.
\end{proof}    

\bibliographystyle{plain}
\bibliography{biblio}

\end{document}